\documentclass[reqno,centertags, 12pt]{amsart}
\usepackage{amsmath,amsthm,amscd,amssymb}
\usepackage{latexsym}
\usepackage{graphicx}
\usepackage[mathscr]{eucal}
\newcommand{\bbC}{{\mathbb{C}}}
\newcommand{\bbD}{{\mathbb{D}}}

\newcommand{\bbF}{{\mathbb{F}}}
\newcommand{\bbG}{{\mathbb{G}}}

\newcommand{\bbL}{{\mathbb{L}}}

\newcommand{\bbR}{{\mathbb{R}}}

\newcommand{\bbZ}{{\mathbb{Z}}}

\newcommand{\fre}{{\frak{e}}}
\newcommand{\frA}{{\frak{A}}}
\newcommand{\x}{{\mathbf{x}}}
\newcommand{\z}{{\mathbf{z}}}
\newcommand{\J}{{\mathscr{J}}}

\newcommand{\calB}{{\mathcal{B}}}

\newcommand{\calF}{{\mathcal F}}

\newcommand{\calJ}{{\mathcal J}}

\newcommand{\calL}{{\mathcal L}}
\newcommand{\calM}{{\mathcal M}}

\newcommand{\calR}{{\mathcal R}}
\newcommand{\calS}{{\mathcal S}}
\newcommand{\calT}{{\mathcal T}}

\newcommand{\bddot}{{\boldsymbol{\cdot}}}

\newcommand{\dott}{\,\cdot\,}

\newcommand{\lb}{\label}
\newcommand{\f}{\frac}

\newcommand{\ol}{\overline}
\newcommand{\ti}{\tilde  }
\newcommand{\wti}{\widetilde  }
\newcommand{\Oh}{O}

\newcommand{\Wr}{\text{\rm{Wr}}}

\newcommand{\dist}{\text{\rm{dist}}}
\newcommand{\Var}{\text{\rm{Var}}}

\newcommand{\ess}{\text{\rm{ess}}}
\newcommand{\ac}{\text{\rm{ac}}}

\newcommand{\s}{\text{\rm{s}}}

\newcommand{\supp}{\text{\rm{supp}}}
\newcommand{\intt}{\text{\rm{int}}}

\newcommand{\bi}{\bibitem}

\newcommand{\beq}{\begin{equation}}
\newcommand{\eeq}{\end{equation}}
\newcommand{\ba}{\begin{align}}
\newcommand{\ea}{\end{align}}
\newcommand{\veps}{\varepsilon}
\newcommand{\vy}{{\vec{y\!}}\,}
\DeclareMathOperator{\ca}{cap}

\let\det=\undefined\DeclareMathOperator{\det}{det}

\newcounter{smalllist}
\newenvironment{SL}{\begin{list}{{\rm\roman{smalllist})}}{%
\setlength{\topsep}{0mm}\setlength{\parsep}{0mm}\setlength{\itemsep}{0mm}%
\setlength{\labelwidth}{2em}\setlength{\leftmargin}{2em}\usecounter{smalllist}%
}}{\end{list}}

\newcommand{\comm}[1]{}

\DeclareMathOperator{\Real}{Re}
\DeclareMathOperator{\Ima}{Im}

\allowdisplaybreaks
\numberwithin{equation}{section}

\newtheorem{theorem}{Theorem}[section]

\newtheorem*{p2.1}{Proposition 2.1}
\newtheorem{proposition}[theorem]{Proposition}
\newtheorem{lemma}[theorem]{Lemma}
\newtheorem{corollary}[theorem]{Corollary}
\theoremstyle{definition}
\newtheorem*{remark}{Remark}
\newtheorem*{remarks}{Remarks}
\newtheorem*{definition}{Definition}

\newcommand{\abs}[1]{\lvert#1\rvert}
\newcommand{\jap}[1]{\langle #1 \rangle}
\newcommand{\norm}[1]{\lVert#1\rVert}
\newcommand{\ve}[1]{{\vec{#1\!}}\,}

\begin{document}

\title[Finite Gap Jacobi Matrices, I]{Finite Gap Jacobi Matrices,\\I. The Isospectral Torus}
\author[J.~S.~Christiansen, B.~Simon, and M.~Zinchenko]{Jacob S.~Christiansen$^*$, Barry Simon$^{*,\dagger}$, and
Maxim Zinchenko$^*$}

\thanks{$^*$ Mathematics 253-37, California Institute of Technology, Pasadena, CA 91125.
E-mail: stordal@caltech.edu; bsimon@caltech.edu; maxim@caltech.edu}
\thanks{$^\dagger$ Supported in part by NSF grant DMS-0652919}

\date{September 25, 2008}

\begin{abstract}
Let $\fre\subset\bbR$ be a finite union of disjoint closed
intervals. In the study of OPRL with measures whose essential
support is $\fre$, a fundamental role is played by the isospectral
torus. In this paper, we use a covering map formalism to define and
study this isospectral torus. Our goal is to make a coherent
presentation of properties and bounds for this special class as a
tool for ourselves and others to study perturbations. One important
result is the expression of Jost functions for the torus in terms of
theta functions.
\end{abstract}

\maketitle

\section{Introduction} \lb{s1}

Let $\fre\subset\bbR$ be a union of $\ell+1$ disjoint closed intervals
\begin{align}
& \fre= \fre_1 \cup \fre_2 \cup \cdots \cup \fre_{\ell+1} \lb{1.1} \\
& \fre_j = [\alpha_j, \beta_j] \lb{1.2} \\
& \alpha_1 < \beta_1 < \alpha_2 < \cdots < \alpha_{\ell+1} < \beta_{\ell+1} \lb{1.3}
\end{align}
$\ell$ counts the number of gaps.

For later purposes, we will need to exploit potential theoretic
objects associated to $\fre$. $\ca(\fre)$ will be its {logarithmic
capacity}, $d\rho_\fre$ the {equilibrium measure} (normalized by
$\rho_\fre(\bbR)=1$)
\begin{equation} \lb{1.4}
d\rho_\fre(x) = \rho_\fre(x)\, dx
\end{equation}
and $\rho_\fre (\fre_j)$ the {harmonic measures}. For reasons that
become clear soon, we say $\fre$ is {periodic} if all harmonic
measures, $\rho_\fre(\fre_j)$, $j=1, \dots, \ell+1$, are rational.
See \cite{Helms,Land,Ran,EqMC, StT,Tsu} for discussions of potential
theory.

We will be interested in one- and two-sided Jacobi matrices: one-sided with parameters labelled
$\{a_n,b_n\}_{n=1}^\infty$,
\begin{equation} \lb{1.5}
J=
\begin{pmatrix}
b_1 & a_1 & 0 & 0 & \cdots \\
a_1 & b_2 & a_2 & 0 & \cdots \\
0 & a_2 & b_3 & a_3 & \cdots \\
\vdots & \vdots & \vdots & \vdots & \ddots
\end{pmatrix}
\end{equation}
and two-sided with $\{a_n, b_n\}_{n=-\infty}^\infty$ extended to the top and left in the
obvious way. And, of course, we want to consider the orthogonal polynomials on the real line
(OPRL) \cite{FrBk,Rice,SzBk} defined by
\begin{equation} \lb{1.6}
\begin{gathered}
p_{-1}(x)=0 \qquad p_0(x) =1 \\
xp_n(x) =a_{n+1} p_{n+1}(x) + b_{n+1} p_n(x) + a_n p_{n-1}(x)
\end{gathered}
\end{equation}
If $d\mu$ is the spectral measure for $J$ and vector
$(1,0,0,\dots)^t$, then the $p_n$'s are orthonormal
\begin{equation} \lb{1.7}
\int p_n(x) p_m(x) \, d\mu(x) = \delta_{nm}
\end{equation}
We will also want to consider monic OPs, $P_n$, the multiple of $p_n$ with leading coefficient $1$,
\begin{align}
p_n(x) &= (a_1 \cdots a_n)^{-1} P_n(x) \lb{1.8} \\
xP_n(x) &= P_{n+1}(x) + b_{n+1} P_n(x) + a_n^2 P_{n-1}(x) \lb{1.9}
\end{align}

We want to analyze the case where
\begin{equation} \lb{1.10}
\sigma_\ess (J) \equiv \sigma_\ess(d\mu) = \fre
\end{equation}
Here $\sigma_\ess(J)$ is the essential spectrum of $J$, aka the
derived set of $\supp(d\mu)$. We will use $\sigma(J)$ (or
$\sigma(d\mu)$) for the spectrum of $J$ and
$\Sigma_\ac(d\mu)=\{x\mid\f{d\mu}{dx}\neq 0\}$ for the essential
support of the a.c.\ part of $d\mu$. In this paper, we will focus on
the isospectral torus, in \cite{CSZ2} on the Szeg\H{o} class, and in
\cite{CSZ3} on results that go beyond the Szeg\H{o} class. Some of
our results were announced in \cite{CSZann}.

The goal is to extend what is known about the case $\fre=[-1,1]$. This can be viewed as a problem
in approximation theory where polynomial asymptotics is critical or as a problem in spectral
theory where Jacobi parameter asymptotics is critical. As usual, there are three main levels from
the point of view of polynomial asymptotics:

\smallskip
\noindent{\bf (a) Root asymptotics.}  Asymptotics of
$\abs{P_n(x)}^{1/n}$. For $[-1,1]$, the theory is due to
Erd\"os--Tur\'{a}n \cite{ET} and Ullman \cite{Ull,Ull84,Ull85}. For
general sets, including finite gap sets, the theory is due to
Stahl--Totik \cite{StT} (see Simon \cite{EqMC} for a review). One
has for $x\notin\sigma(d\mu)$ and $d\mu$ regular (i.e.,
$\sigma_\ess(d\mu)=\fre$ and $\lim(a_1 \cdots a_n)^{1/n}=\ca(\fre)$)
that
\begin{equation} \lb{1.11x}
\abs{P_n(x)}^{1/n} \to \exp \biggl( \int\log\abs{x-y}\, d\rho_\fre(y)\biggr)
\end{equation}

\smallskip
\noindent{\bf (b) Ratio asymptotics.}  Traditionally, this involves
the ratio $P_{n+1}(x)/P_n(x)$ having a limit. Nevai \cite{Nev79}
showed that if $a_n\to a$, $b_n\to b$ ($a\neq 0$) so that
$\sigma_\ess (d\mu)=[b-2a, b+2a]$, then the limit exists for
$x\notin\sigma(d\mu)$. Simon \cite{S290} proved a converse: if the
limit exists at a single point in $\bbC_+=\{z\mid\Ima z >0\}$, then
for some $a,b$, we have that $a_n\to a$, $b_n\to b$. Thus, the
proper analog for $\sigma_\ess (d\mu)=\fre$ will not be existence of
a limit but something more subtle. This is an interesting open
question which we will not address.

\smallskip
\noindent{\bf (c) Szeg\H{o} asymptotics.}  This says that for
$z\notin\sigma(d\mu)$, $P_n(z)/D(z)E(z)^n \to 1$ for an explicit
function $E$ ($(z+\sqrt{z^2-1})$ for $\fre = [-1,1]$) and a function
$D$ which is $\mu$-dependent. The proper analog for general finite
gap sets was obtained by Widom \cite{Widom} (see also Aptekarev
\cite{Apt}) and by Peherstorfer--Yuditskii \cite{PY} using
variational methods. The ratio is only asymptotically (almost)
periodic. One of our main goals in this series is to provide a new
nonvariational approach to this result. In addition, following
Damanik--Simon \cite{Jost1} for $[-1,1]$, we want to consider cases
where the Szeg\H{o} condition fails.

\medskip
From the spectral theory point of view, the analogs of $a_n\to\f12$,
$b_n\to 0$ (aka the Nevai class) concern the isospectral torus, an
object we will discuss extensively in this paper. For now, we note
that if $\fre$ is periodic, the $J$'s in the isospectral torus are
all periodic Jacobi matrices with $\sigma_\ess(J)=\fre$. In the
general case, it is an $\ell$-dimensional torus of almost periodic
$J$'s with $\sigma_\ess (J)=\fre$. It can be singled out via minimal
Herglotz functions \cite{Rice} or reflectionless potentials
\cite{Remppt}; see Section \ref{s6} below.

The key realization is that the Nevai class needs to be replaced by approach to an isospectral
torus. This was first noted by Simon \cite{OPUC1,OPUC2} as conjectures in the context of  the
OPUC case. In turn, Simon was motivated by work of L\'opez  and collaborators \cite{BRLL,BHLL}
who studied the case of a single gap for OPUC.

From a spectral point of view, the analogs of the asymptotics results are:

\smallskip
\noindent{\bf (a)} \ Regularity implies more restrictions on the
Jacobi parameters than $(a_1 \cdots a_n)^{1/n} \to \ca(\fre)$. For
example, for $\fre=[-1,1]$, it is known that $\f{1}{n} \sum_{j=1}^n
(a_j-\f12)^2 + b_j^2\to 0$ and, for $\fre$ periodic, a similar
Ces\`aro convergence result for distances to the isospectral torus
is proven in \cite{S319}. The analog for general finite gap sets 
remains an interesting open question.

\smallskip
\noindent{\bf (b)} \ The key result here in the case $\fre=[-1,1]$
is the theorem of Denisov--Rakhmanov \cite{Denpams} stating that if
$\Sigma_\ac(d\mu)=\sigma_\ess(d\mu)=[-1,1]$, then $a_n\to \f12$,
$b_n\to 0$. Simon \cite{OPUC2} conjectured that for periodic $\fre$,
the proper result is that if
$\Sigma_\ac(d\mu)=\sigma_\ess(d\mu)=\fre$, then all right limits lie
in the isospectral torus. For periodic $\fre$, this was proven by
Damanik--Killip--Simon \cite{DKSppt} who conjectured the result for
general $\fre$. It was then proven for general finite gap sets by
Remling \cite{Remppt}. Remling's result plays a key role in our work
in paper~II \cite{CSZ2}. We note that in the opposite direction, Last--Simon
\cite{LS2006} have shown that if all right limits lie in the
isospectral torus, then $\sigma_\ess (d\mu)=\fre$.

\smallskip
\noindent{\bf (c)} \ Here there are two main results. When
$\sigma(d\mu)=\fre$ (no bound states), Widom proved that a Szeg\H{o}
condition implies
\begin{align}
\liminf\, \f{a_1 \cdots a_n}{\ca(\fre)^n} &> 0  \lb{1.11} \\
\limsup\, \f{a_1 \cdots a_n}{\ca(\fre)^n} &< \infty \lb{1.12}
\end{align}
The Szeg\H{o} condition in this situation is
\begin{equation} \lb{1.13}
\int_\fre \dist(x,\bbR\setminus\fre)^{-1/2} \log \biggl( \f{d\mu}{dx}\biggr)\, dx > -\infty
\end{equation}
Widom allowed no eigenvalues outside $\fre$. Peherstorfer--Yuditskii \cite{PY} had eigenvalues, but only
in a later note \cite{PYarx} did they have the natural (from their paper \cite{PYpams}) condition
\begin{equation} \lb{1.14}
\sum_j \dist (x_j, \fre)^{1/2} <\infty
\end{equation}
where $x_j$ are the point masses of $d\mu$ (or eigenvalues of $J$)
outside $\fre$. Thus, Peherstorfer--Yuditskii \cite{PYarx} showed
\begin{equation} \lb{1.15}
\eqref{1.13} + \eqref{1.14} \Rightarrow \eqref{1.11} + \eqref{1.12}
\end{equation}
One of our main results in paper~II \cite{CSZ2} is to show
\begin{equation} \lb{1.16}
\eqref{1.11} + \eqref{1.14} \Rightarrow \eqref{1.13} + \eqref{1.12}
\end{equation}
Peherstorfer remarked to us that, while this result is new, it can also be derived from the results of
\cite{PY}.

The key to our analysis is a machinery developed by Sodin--Yuditskii \cite{SY} and exploited by
Peherstorfer--Yuditskii \cite{PY,PYarx}. To explain it, we note that the key to recent sum rule
discussions (summarized in \cite{Rice}) is to take the $m$-function given by
\begin{equation} \lb{1.17}
m(z) =\int \f{d\mu(x)}{x-z}
\end{equation}
and in the case $\fre=[-2,2]$, move it to $\bbD=\{z\mid\abs{z}<1\}$ via
\begin{equation} \lb{1.18}
M(z)=-m(z+z^{-1})
\end{equation}
The map
\begin{equation} \lb{1.19}
\x(z)=z+z^{-1}
\end{equation}
is the unique analytic bijection of $\bbD$ to $\bbC\cup\{\infty\}\setminus [-2,2]$ with
\begin{equation} \lb{1.20}
\x(0)=\infty \qquad \lim_{\substack{z\to 0 \\ z\neq 0}} z \x(z) >0
\end{equation}
The minus sign in \eqref{1.18} comes from the fact that $\x$ maps
$\bbD\cap\bbC_+$ to $-\bbC_+$ (where $\bbC_+=\{z\mid\Ima z > 0\}$).

In our case, there cannot be an analytic bijection of $\bbD$ to
$\bbC\cup\{\infty\}\setminus\fre$ since
$\bbC\cup\{\infty\}\setminus\fre$ is not simply connected. However,
because the holomorphic universal cover of
$\bbC\cup\{\infty\}\setminus\fre$ is $\bbD$, there is an analytic
map $\x\colon\bbD\to \bbC\cup\{\infty\}\setminus\fre$ which is
locally one-one and obeys \eqref{1.20}. Moreover, there is a group
$\Gamma$ of M\"obius transformations of $\bbD$ to $\bbD$ so that
\begin{equation} \lb{1.21}
\x(z) = \x(w) \;\Leftrightarrow\; \exists \gamma\in\Gamma \text{ so
that } z=\gamma(w)
\end{equation}
This group is isomorphic to $\pi_1 (\bbC\cup\{\infty\}\setminus\fre)
= \bbF_\ell$, the free nonabelian group on $\ell$ generators. We
mention that $\x$ is uniquely determined if
\eqref{1.20}--\eqref{1.21} hold and $\x$ is locally one-one.

Our goal in this paper is to discuss the isospectral torus in terms
of this formalism. It turns out that basic objects for the
isospectral torus, like Bloch waves and Green's function behavior,
are not discussed in detail anywhere. We will remedy that here.
While these results will not be surprising to experts, they are
exceedingly useful both in our further works \cite{CSZ2,CSZ3} and in
\cite{BLS,FSW,HS2008,2ext}.

We should expand on the point we already remarked upon that there
are two distinct ways of describing the isospectral torus: as a set
of minimal Herglotz functions or as the family of reflectionless
Jacobi matrices with spectrum $\fre$. The view as minimal Herglotz
functions goes back to the earliest periodic KdV work
\cite{DubMatNov,McvM1} (see also \cite{FlMcL,Krich1,vMoer}), while
the reflectionless definition goes back at least to Sodin--Yuditskii
\cite{SY} (see also \cite{Remppt}).

There is an important distinction: reflectionless objects are
natural whole-line (doubly infinite) Jacobi matrices, while minimal
Herglotz functions are associated to half-line objects. Of course,
the passage from whole-line to half-line objects is by
restriction---but the converse is not so simple. From our point of
view, the key is that the $J$'s associated to minimal Herglotz
functions are quasiperiodic and such functions are determined by
their values on a half-line (because a quasiperiodic function
vanishing on a half-line is identically zero). Alternatively, if
$m(z)$ is a minimal Herglotz function, the demand that
\begin{equation} \lb{1.23a}
m_0(z) =\calM (a_0,b_0, m(z))
\end{equation}
where
\begin{equation} \lb{1.23b}
\calM(a,b,f(z)) = \f{1}{-z+b-a^2 f(z)}
\end{equation}
be a minimal Herglotz function determines $a_0$ and $b_0$, and so
inductively, minimality allows a unique continuation from the
half-line.

In Section~\ref{s2}, we describe the map $\x$ in \eqref{1.20} and its natural extension to a covering
(albeit not universal covering) map of the two-sheeted Riemann surface, $\calS$, that the $m$-function
for elements of the isospectral torus lives on. In Section~\ref{s3}, we describe a critical result of
Beardon \cite{Bear} on the Poincar\'e index of $\Gamma$. Section~\ref{s4} reviews the facts about
character automorphic Blaschke products and their connection to potential theory. We will also present
estimates on these products needed in later papers \cite{CSZ2,CSZ3}. In Section~\ref{s5}, we use this
machinery to prove Abel's theorem. In Section~\ref{s6}, we describe the isospectral torus as the family
of minimal Herglotz functions on $\calS$. Sections~\ref{s7} and \ref{s8} will describe the Jost functions
of elements of the isospectral torus and will prove that the natural map from the isospectral torus to the
group of characters of $\Gamma$, given by the character of the Jost function, is an isomorphism of tori.
We will also relate Jost functions to theta functions, one of the more significant results of the present
paper. Section~\ref{s9} will discuss Jost solutions and the associated Bloch waves. Finally, Section~\ref{s10}
will apply these solutions to the study of the Green's function. Some of the material in Sections~\ref{s2},
\ref{s4}, and \ref{s6} is in suitable texts but included here because we wish to make this paper more
accessible to approximation theorists who may be unfamiliar with it.

We also mention the enormous debt this paper owes to the seminal work of Sodin--Yuditskii \cite{SY} and
Peherstorfer--Yuditskii \cite{PY}. About the only real advantage of our presentation in this first
paper over ideas implicit in \cite{SY,PY} is that we are more explicit and our Jost functions,
unlike the close relatives in \cite{SY,PY}, are strictly character automorphic. We emphasize that \cite{SY,PY}
had as their focus the theory of certain infinite gap sets for which $\fre$ is typically a Cantor set of
positive Lebesgue measure. But they include finite gap sets and provide useful tools in that special case.
Our work makes use of some results special to this finite gap situation.

We note that while we discuss Jost functions and solutions here for the isospectral torus, in
\cite{CSZ2,CSZ3} we will present them for any $J$ in the Szeg\H{o} class. For us, they are the key
to understanding Szeg\H{o} asymptotics in this finite gap situation.

\medskip
We want to thank D.~Calegari, H.~Farkas, F.~Gesztesy, I.~Kra,
N.~Makarov, F.~Peherstorfer, and P.~Yuditskii.

\section{The Covering Map and the Fuchsian Group} \lb{s2}

In this section, we describe the basic objects and setup that we will use. We emphasize that these
constructs are not new here, and more than anything else, this section sets up notation and gives
a pedagogical introduction. The Riemann surface, $\calS$, was introduced for finite gap KdV in
\cite{DubMatNov,McvM1} and for finite gap Jacobi matrices in \cite{FlMcL,Krich1,vMoer}.
The Fuchsian group formalism is from \cite{SY}.

Let $\calS_+$ be the set $\bbC\cup\{\infty\}\setminus\fre$ viewed as a Riemann surface. First of all,
we want to view this as one sheet of the Riemann surface of the function
\begin{equation} \lb{2.1a}
w=(R(z))^{1/2}
\end{equation}
where
\begin{equation} \lb{2.1b}
R(z) = \prod_{j=1}^{\ell+1} (z-\alpha_j)(z-\beta_j)
\end{equation}
More explicitly, we consider pairs $(w,z)$ in $\bbC^2$ obeying
\begin{equation} \lb{2.2}
w^2 -R(z) \equiv G(w,z) =0
\end{equation}
Since $\f{\partial G}{\partial z} \neq 0$ at those $2\ell+2$ points where $\f{\partial G}{\partial w}
=0$, this set is a one-dimensional complex manifold, aka a Riemann surface.

With two points at infinity added, this set becomes a compact
surface $\calS$. One can formally define $\calS$ by looking in
$\bbC^3\setminus\{0\}$ at triples, $(w,z,u)$, with
\begin{equation} \lb{2.3}
w^2 u^{2\ell} =\prod_{j=1}^{\ell+1} (z-\alpha_j u)(z-\beta_j u)
\end{equation}
and regarding $(w,z,u)$ as equivalent to $(w',z',u')$ if there is $\lambda\in\bbC\setminus\{0\}$,
so $w=\lambda w'$, $z=\lambda z'$, $u=\lambda u'$. Rather than this formal projective space view,
we will think of two points $\infty_\pm\in\calS$, obtained by using $\zeta=1/z$ coordinates on $\calS_\pm$
and adding the missing point $\zeta=0$.

There is a natural map $\pi\colon\calS\to\bbC\cup\{\infty\}$ given by $(w,z)\to z$. It sets up
$\calS$ as a branched cover of $\bbC\cup\{\infty\}$. $\pi$ is two-one on all points in $\bbC\cup\{\infty\}$
except $\{\alpha_j,\beta_j\}_{j=1}^{\ell+1}$---these latter points are the branch points. There is a
second natural map $\tau\colon\calS\to\calS$ that in $(w,z)$ coordinates takes $w\to -w$. $\calS_-$
will denote the image of $\calS_+$ under $\tau$. $\tau(\infty_+)=\infty_-$. $\calS\setminus
(\calS_+\cup\calS_-)$ is thus $\pi^{-1}(\fre)$. Each $\pi^{-1} (\fre_j)$ is topologically a circle.

There is a close connection between $\calS$ and the potential theory associated to $\fre$. In terms of
the equilibrium measure, $d\rho_\fre$, consider its Borel transform,
\begin{equation} \lb{2.3a}
M_\fre(z) = \int \f{d\rho_\fre(x)}{x-z}
\end{equation}
It is a basic fact (due to Craig \cite{Craig}; see also \cite{EqMC,Rice}) that for suitable points,
$x_j\in (\beta_j,\alpha_{j+1})$, we have
\begin{equation} \lb{2.3b}
M_\fre(z) = \f{-\prod_{j=1}^\ell (z-x_j)}{\Bigl(\prod_{j=1}^{\ell+1}
(z-\alpha_j)(z-\beta_j)\Bigr)^{1/2}}
\end{equation}
so $M_\fre$ has a natural analytic continuation from $\bbC\cup\{\infty\}\setminus\fre$ to $\calS$.

Topologically, $\calS$ is the sphere with $\ell$ handles attached---the canonical surface of
genus $\ell$. Its first homology group (see \cite{Arm,Hat,Vas} for basic topological notions we use
here) is $\bbZ^{2\ell}$. One way of looking at the generators is picking curves that loop around
each $\pi^{-1} (\fre_j)$ but one (the sum of all $\ell+1$ is homologous to zero) and also curves that
loop around each
\begin{equation} \lb{2.5a}
\pi^{-1}([\beta_j, \alpha_{j+1}]) \equiv G_j \qquad j=1, \dots, \ell
\end{equation}

For a while, we put $\calS$ aside and focus on $\calS_+$. $\calS_+$ is not simply connected. Its
fundamental group is the free nonabelian group on $\ell$ generators. We will pick $\infty$ as the
base point. One way of picking generators is to pick $\ti\gamma_1, \dots, \ti\gamma_\ell$ where
$\ti\gamma_j$ is the curve that starts at $\infty$, traverses in $\bbC_-$ to $\f12(\beta_j
+ \alpha_{j+1})$ (in the $j$-th gap), and returns to $\infty$ in $\bbC_+$ (see the lower half
of Fig.~2 below).

The universal cover of $\calS_+$ inherits the local complex structure of $\calS_+$ and so is a Riemann
surface. The deck transformations preserve this complex structure so there is a discrete group,
$\Gamma$, of complex automorphisms of the universal cover where each $\gamma\in\Gamma$ has no fixed
points.

It is a fundamental result in the theory of Riemann surfaces (the uniformization theorem; see
\cite{FarKra,GriHar,Miran}) that the only simply connected Riemann surfaces are the Riemann sphere,
the complex plane, and the unit disk, $\bbD$. The sphere has no fixed point free complex
automorphism and the only discrete groups of automorphisms on $\bbC$ are one- and two-dimensional
lattices, so the only Riemann surfaces with cover $\bbC$ are the tori and the punctured disk. Since
$\calS_+$ is neither of these, its universal cover is $\bbD$.

Thus, there exists a map $\x\colon \bbD\to\calS_+$, which is locally
one-one, and a group, $\Gamma$, of M\"obius transformations on
$\bbD$ so that \eqref{1.21} holds. By requiring \eqref{1.20} (which,
by the action of maps on $\bbD$, we can always do), we uniquely fix
$\x$. There is a lovely proof of the existence of $\x$ due to
Rad\'{o} \cite{Rado} that follows the standard proof of the Riemann
mapping theorem (\cite{Ahl}); see \cite[Sect.~9.5]{Rice}.

$\Gamma$ is a discrete group of M\"obius transformations leaving $\bbD$ setwise fixed, aka a
Fuchsian group. For background on such groups, see \cite{Katok} or \cite[Ch.~9]{Rice}.

$\calS_+$ is invariant under complex conjugation, as is $\bbD$, so
$\ol{\x(\bar z)}$ is also a covering map of $\bbD$ over $\calS_+$.
But it obeys \eqref{1.20}, so by uniqueness,
\begin{equation} \lb{2.4}
\x(\bar z) = \ol{\x(z)}
\end{equation}

We define the {fundamental region}, $\calF^\intt\subset\bbD$, as
follows: $\x^{-1}(\bbC\cup\{\infty\} \setminus
[\alpha_1,\beta_{\ell+1}])$ consists of connected components on
which $\x$ is a bijection (this is because
$\bbC\cup\{\infty\}\setminus[\alpha_1,\beta_{\ell+1}]$ is simply
connected, and so contains no closed curve nonhomotopic to the
trivial curve in $\calS_+$). We let $\calF^\intt$ be the component
containing $0\in \x^{-1}(\{\infty\})$; we will shortly enlarge
$\calF^\intt$ to a fundamental set, $\calF$.

In $\calF^\intt$, consider $\x^{-1} (\bbR\cup\{\infty\}\setminus
[\alpha_1, \beta_{\ell+1}])$. By \eqref{2.4}, the set is a subset of
$\bbD\cap\bbR$. But as $y\to\alpha_1$ or $\beta_{\ell+1}$,
$\x^{-1}(y)$ must approach the boundary of $\bbD$. It follows that
$\x^{-1}(\bbR\cup\{\infty\}\setminus [\alpha_1,
\beta_{\ell+1}])=(-1,1) \subset\bbD$. The other inverse images of
this set are, by \eqref{1.21}, images of $(-1,1)$ under M\"obius
transformations, so arcs of orthocircles, that is, pieces of circles
orthogonal to $\partial\bbD$.

In place of \eqref{1.20}, we could have required $\x(0)=\f12
(\beta_j + \alpha_{j+1})$ together with $\x'(0) >0$ and seen that
for this $\x$, one has $(-1,1)$ in the inverse image of the gap
$(\beta_j, \alpha_{j+1})$. Since the two $\x$'s are related by a
M\"obius transformation, we conclude that under our $\x$ (normalized
by \eqref{1.20}) the inverse images of gaps are also arcs of
orthocircles. The boundary of $\calF^\intt$ in $\bbD$ (not
$\ol{\bbD}$) clearly has $2\ell$ pieces corresponding to the tops
and bottoms of the $\ell$ gaps. Thus, $\calF^\intt$ is $\bbD$ with
$2\ell$ orthocircles (and their interior) removed---$\ell$ in each
half-plane---these are conjugate to one another. We label the
boundary pieces in the upper half-disk $C_1^+, \dots, C_\ell^+$.
Figure~1 shows a typical $\calF^\intt$ for $\ell=2$ with the inverse
image of $\bbC_-\cap\calS_+$ shaded.
\begin{center}
\begin{figure}[h]
\includegraphics[scale=0.4]{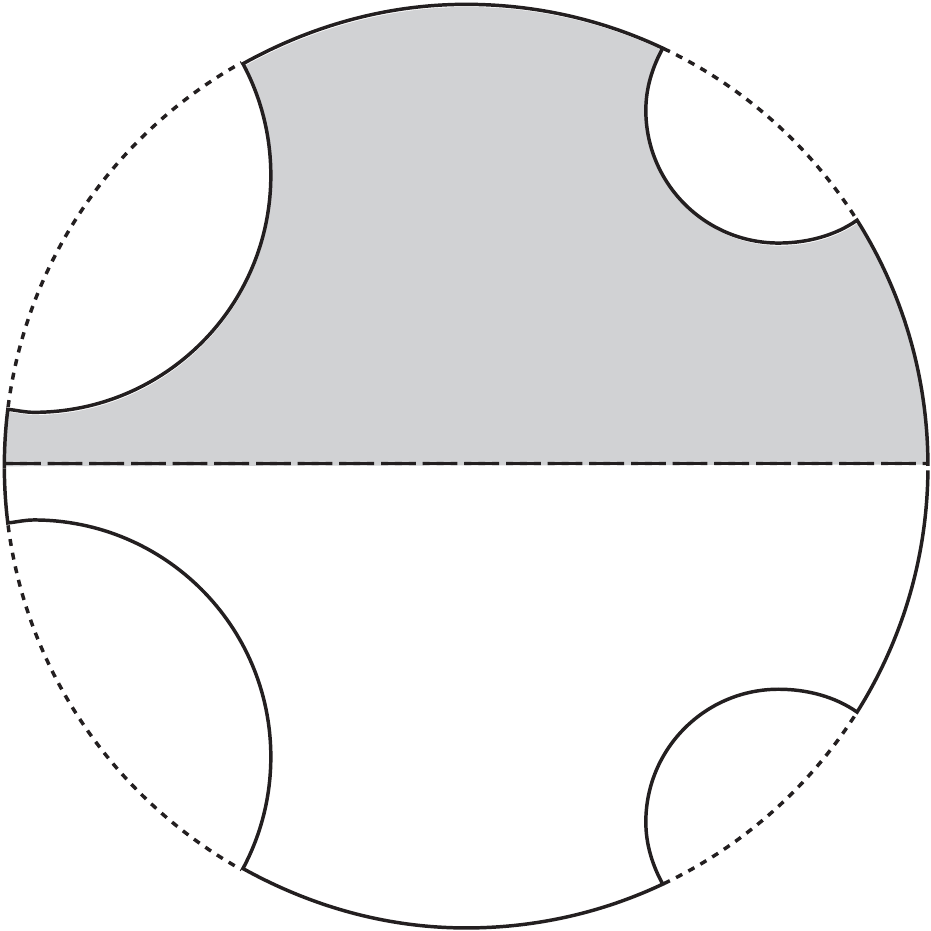}
\caption{The fundamental region}
\end{figure}
\end{center}

We now define a {fundamental set}, $\calF$, by adding the arcs
$C_1^+, \dots, C_\ell^+$ to $\calF^\intt$. With this definition,
every point in $\bbD$ can be uniquely written as $\gamma(w)$ for
some $w\in\calF$ and some $\gamma\in\Gamma$. The fundamental region
$\calF^\intt$ is indeed the interior of $\calF$. 
As a subset of $\bbD$, $\ol\calF$ has $C_1^-, \dots, C_\ell^-$
added. Here $C_j^-$ denotes the complex conjugate of $C_j^+$.
Sometimes we want to consider the closure of $\calF$ in $\ol{\bbD}$,
that is, also add the $2\ell$ arcs in $\partial\bbD$ at the ends.

To describe the Fuchsian group, $\Gamma$, we begin with the $\ell$
generators: the deck transformations that go into the generators,
$\ti\gamma_1, \dots, \ti\gamma_\ell$, of the homotopy group, $\pi_1
(\calS_+)$. Figure~2 shows the lift of the curve associated to
$\ti\gamma_2$ in our example. The bottom half of the curve in
$\calS_+$ under $\x^{-1}$ goes from $0$ to a point on $C_2^+$. Since
that half of $\ti\gamma_2$ lies in $\bbC_-$, this piece of curve
lies in $\bbC_+\cap\calF$. The second half must be the inversion in
the curve $C_2^+$ of the first half, and so it is as shown.
\begin{center}
\begin{figure}[h]
\includegraphics[scale=0.4]{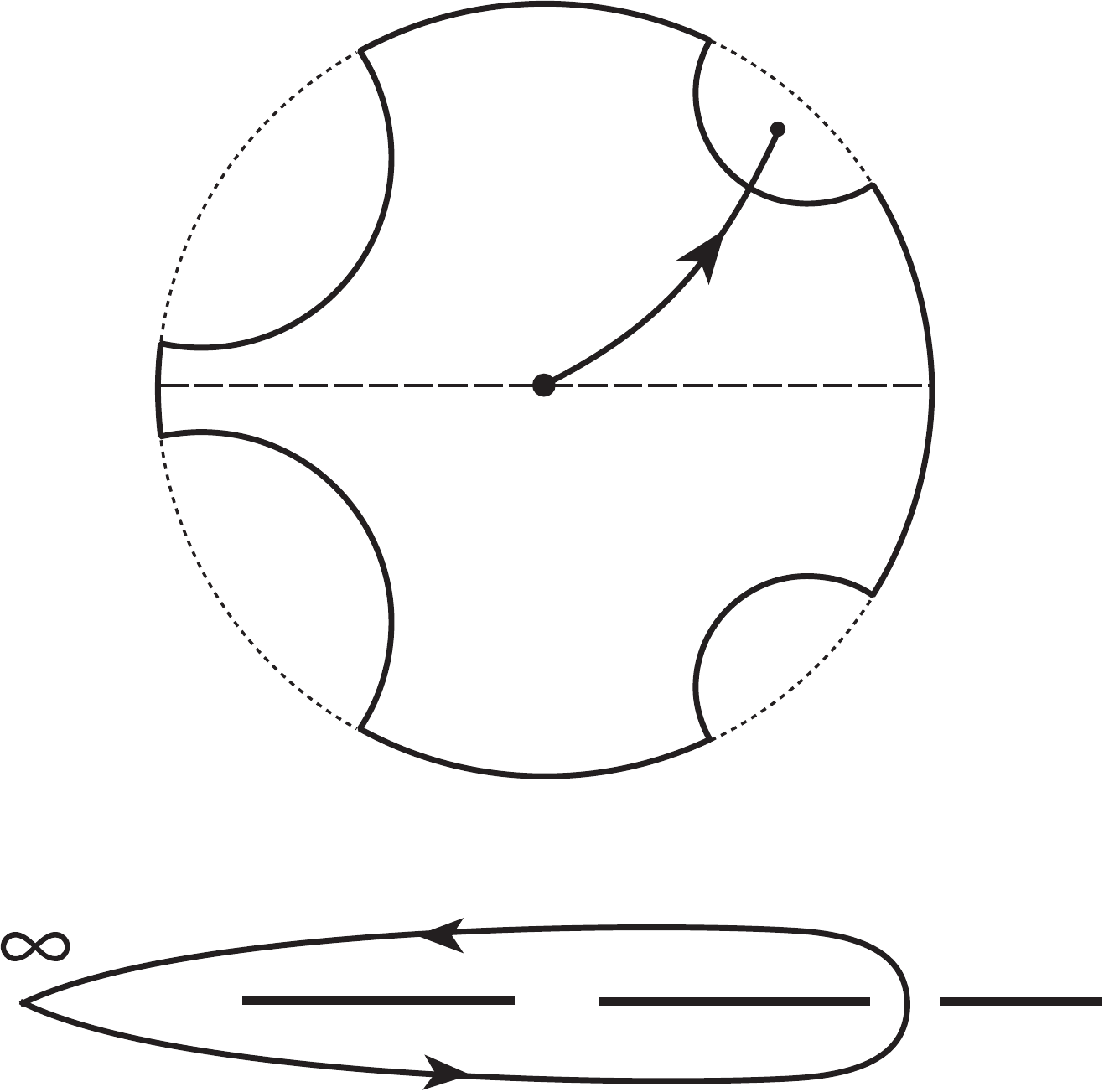}
\caption{Fuchsian group generators}
\end{figure}
\end{center}

A little thought shows that the M\"obius map that corresponds to
$\ti\gamma_2$, which we will call $\gamma_2$, is what one gets by
composing complex conjugation with inversion in $C_2^+$. Inversion
in the circle $\abs{z-z_0}=r$ is the map
\begin{equation} \lb{2.5}
z\to z_0 + \f{r^2}{\bar z-\bar z_0}
\end{equation}
Thus,
\[
\gamma_j = r_j^+ c
\]
where $c(z)=\bar z$ and $r_j^+$ is inversion in $C_j^+$.

$\Gamma$ is the free nonabelian group generated by
$\{\gamma_j\}_{j=1}^\ell$. Every element of $\Gamma$ can be uniquely
written as $\alpha_{w(\gamma)} \cdots \alpha_2\alpha_1$ where each
$\alpha_j$ is either a $\gamma_j$ or a $\gamma_j^{-1}$ and no
$\alpha_j$ is an $\alpha_{j-1}^{-1}$. $w(\gamma)$ is the word length
of $\gamma$. It will be convenient to define
\begin{equation} \lb{2.6}
\Gamma_k =\{\gamma\mid w(\gamma) =k\}
\end{equation}
We have $\#\Gamma_k = 2\ell(2\ell-1)^{k-1}$ since $\alpha_1$ has
$2\ell$ choices ($\gamma_1, \dots, \gamma_\ell$, $\gamma_1^{-1},
\dots, \gamma_\ell^{-1}$) and each other $\alpha_j$ has $2\ell-1$
choices. By definition, $\Gamma_0 = \{1\}$.

Alternatively, one can write for $\gamma\in\Gamma_{2m}$,
\begin{equation} \lb{2.6a}
\gamma = s_1 \cdots s_{2m}
\end{equation}
with each $s_k$ an $r_j^\pm$ ($r_j^-$ is inversion in $C_j^-$), and for $\gamma\in\Gamma_{2m+1}$,
\begin{equation} \lb{2.6b}
\gamma = s_1 \cdots s_{2m+1} c
\end{equation}

We point out that $\calF$ is the Dirichlet fundamental region for $\Gamma$, that is,
\begin{equation} \lb{2.6c}
\ol\calF = \{z\mid \abs{\gamma(z)}\geq\abs{z}\text{ for all }
\gamma\in\Gamma\}
\end{equation}
Moreover, $C_j^+$ is the perpendicular bisector in the hyperbolic
metric of $0$ and $\gamma_j(0)$ (see, e.g., \cite[Sect.~9.3]{Rice}).

Since $\gamma\in\Gamma$ has no fixed points in $\bbD$, it cannot be elliptic, and it is not hard to
see \cite{Rice} that it is, in fact, hyperbolic.

The fact that $\calF$ is a fundamental set implies that
\begin{equation} \lb{2.7}
\bbD = \bigcup_{\gamma\in\Gamma} \gamma (\calF)
\end{equation}
We will let
\begin{equation} \lb{2.8}
\bbD_k = \bigcup_{w(\gamma) \leq k} \gamma (\calF)
\end{equation}
and
\begin{equation} \lb{2.9}
\calR_k = \bbD\setminus\bbD_k
\end{equation}
and finally define $\partial\calR_k\subset\partial\bbD$ as
\begin{equation} \lb{2.10}
\partial\calR_k = \ol\calR_k \cap \partial\bbD
\end{equation}
where the closure is taken in $\ol{\bbD}$. Thus, $\bbD_k$ is
connected, while $\calR_k$ consists of $2\ell(2\ell-1)^k$ disks
(intersected with $\bbD)$ with only some boundaries included and
$\partial\calR_k$ is $2\ell(2\ell -1)^k$ connected arcs in
$\partial\bbD$. Figure~3 shows the arcs $C_j^\pm$ and their images
under $\Gamma_1$ and $\Gamma_2$.
\begin{center}
\begin{figure}[h]
\includegraphics[scale=0.4]{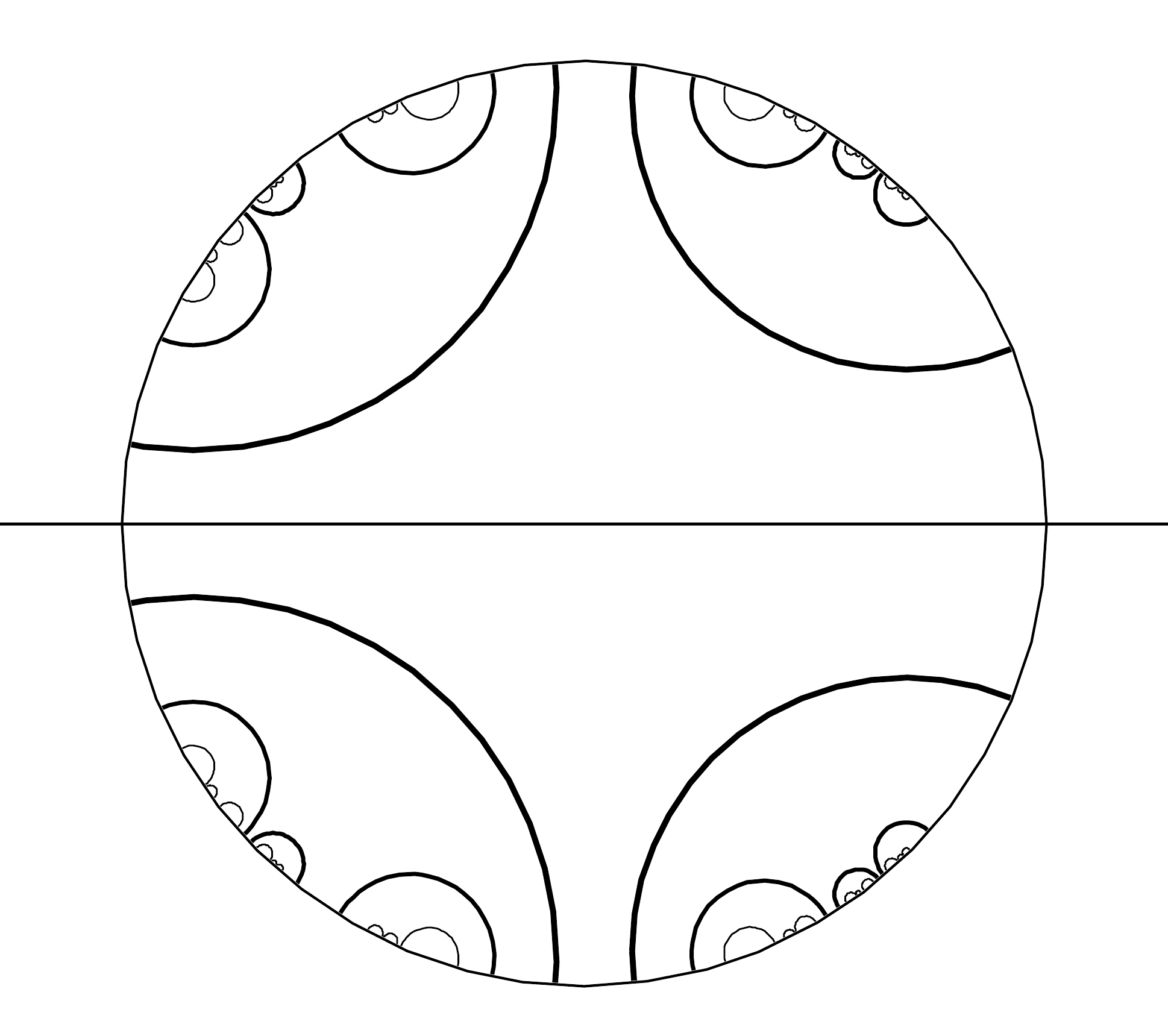}
\caption{Iterated generators}
\end{figure}
\end{center}
In Figure~3, the union of the large partial disks is $\calR_0$, the
union of the medium partial disks, $\calR_1$, and the union of the
tiny partial disks, $\calR_2$.

Notice the geometry is such that
\begin{equation} \lb{2.11}
z\in\calR_k \;\Rightarrow\; \f{z}{\abs{z}} \in \partial\calR_k
\end{equation}
which we will need in Section~\ref{s7} and \cite{CSZ2}.

We denote by $\calL$ the set of limit points of $\Gamma$. It is a subset of $\partial\bbD$
and can be defined via several equivalent definitions:
\begin{SL}
\item[(i)] $\calL = \cap_k \ol\calR_k$
\item[(ii)] $\calL = \ol{\{z\in\partial\bbD \mid\gamma(z)=z \text{ for some } \gamma\neq 1;
\, \gamma\in\Gamma\}}$
\item[(iii)] $\calL = \ol{\{\gamma(0)\mid\gamma\in\Gamma\}} \cap\partial\bbD$.
\end{SL}

In (ii), each $\gamma\in\Gamma$ is hyperbolic, so it has two fixed
points on $\partial\bbD$, each of which is either
$\lim_{n\to\infty}\gamma^n(0)$ or $\lim_{n\to\infty}\gamma^{-n}(0)$.
This is the key to proving that (ii) and (iii) are equivalent and
that (iii) is the same if $\gamma(0)$ is replaced by $\gamma (z_0)$
for any fixed $z_0\in\bbD$. The key to understanding (i) is that, by
definition, $\bbD_k$ contains only finitely many $\gamma(0)$, all of
which are a finite distance from $\partial\bbD$. As we will explain
in the next section, $\calL$ is of one-dimensional Lebesgue measure
zero, indeed, of Hausdorff dimension strictly smaller than one.

Now, we return to $\calS_+$ and the two-sheeted Riemann surface
$\calS$. The basic fact here is that the map $\x\colon
\bbD\to\calS_+$ has an analytic continuation both to a map of
$\bbC\cup\{\infty\}\setminus\calL$ to $\bbC\cup\{\infty\}$ and to a
map
\begin{equation}
\x^\sharp: \bbC\cup\{\infty\}\setminus\calL\to\calS.
\end{equation}
By construction, $\x(z)$ approaches $\bbR$ as $z\to\partial\bbD$
with $z\in\calF$. So, by the strong form of the reflection
principle, $\x$ is continuous and real-valued up to
$\partial\bbD\cap\ol\calF$ and can be meromorphically continued to
$(\partial\bbD\cap\ol\calF)\cup\calF^{-1}$. Utilizing \eqref{1.21},
we can thus extend $\x$ to a map of
$\bbC\cup\{\infty\}\setminus\calL$ onto $\bbC\cup\{\infty\}$. $\x$
outside $\ol{\bbD}$ is defined by
\begin{equation} \lb{2.16a}
\x(1/z) = \x(z)
\end{equation}

At points, $z_0$, where $\x(z_0)$ is real and $\x(z)-\x(z_0)$ has a
zero of order $k$, there are $2k$ curves (asymptotically rays)
coming out of $z_0$ on which $\x(z)$ is real. On $\bbC_+\cap\calF$,
$\x$ has negative imaginary part and so, by reflection, at points,
$z_0$, in $\bbC_+\cap\partial\bbD\cap\ol\calF$---except for the
endpoints, there are two rays near $z_0$ where $\x$ is real. It
follows that on the set
\begin{equation}
\bigl\{z\in\bbC\cup\{\infty\}\setminus \calL \mid \x(z)\notin
\{\alpha_j,\beta_j\}_{j=1}^{\ell+1}\bigr\}
\end{equation}
$\x'$ is nonzero. At points, $z_0$, where $\x(z_0)\in
\{\alpha_j,\beta_j\}_{j=1}^{\ell+1}$, images of $\bbR$ or a $C_j^+$
under an element $\gamma\in\Gamma$ intersect $\partial\bbD$
orthogonally, so four rays on which $\x$ is real come out of $z_0$.
Hence, $\x(z)-\x(z_0)$ has a double zero, that is, $\x'(z_0)=0$ and
$\x''(z_0)\neq 0$, and the extended map $\x$ is therefore not a
local bijection at such points $z_0$. The same is true for the
canonical projection $\pi\colon\calS\to\bbC\cup\{\infty\}$ so if we
think of $\x\colon \bbD\to\calS_+\subset\calS$ (rather than into a
subset of $\bbC\cup\{\infty\})$, we can extend it to a map
$\x^\sharp\colon\bbC\cup\{\infty\} \setminus\calL\to\calS$ via
\begin{equation} \lb{2.12}
\x^\sharp(1/z) = \tau(\x(z))
\end{equation}
Then the maps $\x$ and $\x^\sharp$ are related via
\begin{equation}
\x=\pi\circ\x^\sharp
\end{equation}

The elements of $\Gamma$ are rational functions and so maps of $\bbC\cup\{\infty\}$ to itself. It is easy
to see that each $\gamma$ maps $\calL$ to $\calL$ (for if $\gamma_n(0)\to z_0$, then $\gamma\circ\gamma_n
(0)\to\gamma(z_0)$) and so $\bbC\cup\{\infty\}\setminus\calL$ to itself. Of course, we have
\begin{equation} \lb{2.13}
\x(\gamma(z)) = \x(z)
\end{equation}
since $\gamma$ is analytic and this holds on $\bbD$. $\x^\sharp$ has
a similar relation. Indeed, since \eqref{2.12} holds, we have
\eqref{1.21} for $\x^\sharp$ on all of $\calS$. By the unfolding
discussed above, $\x^\sharp$ is a local bijection on all of
$\bbC\cup\{\infty\}\setminus\calL$, that is, a covering map, albeit
not the universal cover.

A major role will be played by automorphic and character automorphic functions. These are functions, $f$,
defined on $\bbD$ (usually analytic but occasionally meromorphic and occasionally only harmonic) or on
$\bbC\cup\{\infty\}\setminus \calL$ (always meromorphic) which obey
\begin{equation} \lb{2.14}
f(\gamma(z)) = c(\gamma) f(z)
\end{equation}
for all $\gamma$ and $z$. Here $c\neq 0$ and must obey
\begin{equation}
c(\gamma\gamma') = c(\gamma) c(\gamma')
\end{equation}
If $c\equiv 1$, $f$ is called {\it automorphic}. If $\abs{c}=1$ so
that $c$ is a unitary group character, we call $f$ {\it character
automorphic}. A character is determined by
$\{c(\gamma_j)\}_{j=1}^\ell$ which can be chosen independently so
the set of all characters is an $\ell$-dimensional torus,
$\Gamma^*$. This set has a group structure if $c\ti c(\gamma)
=c(\gamma) \ti c(\gamma)$ is the product and $c(\gamma)=1$ the
identity. Moreover, $c^{-1}(\gamma) = \ol{c(\gamma)}$.

Notice that $\x$ is automorphic on $\bbD$ and $\x^\sharp$ is automorphic on $\bbC\cup\{\infty\}\setminus
\calL$ if we extend the notion to include $\calS$-valued functions. Moreover, $f$ is automorphic and analytic
(resp.\ meromorphic) on $\bbD$ if and only if there is a function, $F$, on $\calS_+$ which
is analytic (resp.\ meromorphic) with
\begin{equation} \lb{2.15}
F(\x(z)) = f(z)
\end{equation}

Similarly, \eqref{2.15} with $\x$ replaced by $\x^\sharp$ sets up a one-one correspondence between meromorphic
functions on $\calS$ and meromorphic automorphic functions on $\bbC\cup\{\infty\}\setminus\calL$.

In particular, we have that the analog of \eqref{1.18},
\begin{equation} \lb{2.22}
M(z) = -m(\x(z))
\end{equation}
is an automorphic function with $\Ima M(z)>0$ for $z\in\calF^\intt\cap\bbC_+$.

We will require the following result:

\begin{theorem} \lb{T2.1} Fix $\ell$ and let $Q_\ell \subset\bbR^{2\ell+2}$ be the set of
$(\alpha_1, \beta_1, \dots, \alpha_{\ell+1},\beta_{\ell+1})$ for
which \eqref{1.3} holds. Given $q\in Q_\ell$, let $\x_q$ be the
covering map and $\gamma_{j;q}$ the Fuchsian group generators. Then
$\x_q$ and $\gamma_{j;q}$ are continuous in $q$ on $Q_\ell$.
\end{theorem}

This is a special case of a theorem of Hejhal \cite{Hej} who noted
that one can also base a proof using ideas from Ahlfors--Bers
\cite{AB}. We have found a proof for the case at hand and given it
in \cite[Sect.~9.8]{Rice}. We note that the Blaschke product,
$B(z)$, that we discuss in Section~\ref{s4} below is also continuous
in $q$.

We will also need the following well-known fact about functions on $\calS$ (see \cite{FarKra,GriHar} or
\cite[Thm.~5.12.5]{Rice}):

\begin{theorem}\lb{T2.2} Let $F$ be a nonconstant meromorphic function on $\calS$. Then $F$ has a degree, $d$,
so that for all $a$, $\{w\mid F(w)=a\}$ has $d$ points, counting
multiplicity {\rm{(}}i.e., $d$ is the sum of the orders of zeros of
$F(w)-a$ in local coordinates{\rm{)}}. If $F\circ\tau\not\equiv F$,
then the degree of $F$ is at least $\ell+1$.
\end{theorem}

\begin{remark}
In particular, if $F$ is analytic on $\calS$ it must be constant.
\end{remark}

\section{Beardon's Theorem} \lb{s3}

From our point of view, a theorem of Beardon \cite{Bear} plays a
critical role. To state the theorem, we need some notions. Fix $s>0$
and a Fuchsian group, $\Gamma$. The \emph{Poincar\'e series} is
given by
\begin{equation} \lb{3.1}
\sum_{\gamma\in\Gamma}\, \abs{\gamma'(0)}^s
\end{equation}
We are interested in when this series is convergent. It is a basic fact (see \cite{Katok,Rice}) that if
the series in \eqref{3.1} converges, then uniformly for $z$ in compacts of $\bbD$, the series
\begin{equation} \lb{3.2}
\sum_{\gamma\in\Gamma}\, \abs{\gamma'(z)}^s
\end{equation}
converges, as does uniformly on compacts of $\bbD$, the series
\begin{equation} \lb{3.3}
\sum_{\gamma\in\Gamma} (1-\abs{\gamma(z)})^s
\end{equation}
Indeed, convergence of \eqref{3.2} for one $z$ implies convergence
uniformly on compacts. What is also true (see, e.g.,
\cite[Sect.~9.4]{Rice}) is that if $K\subset\ol{\bbD}\setminus\calL$
is compact, then there is $C>0$ so that for all $z\in K$ and all
$\gamma\in\Gamma$,
\begin{equation} \lb{3.4}
\abs{\gamma'(z)} \leq C\abs{\gamma'(0)}
\end{equation}
so convergence of \eqref{3.1} implies convergence of \eqref{3.2}
uniformly for $z\in K$. In particular, since
$\ol{\bbD}\setminus\calL$ is connected, we see that the series
\begin{equation} \lb{3.5}
\sum_{\gamma\in\Gamma}\, \abs{\gamma(z)-\gamma(w)}^s
\end{equation}
converges uniformly for $z,w$ in compacts
$K\subset\ol{\bbD}\setminus\calL$.

Poincar\'e \cite{Poin1882} proved that for any Fuchsian group, \eqref{3.1} converges if $s=2$ and
Burnside \cite{Burn1,Burn2} proved that if the set of limit points is not all of $\partial\bbD$, then
\eqref{3.1} converges if $s=1$. Beardon proved

\begin{theorem}[Beardon \cite{Bear}]\lb{T3.1} If \,$\Gamma$ is a finitely generated Fuchsian group whose
limit points are not all of $\partial\bbD$, then there is some $s <1$ so that \eqref{3.1} converges.
\end{theorem}

As Beardon noted, this is equivalent to the set of limit points
having Hausdorff dimension less than $1$. Indeed, it is known (work
later than Beardon, see \cite{Patterson, Sullivan}) that the infimum
over $s$ for which \eqref{3.1} converges is the Hausdorff dimension
of $\calL$. We note that Beardon's proof is very involved, in part
because of the need to consider issues such as elliptic and
parabolic elements that are irrelevant to our setup. The result for
our case is proven using some simple geometry in Simon \cite{Rice}.

There is an important consequence of Beardon's theorem that we need. Let
\begin{equation} \lb{3.5a}
\ti\calR=\partial\bbD\setminus \partial\calR_0
\end{equation}
that is, $\ol\calF\cap\partial\bbD$. This set consists of $2\ell$
arcs. For each $\gamma\in\Gamma$, $\gamma (\ti\calR)$ is also
$2\ell$ arcs, so
\begin{equation} \lb{3.10}
\partial\calR_k = \calL\cup\Bigl[\bigcup_{w(\gamma) > k} \gamma (\ti\calR)\Bigr]
\end{equation}

It is not hard to see that on $\ti\calR$ and its images, each $\gamma_j^\pm$, but one,
decreases sizes by a fixed amount so that ($\abs{\dott}$ is total arc length)
\begin{equation} \lb{3.11}
\abs{\gamma(\ti\calR)} \leq Ce^{-Dw(\gamma)}
\end{equation}
for some fixed constants $C, D>0$ (proven in
\cite[Sect.~9.6]{Rice}).

By \eqref{3.4}, for some constant $Q$,
\begin{equation} \lb{3.12}
\abs{\gamma(\ti\calR)} \leq Q\abs{\gamma'(0)}
\end{equation}
Hence,
\begin{equation} \lb{3.13}
\abs{\gamma(\ti\calR)} \leq \abs{Ce^{-Dw(\gamma)}}^{1-s}
\abs{Q\abs{\gamma'(0)}}^s
\end{equation}
So, by Beardon's theorem for some $s<1$,
\begin{equation} \lb{3.14}
\abs{\partial\calR_k} \leq C^{1-s} Q^s e^{-D(1-s)k} \sum_{\gamma\in\Gamma}\, \abs{\gamma'(0)}^s
\end{equation}
and thus,

\begin{theorem}\lb{T3.2} For some constants $C_0, D_0 >0$, we have
\begin{equation} \lb{3.15}
\abs{\partial\calR_k} \leq C_0 e^{-D_0 k}
\end{equation}
\end{theorem}

\section{Blaschke Products and (Potential Theorist's) Green's Function} \lb{s4}

The initial elements of this section are classical; see, for example, Tsuji \cite{Tsu}. Given $w\in\bbD$,
we define $b(z,w)$ by
\begin{equation} \lb{4.1}
b(z,w) = \begin{cases}
\f{\abs{w}}{w}\, \f{w-z}{1-\bar wz} & w\neq 0 \\
\quad\; z & w=0
\end{cases}
\end{equation}
which is meromorphic in $z$ on $\bbC\cup\{\infty\}$, analytic in $z$ on $\bbD$, and is the unique bijective
map of $\bbD\to\bbD$ with
\begin{equation} \lb{4.2}
b(w,w) =0
\end{equation}
and
\begin{equation} \lb{4.3}
b(0,w)>0 \quad (w\neq 0); \qquad b'(0,w)>0 \quad (w=0)
\end{equation}
Note that $b$ is continuous in $z$ on $\ol{\bbD}$ and
\begin{equation} \lb{4.3a}
\abs{b(e^{i\theta},w)} =1
\end{equation}

The following is standard (see Rudin \cite{Rudin}):

\begin{lemma}\lb{L4.1} Let $(w_j)_{j=1}^\infty$ be a sequence of points in $\bbD$. Then
\begin{SL}
\item[{\rm{(a)}}] If
\begin{equation} \lb{4.3b}
\sum_{j=1}^\infty\, (1-\abs{w_j})=\infty
\end{equation}
then as $N\to\infty$,
\begin{equation} \lb{4.4}
\prod_{j=1}^N b(z,w_j) \to 0
\end{equation}
uniformly on compact subsets of $\bbD$.

\item[{\rm{(b)}}] If
\begin{equation} \lb{4.5}
\sum_{j=1}^\infty\, (1-\abs{w_j}) <\infty
\end{equation}
then as $N\to\infty$,
\begin{equation} \lb{4.6}
\prod_{j=1}^N b(z,w_j) \to  B(z,(w_j))
\end{equation}
uniformly on compact subsets of $\bbD$, where $B$ is analytic in
$\bbD$ and obeys
\begin{equation} \lb{4.7}
B(z, (w_j)) = 0 \;\Leftrightarrow\; z\in (w_j)
\end{equation}
Moreover, for Lebesgue a.e.\ $\theta$,
\begin{equation} \lb{4.8}
\lim_{r\uparrow 1}\, B(re^{i\theta}, (w_j)) \equiv B(e^{i\theta}, (w_j))
\end{equation}
exists obeying
\begin{equation} \lb{4.9}
\abs{B(e^{i\theta}, (w_j))} =1
\end{equation}
\end{SL}
\end{lemma}

\begin{remarks} 1. The refined form of \eqref{4.7} says that the order of the zero at some $w_j$ is the
number of times it occurs in $(w_j)$.

\smallskip
2. The proof shows that when \eqref{4.5} holds, uniformly for $\abs{z}\leq \rho <1$, we have
\begin{equation} \lb{4.9a}
\sum_{j=1}^\infty\, \abs{1-b(z,w_j)} <\infty
\end{equation}

\smallskip
3. The proof of \eqref{4.9a} follows from the simple inequality
\begin{equation} \lb{4.9b}
\abs{1-b(z,w)} \leq \f{1+\abs{z}}{\abs{1-\bar wz}}\, (1-\abs{w})
\end{equation}
which also proves that the product converges on $\bbC\setminus
\bigl[\,\ol\bbD \cup (1/\bar w_j)\bigr]$ and on any set
$K\subset\partial\bbD$ with
\begin{equation} \lb{4.9c}
\inf_{e^{i\theta}\in K,w_j}\, \abs{e^{i\theta}-w_j} >0
\end{equation}
In particular, if there is any set $K\subset\partial\bbD$ for which
\eqref{4.9c} holds, we can find such an open set and so get a
product analytic across $K$. This product is meromorphic, with poles
at the points $1/\bar w_j$.
\end{remarks}

We also need the following:

\begin{lemma}\lb{L4.2} Suppose $\gamma$ is an analytic bijection of $\bbD$ to $\bbD$. For any $z,w\in\bbD$, we have
\begin{alignat}{2}
&\text{\rm{(i)}} \qquad && \abs{b(z,w)}=\abs{b(w,z)} \lb{4.10} \\
&\text{\rm{(ii)}} \qquad && \abs{b(\gamma(z), \gamma(w))} = \abs{b(z,w)} \lb{4.11}
\end{alignat}
\end{lemma}

\begin{proof} (i) is immediate. For \eqref{4.11}, fix $w$ and let
\begin{equation} \lb{4.12}
h(z) = \f{b(\gamma(z),\gamma(w))}{b(z,w)}
\end{equation}
It is easy to see that $h$ has a removable singularity at $z=w$ and
so, it is analytic in $\bbD$ and continuous on $\ol{\bbD}$. By
\eqref{4.3a}, $\abs{h(e^{i\theta})}=1$ so, by the maximum principle,
$\abs{h(z)}\leq 1$ on $\bbD$. But $1/h$ has the same properties as
$h$, so $\abs{1/h(z)}\leq 1$, which implies that $\abs{h(z)}=1$,
that is, \eqref{4.11} holds.
\end{proof}

The following is true for any Fuchsian group whose limit points are not dense in $\partial\bbD$---but we only
care here about the $\Gamma$'s associated to finite gap sets:

\begin{theorem}\lb{T4.3} Let $\Gamma$ be the Fuchsian group of a finite gap set. For any $w\in\bbD$, the product
\begin{equation} \lb{4.13}
\prod_{\gamma\in\Gamma} b(z,\gamma(w)) \equiv B(z,w)
\end{equation}
converges for all
$z\in\bbC\cup\{\infty\}\setminus\bigl[\calL\cup\{\ol{\gamma(w)}^{-1}\}_{\gamma\in\Gamma}\bigr]$
and defines a function analytic there and meromorphic in
$\bbC\cup\{\infty\}\setminus\calL$. $B$ has simple poles at the
points $\{\ol{\gamma(w)}^{-1}\}_{\gamma\in\Gamma}$, simple zeros at
$\{\gamma(w)\}_{\gamma\in\Gamma}$ and no other zeros or poles.
Moreover,
\begin{SL}
\item[{\rm{(i)}}] For $z,w\in\bbD$,
\begin{equation} \lb{4.14}
\abs{B(z,w)} = \abs{B(w,z)}
\end{equation}

\item[{\rm{(ii)}}] Each $B(\dott,w)$ is character automorphic, that is, for every $w\in\bbD$ there is a
character, $c_w$, on $\Gamma$ so that
\begin{equation} \lb{4.15}
B(\gamma(z),w) =c_w (\gamma) B(z,w)
\end{equation}

\item[{\rm{(iii)}}] If
\begin{equation} \lb{4.16}
B(z)\equiv B(z,0)
\end{equation}
then for
$z\in\bbC\cup\{\infty\}\setminus\bigl[\calL\cup\{\ol{\gamma(0)}^{-1}\}_{\gamma\in\Gamma}\bigr]$,
\begin{equation} \lb{4.17}
\abs{B(z)}=\prod_{\gamma\in\Gamma}\, \abs{\gamma(z)}
\end{equation}

\item[{\rm{(iv)}}] For $e^{i\theta}\in\partial\bbD\setminus\calL$ {\rm{(}}with $^\prime =\partial/\partial\theta${\rm{)}},
\begin{equation} \lb{4.17a}
\abs{B'(e^{i\theta})} =\sum_{\gamma\in\Gamma}\, \abs{\gamma'(e^{i\theta})}
\end{equation}
\end{SL}
\end{theorem}

\begin{remarks} 1. We will see below (Theorem~\ref{T4.4}) that $c_0(\gamma)$ is not the identity.

\smallskip
2. When $z\in\{\ol{\gamma(0)}^{-1}\}_{\gamma\in\Gamma}$, both sides
of \eqref{4.17} are infinite.
\end{remarks}

\begin{proof} By Theorem~\ref{T3.1} (or Burnside's theorem), for any $w\in\bbD$,
\begin{equation} \lb{4.18}
\sum_{\gamma\in\Gamma}\, \abs{1-\gamma(w)} <\infty
\end{equation}
So, by Lemma~\ref{L4.1} and \eqref{4.9b}, the product converges where claimed and has the claimed
analytic/zero/pole properties.

(i) is immediate from \eqref{4.10}. By the proof of \eqref{4.11}, there exists $\eta(w,\gamma)\in\partial\bbD$
so that
\begin{equation} \lb{4.19}
b(\gamma(z),\gamma(w)) =\eta(w,\gamma) b(z,w)
\end{equation}
Thus, for any finite subset, $G\subset\Gamma$,
\begin{align}
\prod_{\gamma'\in G} b(\gamma(z), \gamma'(w))
&= \prod_{\gamma'\in G} b(\gamma(z), \gamma(\gamma^{-1} \gamma'(w))) \notag \\
&= \prod_{\gamma'\in G} \eta(\gamma^{-1} \gamma'(w),\gamma)
\prod_{\gamma''\in\gamma^{-1}(G)} b(z,\gamma''(w)) \lb{4.20}
\end{align}

If $G_n\subset G_{n+1}$ with $\cup_n G_n=\Gamma$ and $\gamma$ is
fixed, $\gamma^{-1}(G_n)\subset \gamma^{-1} (G_{n+1})$ and $\cup_n
\gamma^{-1}(G_n)=\Gamma$, so the left side of \eqref{4.20} and the
last factor on the right converge to $B(\gamma(z),w)$ and $B(z,w)$,
respectively. Thus, the product of $\eta$'s converges to some
$c_w(\gamma)\in\partial\bbD$, that is, \eqref{4.15} holds. From
\eqref{4.15}, it is easy to see that $c_w(\gamma\gamma')=c_w(\gamma)
c_w(\gamma')$. That proves (ii).

To get \eqref{4.17}, suppose first $z\in\bbD$. Then by \eqref{4.14},
\[
\abs{B(z)} =\abs{B(z,0)} = \abs{B(0,z)} = \biggl|\, \prod_{\gamma\in\Gamma} \gamma(z)\biggr|
\]
proving \eqref{4.17} for $z\in\bbD$.

Since for $z\in\bbD$,
\begin{equation} \lb{4.15x}
\ol{B(1/\bar z)}=B(z)^{-1} \qquad \ol{\gamma(1/\bar z)} = \gamma(z)^{-1}
\end{equation}
(on account of $\abs{B(e^{i\theta})}=\abs{\gamma(e^{i\theta})}=1$
and the reflection principle), we get \eqref{4.17} on
$\bbC\cup\{\infty\}\setminus\bigl[\,\ol{\bbD}\cup\{\ol{\gamma(0)}^{-1}\}_{\gamma\in\Gamma}\bigr]$.
Finally, on $\partial\bbD\setminus\calL$, both sides of \eqref{4.17}
are $1$. This completes (iii).

To prove (iv), we note that if $g$ is an analytic function on $\bbD$ with $\abs{g(z)}<1$ on $\bbD$ and
$I\subset\partial\bbD$ is an open interval so that $g$ has an analytic continuation across $I$ with
$\abs{g(e^{i\theta})}=1$ on $I$, then
\begin{equation} \lb{4.22}
\f{\partial}{\partial\theta}\, \abs{g(e^{i\theta})} =0
\end{equation}
so, by Cauchy--Riemann equations,
\begin{equation} \lb{4.23}
\left. \f{\partial}{\partial r}\, \arg (g(re^{i\theta}))\right|_{r=1} =0
\end{equation}
and
\begin{equation} \lb{4.22x}
\f{\partial}{\partial\theta}\, \arg (g(e^{i\theta})) = \left.
\f{\partial}{\partial r}\, \abs{g(re^{i\theta})}\right|_{r=1}
>0
\end{equation}
Hence, with $^\prime = \f{\partial}{\partial\theta}$,
\begin{equation} \lb{4.23x}
\abs{g'(e^{i\theta})} = \left. \f{\partial}{\partial r}\, \abs{g(re^{i\theta})}\right|_{r=1}
= \biggl| \f{dg}{dz} \, (e^{i\theta})\biggr|
\end{equation}
and
\begin{equation} \lb{4.24}
g'(e^{i\theta}) = \abs{g'(e^{i\theta})} g(e^{i\theta})
\end{equation}

In \eqref{4.22x}, we have strict positivity by the same argument that shows boundary values of Herglotz
functions are strictly monotone.

Let $G$ be a finite subset of $\Gamma$ and
\begin{equation} \lb{4.25}
B_G(z) =\prod_{\gamma\in G} b(z,\gamma(0))
\end{equation}
Then
\begin{equation} \lb{4.26}
\left. \f{\partial}{\partial r}\, \abs{B_G(z)}\right|_{z=e^{i\theta}} =
\sum_{\gamma\in G} \, \left. \f{\partial}{\partial r} \, \abs{b(z,\gamma(0))}\right|_{z=e^{i\theta}}
\end{equation}
by Leibnitz's rule and $\abs{b(e^{i\theta},\gamma(0))}=1$. If
$z_0\in\partial\bbD\setminus\calL$, $B_G(z)\to B(z)$ for $z$ in a
neighborhood of $z_0$, so derivatives converge. Since $\f{\partial}
{\partial r} \abs{B_G(z)} = \abs{\f{d}{dz} B_G(z)}$, the
$\f{\partial}{\partial r}$ derivatives converge. The terms in the
sum are positive, so the sum over all of $\Gamma$ is absolutely
convergent and \eqref{4.26} extends to the limit. By
\eqref{4.22x}--\eqref{4.23x}, we get \eqref{4.17a}.
\end{proof}

We emphasize for later use that \eqref{4.22x}--\eqref{4.23x} imply
\begin{equation} \lb{4.25x}
\f{\partial}{\partial\theta}\, \arg(\gamma (e^{i\theta})) = \abs{\gamma'(e^{i\theta})} >0
\end{equation}

Recall that the (real-valued) \emph{potential theoretic Green's
function}, $G_\fre$, is uniquely determined by requiring
$G_\fre(z)-\log\abs{z}$ to be harmonic on
$\bbC\cup\{\infty\}\setminus\fre$, and for quasi-every $x\in\fre$,
\begin{equation}
\lim_{\substack{z\to x \\ z\notin \fre }} G_\fre (z)=0.
\end{equation}
In fact, when $\fre$ has the form \eqref{1.1}--\eqref{1.3}, $G_\fre$
can be chosen globally continuous on $\bbC$ with
\begin{equation} \lb{4.27x}
G_\fre\restriction \fre =0
\end{equation}
Moreover, near infinity,
\begin{equation} \lb{4.26x}
G_\fre(z) =\log\abs{z} - \log (\ca(\fre)) + O(z^{-1})
\end{equation}
and, by the reflection principle, $G_\fre(z)$ is real analytic in
$x$ and $\abs{y}$ near any $z_0\in\fre^\intt$. For $x\in\fre^\intt$,
we have
\begin{equation} \lb{4.27}
\rho_\fre(x) = \lim_{\veps\downarrow 0}\, \frac{1}{\pi}\f{\partial
G_\fre}{\partial y}\, (x+i\veps)
\end{equation}
which we will write as
\begin{equation} \lb{4.28}
\frac{1}{\pi} \f{\partial}{\partial n}\, G_\fre (x+i0)
\end{equation}
the normal derivative in the positive direction.

\begin{theorem}\lb{T4.4} Let \,$\Gamma$ be the Fuchsian group of a finite gap set, $\fre$. Then for all $z\in
\bbD\setminus\{0\}$,
\begin{equation} \lb{4.30}
\abs{B(z)} = e^{-G_\fre (\x(z))}
\end{equation}
Moreover,
\begin{SL}
\item[{\rm{(i)}}] If $x_\infty$ is given by
\begin{equation} \lb{4.30a}
\x(z) = \f{x_\infty}{z} + O(1)
\end{equation}
near $z=0$, then, also near $z=0$,
\begin{equation} \lb{4.31}
B(z) = \f{\ca(\fre)}{x_\infty} \, z+O(z^2)
\end{equation}

\item[{\rm{(ii)}}] The character, $c_0$, of $B(z)$ is given by
\begin{equation} \lb{4.32}
c_0 (\gamma_j) = \exp(2\pi i \rho_\fre ([\alpha_1, \beta_j]))
\end{equation}

\item[{\rm{(iii)}}] At any $x_0\in \{\alpha_j,\beta_j\}_{j=1}^{\ell+1}$, $G_\fre$ has a square root zero
in the sense that
\begin{align}
\lim_{x\uparrow\alpha_j}\, G_\fre(x) (\alpha_j-x)^{-1/2} &= a_j \lb{4.33} \\
\lim_{x\downarrow \beta_j}\, G_\fre (x) (x-\beta_j)^{-1/2} &= b_j \lb{4.34}
\end{align}
for nonzero $a_j,b_j$.
\end{SL}
\end{theorem}

\begin{proof} By \eqref{4.15}, $\abs{B(z)}$ is automorphic, so there exists a function $h$ on $\bbC\cup
\{\infty\}\setminus\fre$ such that
\begin{equation} \lb{4.35}
\abs{B(z)} =h(\x(z))
\end{equation}
Since $\x$ is analytic and $\x B$ is nonvanishing and analytic in a
neighborhood of $\ol\calF$, $\log(h(x)) + \log|x|$ is harmonic on
$\bbC\cup\{\infty\}\setminus\fre$. By $\abs{B(z)}=1$ on
$\ol\calF\cap\partial \bbD$, $\log (h(x))\to 0$ as $x\to\fre$. Thus,
$-\log(h(x))$ is $G_\fre(x)$, proving \eqref{4.30}.

\eqref{4.31} is immediate from \eqref{4.30a} and \eqref{4.26x}.
\eqref{4.33} and \eqref{4.34} follow from \eqref{4.30}, $B'(z)\neq
0$ on $\ol\calF\cap\partial\bbD$, and the fact that on $\x^{-1}
(\{\alpha_j,\beta_j\}_{j=1}^{\ell+1})$, we have $\x'(z) =0$,
$\x''(z)\neq 0$. That leaves \eqref{4.32}.

Consider the generator, $\gamma_\ell$, whose action takes $0$ into the endpoint of the curve in the top of
Figure~2. Since
\begin{equation} \lb{4.36}
B(\gamma_\ell (0)) = c_0 (\gamma_\ell) B(0)
\end{equation}
we see that
\begin{equation} \lb{4.37}
\arg(c_0 (\gamma_\ell)) = \int_0^{\gamma_\ell(0)}
\negthickspace\negthickspace\negthickspace
\negthickspace\negthickspace\negthickspace\negthickspace\negthickspace\negthickspace
\negthickspace\circlearrowleft \quad\, \f{d}{dz}\, \arg(B(z))\, dz
\end{equation}

$G_\fre(z)$ is harmonic on $\bbC\setminus\fre$ so that, locally, it
is the real part of an analytic function, $\ti G_\fre (z)$, but that
function has a multivalued imaginary part. Thus, $e^{-\ti
G_\fre(z)}\equiv E(z)$ has a multivalued argument. Clearly,
\begin{equation} \lb{4.38}
B(z)=E(\x(z))
\end{equation}
and the change of argument in \eqref{4.37} is given by the change of argument of $E(z)$ over the curve
in the bottom of Figure~2. This is given by a sum of change of argument around curves surrounding each band,
since in the gaps and in $(-\infty,\alpha_1)$, there is cancellation between top and bottom.

By a Cauchy--Riemann equation, $$\f{\partial}{\partial x} \arg(E(z))
=\f{\partial}{\partial y} \log \abs{E(z)} =\f{\partial
G_\fre}{\partial n}$$ The normal derivatives on top and bottom of a
band have opposite sign, so given the opposite directions,
\begin{align}
\arg (c_0(\gamma_\ell))
&= \sum_{j=1}^\ell \int_{\alpha_j}^{\beta_j} 2\, \f{\partial G_\fre(x+i0)}{\partial n}\,dx \lb{4.39} \\
&= 2\pi \sum_{j=1}^\ell \int_{\alpha_j}^{\beta_j} \rho_\fre(x)\, dx
= 2\pi \rho_\fre ([\alpha_1, \beta_\ell]) \lb{4.40}
\end{align}
which is \eqref{4.32} for $j=\ell$. \eqref{4.40} follows from \eqref{4.27}. The argument for general $j=1,
\dots, \ell-1$ is similar.
\end{proof}

\begin{corollary}\lb{C4.5} $B(z)^p$ is automorphic if and only if
\begin{equation} \lb{4.42}
\rho_\fre ([\alpha_j,\beta_j]) = \f{q_j}{p}
\end{equation}
for integers $q_j$.
\end{corollary}

\begin{remark} We will eventually see (Corollary~\ref{C6.3}) that this relates periodic $\fre$ to periodic
Jacobi matrices.
\end{remark}

\begin{proof} $B(z)^p$ is automorphic if and only if $c_0(\gamma_j)^p =1$ for $j=1, 2, \dots, \ell$ and this,
given \eqref{4.32}, is equivalent to \eqref{4.42}.
\end{proof}

\begin{corollary} \lb{C4.6} Let $\fre$ be a finite gap set and $d\rho_\fre$ its equilibrium measure. Then
\begin{equation} \lb{4.43}
\int_{\partial\bbD} f(\x(e^{i\theta}))\, \f{d\theta}{2\pi} = \int_\fre f(x)\, d\rho_\fre(x)
\end{equation}
for all continuous $f$ on $\fre$ and, if integrals are allowed to be
infinite, for any positive measurable $f$ on $\fre$. In particular,
$f\in L^p (\fre,d\rho_\fre)$ if and only if $f\circ\x\in L^p
(\partial\bbD, \f{d\theta}{2\pi})$ so the Szeg\H{o} conditions,
\begin{equation} \lb{4.44}
\int_{\partial\bbD} \log(f(\x(e^{i\theta})))\, \f{d\theta}{2\pi} >-\infty
\end{equation}
and
\begin{equation} \lb{4.45}
\int_\fre \log(f(x))\, d\rho_\fre(x) >-\infty
\end{equation}
are equivalent for $f\in L^1 (\fre, d\rho_\fre)$.
\end{corollary}

\begin{proof} It suffices to prove \eqref{4.43} for continuous functions, $f$, and then use standard
approximation arguments. Recall that $\ti\calR$ is given by
\eqref{3.5a} and consists of $2\ell$ arcs. Except for endpoints, it
is a fundamental domain for the action of $\Gamma$ on
$\partial\bbD$, so
\begin{align}
\int_{\partial\bbD} f(\x(e^{i\theta}))\, \f{d\theta}{2\pi}
&=\sum_{\gamma\in\Gamma} \int_{\gamma (\ti\calR)} f(\x(e^{i\theta}))\, \f{d\theta}{2\pi} \lb{4.46} \\
&=\sum_{\gamma\in\Gamma}\, \int_{\ti\calR} f(\x(e^{i\theta})) \abs{\gamma'(e^{i\theta})}\,
\f{d\theta}{2\pi} \lb{4.47}
\end{align}
by a change of variables and the invariance of $\x$ under $\Gamma$,
that is, $\x\circ\gamma =\x$. Thus, by \eqref{4.17a},
\begin{align}
\int_{\partial\bbD} f(\x(e^{i\theta}))\, \f{d\theta}{2\pi}
&= \int_{\ti\calR} f(\x(e^{i\theta})) \abs{B'(e^{i\theta})}\, \f{d\theta}{2\pi} \lb{4.48} \\
&= \int_{\ti\calR} f(\x(e^{i\theta})) \biggl| \f{\partial}{\partial n}\, e^{-G_\fre (\x(e^{i\theta}))}\biggr|
\, \f{d\x(e^{i\theta})}{d\theta}\, \f{d\theta}{2\pi} \lb{4.49}
\end{align}
where we use \eqref{4.30}, $\abs{B'(e^{i\theta})}= \left.
\f{d}{dr}\abs{B(re^{i\theta})}\right|_{r=1}$, and the chain rule to
go from $d\theta$ to $d\x$ derivatives. $\x$ on $\ti\calR$ is a
two-fold cover of $\fre$, so using $\f{\partial}{\partial n}
e^{-G_\fre(\x)} = -\f{\partial}{\partial n} G_\fre$ (since $G_\fre$
is $0$ on $\fre$) and \eqref{4.27}--\eqref{4.28}, we get
\begin{equation}
\lb{4.50} \int_{\partial\bbD} f(\x(e^{i\theta})) \,
\f{d\theta}{2\pi} = 2\int_\fre \pi\rho_\fre(x) f(x)\, \f{dx}{2\pi} =
\int_\fre f(x)\, d\rho_\fre(x)
\end{equation}
\end{proof}

That concludes what we need about Blaschke products in this paper, but we put in some results on products of
Blaschke products which will be critical in later papers in this series.

\begin{theorem}\lb{T4.7} Let $(w_k)_{k=1}^\infty$ be a sequence in $\calF$. Then
\begin{equation} \lb{4.52}
\sum_{k=1}^\infty\, (1-\abs{w_k})<\infty \;\Leftrightarrow\;
\sum_{k=1}^\infty \, (1-\abs{B(w_k)}) <\infty
\end{equation}
Moreover,
\begin{SL}
\item[{\rm{(i)}}] If $\sum_{k=1}^\infty (1-\abs{w_k})=\infty$, then $\prod_{k=1}^N B(z,w_k)$ converges to $0$
uniformly on compact subsets of $\bbD$.

\item[{\rm{(ii)}}] If
\begin{equation} \lb{4.52a}
\sum_{k=1}^\infty \, (1-\abs{w_k}) <\infty
\end{equation}
then for all $z\in\bbD$,
\begin{equation} \lb{4.53}
\sum_{k=1}^\infty \, \abs{1-B(z,w_k)} <\infty
\end{equation}
uniformly on compact subsets of $\bbD$, so $\prod_{k=1}^N B(z,w_k)$ converges to an analytic limit vanishing
if and only if $z\in\{\gamma(w_k)\}_{\gamma\in\Gamma,\, k=1, \dots}$.
\end{SL}
\end{theorem}

\begin{remarks} 1. If $\{\gamma_k\}$ is a countable set of distinct elements in $\Gamma$ and $w_k=\gamma_k(0)$, then $\sum_{k=1}^\infty
(1-\abs{w_k})<\infty$ (by Burnside or Beardon), but
$\abs{B(z,\gamma_k(0))}=\abs{B(z)}$ and $\prod_{k=1}^N B(z,w_k)$
converges to $0$ uniformly on compact subsets of $\bbD$. Thus, the
condition $w_k\in\calF$ cannot be replaced by $w_k\in\bbD$ in (ii).

\smallskip
2. As in the case of Theorem~\ref{T4.3}, we can prove convergence on
open subsets of $\bbC\cup\{\infty\}\setminus\bigl[\calL\cup
\{\ol{\gamma(w_k)}^{-1}\}_{\gamma\in\Gamma,\, k=1, \dots}\bigr]$
with poles at $\{\ol{\gamma (w_k)}^{-1}\}_{\gamma\in\Gamma,\, k=1,
\dots}$.
\end{remarks}

\begin{proof} $B$ is analytic in a neighborhood of $\ol\calF\cap\bbD$, as noted in \eqref{4.22x},
$\abs{B' (e^{i\theta})}\neq 0$, and, of course, $\abs{B(e^{i\theta})} =1$. Thus, for nonzero constants,
$c,d$, and for all $w\in\calF$,
\begin{equation} \lb{4.54}
c(1-\abs{B(w)}) \leq 1-\abs{w} \leq d(1-\abs{B(w)})
\end{equation}
from which \eqref{4.52} is immediate. (Notice that \eqref{4.54} only holds on $\calF$, not on $\bbD$, and
is where the condition $w_k\in\calF$ is used.)

To prove (i), we need only note that
\begin{equation} \lb{4.55}
\abs{B(z,w_k)} \leq \abs{b(z,w_k)}
\end{equation}
and use Lemma~\ref{L4.1}(a).

To prove (ii), we note that $\prod_{k=1}^\infty B(z,w_k)$ is a
product of Blaschke products, so to prove \eqref{4.53}, it suffices
to prove that
\begin{equation} \lb{4.56}
\sum_{\gamma\in\Gamma, k} (1-\abs{\gamma(w_k)}) <\infty
\end{equation}
Since $\sum_{\gamma\in\Gamma} (1-\abs{\gamma(0)})<\infty$ and zero
occurs at most finitely often in the sequence $(w_k)_{k=1}^\infty$
if \eqref{4.52a} holds, we can suppose that no $w_k$ is zero, in
which case, since $w_k\in \calF$ implies (by \eqref{2.6c})
\begin{equation} \lb{4.57x}
\abs{w_k} = \inf_{\gamma\in\Gamma}\, \abs{\gamma(w_k)}
\end{equation}
we have
\begin{equation} \lb{4.57}
\inf_{\gamma\in\Gamma, k}\, \abs{\gamma(w_k)} >0
\end{equation}
This implies that \eqref{4.56} is equivalent to
\begin{equation} \lb{4.58}
\prod_{\gamma\in\Gamma, k}\, \abs{\gamma(w_k)} >0
\end{equation}
which, by \eqref{4.17}, is equivalent to
\begin{equation} \lb{4.59}
\prod_k \, \abs{B(w_k)} >0
\end{equation}
Now, \eqref{4.59} is implied by
\begin{equation} \lb{4.60}
\sum_k\, (1-\abs{B(w_k)}) <\infty
\end{equation}
As we have seen, \eqref{4.52a} implies \eqref{4.60}, and thus
\eqref{4.56}.
\end{proof}

We are especially interested in the case where $w_k$ is determined
by $w_k\in\calF$ and $\x(w_k)$ real (so $w_k\in [\cup_{j=1}^\ell
C_\ell^+]\cup(-1,1)$). In that case,

\begin{proposition} \lb{P4.8} Let $(x_k)_{k=1}^\infty$ be a sequence in $\bbR\setminus\fre$ and let
$w_k\in\calF$  be uniquely determined by
\begin{equation} \lb{4.61}
\x(w_k)=x_k
\end{equation}
Then the following are equivalent:
\begin{alignat}{2}
&\text{\rm{(i)}} \qquad && \sum_{k=1}^\infty (1-\abs{w_k}) <\infty \lb{4.62} \\
&\text{\rm{(ii)}} \qquad && \sum_{k=1}^\infty G_\fre(x_k) <\infty \lb{4.63} \\
&\text{\rm{(iii)}} \qquad && \sum_{k=1}^\infty \dist(x_k,\fre)^{1/2} <\infty \lb{4.64}
\end{alignat}
\end{proposition}

\begin{proof} By \eqref{4.33}--\eqref{4.34}, we have (ii) $\Leftrightarrow$ (iii). By \eqref{4.52}, we have
that (i) is equivalent to
\begin{equation} \lb{4.65}
\sum_{k=1}^\infty (1-\abs{B(w_k)}) <\infty
\end{equation}
which, by \eqref{4.30} and \eqref{4.61}, is equivalent to
\begin{equation} \lb{4.66}
\sum_{k=1}^\infty\, \abs{1-e^{-G_\fre(x_k)}} <\infty
\end{equation}
In turn, \eqref{4.66} is easily seen to be equivalent to
\eqref{4.63}.
\end{proof}

Finally, we need to discuss alternating Blaschke products. We will
discuss a case with points approaching the top of a gap (or
$\alpha_1$). A similar result holds for approach to a $\beta_j$.

\begin{theorem}\lb{T4.9} Suppose $(\zeta_k)_{k=1}^\infty$, $(\rho_k)_{k=1}^\infty$ obey, for some $j$,
\begin{equation} \lb{4.67}
\beta_{j-1} <\zeta_1 < \rho_1 < \zeta_2 < \rho_2 < \cdots < \alpha_j
\end{equation}
{\rm{(}}where $\beta_0\equiv -\infty${\rm{)}} and
\begin{equation} \lb{4.68}
\lim_{k\to\infty} \zeta_k = \alpha_j
\end{equation}
Let $\{z_k\}\cup\{p_k\} \subset\calF$ and $a_j\in\ol\calF$ be given
by
\[
\x(z_k)=\zeta_k \qquad \x(p_k)=\rho_k \qquad \x(a_j)=\alpha_j
\]
Then as $N\to\infty$,
\begin{equation} \lb{4.69}
\prod_{k=1}^N \f{B(z,z_k)}{B(z,p_k)} \to B_\infty (z)
\end{equation}
uniformly in $z$ on compact subsets of
\begin{equation} \lb{4.70}
\bbC\cup\{\infty\}\setminus [\calL\cup \{\gamma(p_k), 
\gamma(z_k^{-1})\}_{\gamma\in\Gamma,\, k=1, \dots} \cup
\{\gamma(a_j)\}_{\gamma\in\Gamma}]
\end{equation}
to a function which is analytic on the set in \eqref{4.70} with
simple poles at points in
$\{\gamma(p_k), 
\gamma(z_k^{-1})\}_{\gamma\in\Gamma,\, k=1,\dots}$ and with zeros
only at the points $\{\gamma(p_k^{-1}),
\gamma(z_k)\}_{\gamma\in\Gamma,\,
k=1,\dots}$.

Moreover,
\begin{equation} \lb{4.71}
z\in\partial\bbD\setminus
[\calL\cup\{\gamma(a_j)\}_{\gamma\in\Gamma}] \;\Rightarrow\;
\abs{B_\infty(z)} =1
\end{equation}
and for some constant $C$ {\rm{(}}$\fre$-dependent{\rm{)}},
\begin{equation} \lb{4.72}
z\in\calF \Rightarrow \abs{\arg(B_\infty(z))}\leq C
\end{equation}
if $\arg(B_\infty)$ is determined by requiring one value of
$\arg(B_\infty(0))$ to be zero.

In addition, if we remove the arc of $C_{j-1}^+$ $($or segment of
$[-1,0)$ if $j=1$$)$ that runs from $z_1$ to $a_j$ and all its
images under $\gamma\in\Gamma$, we get a region, $\calB$, free of
zeros and poles of $B_\infty$, on which
\begin{equation} \lb{4.73x}
z\in\calB\setminus \calR_{n+1} \Rightarrow
\abs{\arg(B_\infty(z))}\leq (2n+1)C
\end{equation}
\end{theorem}

\begin{remarks}
1. If $z_k\in C_j^+$ for some $j$, then $\gamma_j^{-1}(z_k)=\bar
z_k$, so $\{\gamma(z_k^{-1})\}_{\gamma\in\Gamma} = \{\gamma({\bar
z_k^{-1}})\}_{\gamma\in\Gamma}$, which is why we do not need to put
complex conjugates in \eqref{4.70}.

\smallskip
2. The analog of this result for $\fre=[-2,2]$ is from Simon \cite{S288}.
\end{remarks}

\smallskip
\noindent{\it Sketch} (see \cite{Rice} for details). One first shows
that if $\zeta,\omega$ run through a compact set, $Q$, in a single
${C_j^+}$ (including the endpoints) or $[-1,0)$ or $(0,1]$ and $z$
through a compact subset, $K$\!, of $\bbC \setminus\calL\cup
[\{\gamma (Q)\}_{\gamma\in\Gamma}\cup
\{\gamma(Q^{-1})\}_{\gamma\in\Gamma}]$, then there is $C<\infty$ so
that for all $\gamma\in\Gamma$, $\zeta,\omega\in Q$ and $z\in K$\!,
\begin{equation} \lb{4.73}
\abs{b(z,\gamma(\zeta))-b(z,\gamma(\omega))} \leq C\abs{\gamma(\zeta)-\gamma(\omega)}
\end{equation}
This comes from looking at the three parts of
\begin{equation} \lb{4.74}
b(z,w) = \f{\abs{w}}{w}\f{w-z}{1-\bar w z}
\end{equation}
From this and telescoping, one gets
\begin{equation} \lb{4.75}
\abs{B(z,\zeta)-B(z,\omega)} \leq C \sum_{\gamma\in\Gamma}\, \abs{\gamma(\zeta)-\gamma(\omega)}
\leq C_1 \abs{\zeta-\omega}
\end{equation}

Since
\begin{equation} \lb{4.76}
\inf_{\substack{z\in K \\ w\in Q}}\, \abs{B(z,w)} >0
\end{equation}
\eqref{4.75} implies that
\begin{equation} \lb{4.77}
\biggl| 1-\f{B(z,\zeta)}{B(z,\omega)}\biggr| \leq C_2 \abs{\zeta-\omega}
\end{equation}
which leads to the convergence of \eqref{4.69} if we note that
\begin{equation} \lb{4.78}
\sum_{k=1}^\infty\, \abs{z_k-p_k} <\infty
\end{equation}
since the sum is bounded by the arclength of $C_j^+$ (or by $1$ if
$\alpha_j=\alpha_1$). This easily leads to all the statements except
those about $\arg(B_\infty)$.

For any smooth function, $f$, on a circle $C=\{z\mid z=z_0 + re^{i\theta}\}$, define
\begin{equation} \lb{4.79}
\Var_C(f) = \int_0^{2\pi} \biggl| \f{d}{d\theta}\, f(z)\biggr|\, d\theta
\end{equation}
For $w$ outside $C$, let
\begin{equation} \lb{4.80}
f_w(z) = \arg(w-z)
\end{equation}
Then this $\arg$ is increasing on one arc between the tangents to $C$ from $w$ and decreasing on the other, so
\begin{equation} \lb{4.81}
\Var_C(f_w) =2\times \text{angle between tangents}
\end{equation}
This shows that
\begin{equation} \lb{4.82}
\Var_C(f_w) \leq 2\pi
\end{equation}
and
\begin{equation} \lb{4.83}
\Var_C(f_w) \leq \f{4\text{ radius}(C)}{\dist(w,C)}
\end{equation}

Since $\arg(B(z,w))$ is built out of such $\arg(z-\gamma(w))$ and
$\arg(z-\ol{\gamma(w)}^{-1})$, and the radii of the circles
containing $\gamma(C_j^+)$ decrease so fast that, by Beardon's
theorem,
\begin{equation} \lb{4.84}
\sum_{\gamma\in\Gamma} \text{radius}(\gamma(C_j^+)) <\infty
\end{equation}
we find, uniformly for $z\in\calF$, that
\begin{equation} \lb{4.85}
\Var_{C_j^+} (\arg(B(z,\dott))) \leq C_0
\end{equation}
for some finite constant $C_0$.

Consider $\arg(B_\infty(z;\{z_k\},\{p_k\}))$ as $p_k$ is changed
from $z_k$ to its final value. At $p_k\equiv z_k$, this $\arg$ is
$0$ and the total change is bounded by the variation of
$\arg(B(z,w))$ as $w$ varies over $C_{j-1}^+$. We conclude that on
$\calF$,
\begin{equation} \lb{4.86}
\abs{\arg(B_\infty(z))} \leq C_0
\end{equation}
proving \eqref{4.72}. Since $B_\infty$ is character automorphic, the
variation of $\arg(B_\infty(z))$ over any $\gamma(\calF)$ is at most
$2C_0$, from which \eqref{4.73x} is immediate.

\section{Theta Functions and Abel's Theorem} \lb{s5}

Given a general compact Riemann surface, $\calS$, of genus $\ell$,
one can construct a natural map, $\frA$, called the \emph{Abel map}
from $\calS$ to a $2\ell$-dimensional real torus, called the Jacobi
variety, realized as $\bbC^\ell /\bbL$ where $\bbL$ is a
$2\ell$-dimensional lattice. Once a base point in $\calS$ is fixed,
the group structure comes into play. The theory of meromorphic
functions---essentially, which finite subsets of $\calS$ can occur
as zeros and poles---is described using $\frA$ via a result called
\emph{Abel's theorem}.

As we will see in the next section, certain $m$-functions of Jacobi
matrices with $\sigma_\ess(J)=\fre$ define meromorphic functions on
the Riemann surface, $\calS$, constructed in Section~\ref{s2}. Their
zeros and poles lie only at $\infty_+$, $\infty_-$, or in the sets
$G_j$ of \eqref{2.5a}. $\frA$ takes $\cup_j G_j$ into an
$\ell$-dimensional torus inside the $2\ell$-dimensional Jacobi
variety (the real part of the Jacobi variety), which is also a
subgroup with a suitable choice of base point. The traditional
construction of the isospectral torus
(\cite{DubMatNov,FlMcL,Krich1,McvM1,vMoer}) uses this general theory
of the Abel map and Abel's theorem.

Here, following Sodin--Yuditskii \cite{SY}, we restrict ourselves to
meromorphic functions with poles and zeros only at $\infty_+$,
$\infty_-$, and in $\cup_j G_j$. In that case everything can be made
explicit in a way that the real part of the Jacobi variety becomes
just the $\ell$-dimensional torus, $\Gamma^*$, of characters for the
Fuchsian group, $\Gamma$. The key is the definition of some natural
functions on $\bbC\cup\{\infty\}\setminus\calL$ parametrized by
points in $\cup_{j=1}^\ell \ti C_j^+$ (defined below). Our
construction is motivated by the one in \cite{SY} but is more
explicit.

We will need to use a fundamental set of the action of $\Gamma$ on
$\bbC\cup\{\infty\}\setminus\calL$,
\begin{equation} \lb{5.2}
\ti\calF = (\ol\calF \cup \ol\calF^{-1}) \setminus
\bigcup_{j=1}^\ell \ti C_j^-
\end{equation}
where closure is taken in $\bbC$, and $\ti C_j^\pm$ are the
\emph{complete orthocircles} (obtained by adding the two missing
points on $\partial\bbD$ to $C_j^\pm\cup (C_j^\mp)^{-1}$,
respectively); see Figure~4 below.
$\ti\calF^\intt$ will denote its interior, this is a fundamental
region. $\ti\calF$ is then $\ti\calF^\intt$ with $\cup_{j=1}^\ell
\ti C_j^+$ added.

\begin{center}
\begin{figure}[h]
\includegraphics[scale=0.4]{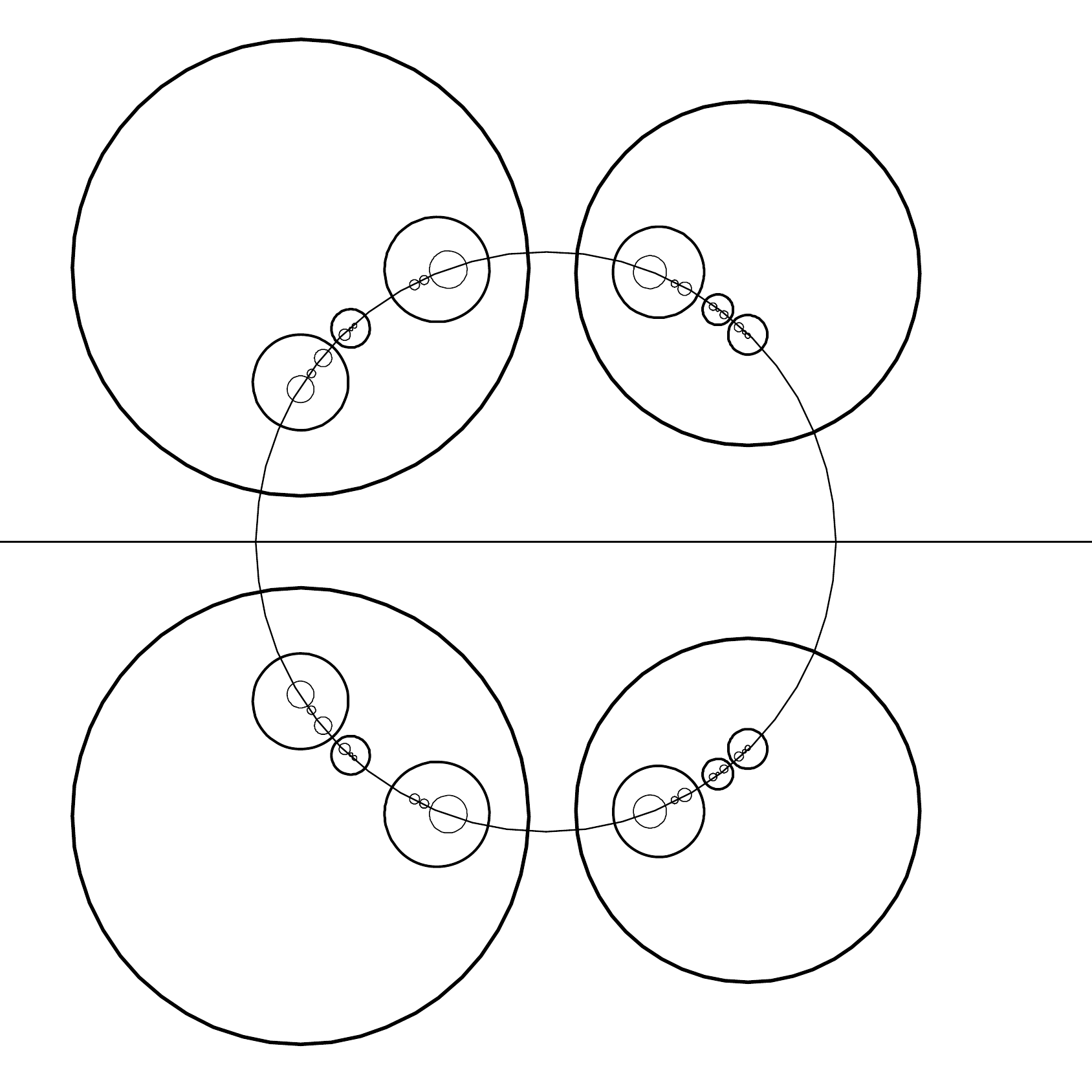}
\caption{Complete orthocircles}
\end{figure}
\end{center}

$\x^\sharp$ maps $\ti C_j^+$ onto $G_j$ bijectively. We will use
$\zeta_1, \dots, \zeta_\ell$ for the unique points on $\ti C_1^+,
\dots, \ti C_\ell^+$ that map to $\beta_1, \dots, \beta_\ell$, that
is,
\begin{equation} \lb{5.1}
\zeta_j\in\ti C_j^+ \cap \partial\bbD \qquad
\x^\sharp(\zeta_j)=\beta_j
\end{equation}

We need the following lemma:

\begin{lemma} \lb{L5.1} Let $f$ be a character automorphic meromorphic function on $\bbC\cup\{\infty\}\setminus\calL$.
Suppose
\begin{SL}
\item[{\rm{(i)}}] $f$ has no zeros or poles in $\ti\calF^\intt$ {\rm{(}}so in $\ti\calF$, the only zeros and poles
are on $\cup_{j=1}^\ell \ti C_j^+${\rm{)}}.
\item[{\rm{(ii)}}] Every zero or pole of $f$ has even order.
\item[{\rm{(iii)}}] If $D_j$ is a counterclockwise contour that is just outside $\ti C_j^+$ {\rm{(}}say, a circle
with the same center but a slightly larger radius{\rm{)}}, then
\begin{equation} \lb{5.4}
\f{1}{2\pi i} \int_{D_j}
\negthickspace\negthickspace\negthickspace\negthickspace\negthickspace\negthickspace\circlearrowleft
\,\, \f{f'(z)}{f(z)}\,  dz =0
\end{equation}
\item[{\rm{(iv)}}]
\begin{equation} \lb{5.5}
f(0) >0
\end{equation}
\end{SL}
Then there is a {\rm{(}}unique{\rm{)}} character automorphic
function, $g$, which we will denote as $\sqrt{f}$, with
\begin{equation} \lb{5.6a}
g(0)>0
\end{equation}
and so that for all $z\in\bbC\cup\{\infty\}\setminus\calL$,
\begin{equation} \lb{5.6}
 g(z)^2 = f(z)
\end{equation}
\end{lemma}

\begin{proof} By \eqref{5.4}, we can define a single-valued function $h(z)$ on $\ti\calF^\intt$ by
\begin{equation} \lb{5.7}
h(z) = \log(f(0)) + \int_0^z \!
\negthickspace\negthickspace\negthickspace\negthickspace\negthickspace\circlearrowleft
\,\, \f{f'(w)}{f(w)}\,  dw
\end{equation}
with any contour in $\ti\calF^\intt$ used. Then
\begin{equation} \lb{5.8}
g(z) = \exp(\tfrac12\, h(z))
\end{equation}
obeys \eqref{5.6a}--\eqref{5.6} and is defined and analytic on
$\ti\calF^\intt$.

Since all poles and zeros on $\ti C_j^+$ are of even order, $g(z)$
can be meromorphically continued to a neighborhood, $N$, of the
closure of $\ti\calF$. Then for each $j$,
$S_j\equiv\{z\in\ti\calF^\intt \mid \gamma_j(z)\in N\}$ is open and
nonempty, and by decreasing $N$, one can suppose each $S_j$ is
connected.

If $z\in S_j$, we have $g(\gamma_j(z))^2 = c_f(\gamma_j) g(z)^2$.
Hence, by continuity and connectedness, there is a single square
root, $c_g(\gamma_j)$, so that
\begin{equation} \lb{5.9}
z\in S_j \;\Rightarrow\; g(\gamma_j(z)) = c_g (\gamma_j) g(z)
\end{equation}
We can use this to extend $g$ to $\cup_j \gamma_j (\ti\calF)$ and
also to $\cup_j \gamma_j^{-1}(\ti\calF)$, and so \eqref{5.9} holds
for all $z$ with $z$ and $\gamma_j(z)$ in the domain of current
definition. In this way, one gets a character automorphic
continuation of $g$ to $\bbC\cup\{\infty\}\setminus\calL$.
\end{proof}

\begin{lemma} \lb{L5.2} Let $\zeta\in\ti C_j^+$ for some $j$. Then
\begin{equation} \lb{5.10}
f(z) = \f{\x(z)-\x(\zeta)}{\x(z)-\x(\zeta_j)}\,\eta(z)\eta(0)^{-1}
\end{equation}
where
\begin{equation} \lb{5.10a}
\eta(z) = \begin{cases}
B(z,\zeta) & \text{if } \zeta\in\bbD \\
\;\quad 1 & \text{if } \zeta\in\partial\bbD \\
B(z,\bar\zeta^{-1})^{-1} &\text{if } \zeta\in\bbC\setminus\ol{\bbD}
\end{cases}
\end{equation}
obeys properties {\rm{(i)--(iv)}} of Lemma~\ref{L5.1}. If
$\zeta\neq\zeta_j$, $f$ has double zeros at
$\{\gamma(\zeta)\}_{\gamma\in\Gamma}$, double poles at
$\{\gamma(\zeta_j)\}_{\gamma\in\Gamma}$, and is otherwise finite and
nonvanishing.
\end{lemma}

\begin{remarks} 1. If $\zeta=\zeta_j$, $f\equiv 1$.

\smallskip
2. When $z\in\{\gamma(0)\}_{\gamma\in\Gamma}$, $\x(z)=\infty$ and
the first factor in \eqref{5.10} is interpreted as $1$. Thus,
because of $\eta(0)^{-1}$, we have $f(0)=1$.
\end{remarks}

\begin{proof} (i) and (iv) are obvious. Moreover, if $\zeta\neq\zeta_j$ then $f(z)/\eta(z)$ has double poles
at $\{\gamma(\zeta_j)\}_{\gamma\in\Gamma}$ (since $\x'(z)=0$,
$\x''(z)\neq 0$ at such points), double zeros at
$\{\gamma(\zeta)\}_{\gamma\in\Gamma}$ if $\zeta$ is the other point
on $\ti C_j^+$ in $\partial\bbD$ (i.e., if
$\x^\sharp(\zeta)=\alpha_j$), and simple zeros at
$\{\gamma(\zeta)\}_{\gamma\in\Gamma}\cup\{\gamma(\bar\zeta^{-1})\}_{\gamma\in\Gamma}$
if $\zeta\notin\partial\bbD$ (since $\x'(z)\neq 0$ at such points
and $\x(\zeta)$ is real). Thus, there are  precisely the claimed
zeros/poles for $f$ since $\eta$ cancels the zeros at
$\{\gamma(\bar{\zeta}^{-1})\}_{\gamma\in\Gamma}$ and doubles the
zeros at $\{\gamma(\zeta)\}_{\gamma\in\Gamma}$. This proves (ii).

To prove (iii), we need only check \eqref{5.4} if $f$ is replaced by
$\eta$ or by $f/\eta$. $f/\eta$ is real on $\partial\bbD$, so if $q$
is the composition of this function and a conformal map of $\bbC$
taking $\bbR$ to $\partial\bbD$, $q$ is real on the set of points in
its domain which lie on $\bbR$. So
\begin{equation} \lb{5.10b}
\f{1}{2\pi i} \int_{\ti D}
\negthickspace\negthickspace\negthickspace\negthickspace\negthickspace\circlearrowleft
\; q(z) \, dz =0
\end{equation}
for any conjugate symmetric curve, and so by contour deformation, for $D_j$ and $(f/\eta)'/(f/\eta)$.

For $\eta$, we note that if $f$ in \eqref{5.4} is replaced by a
finite product $\prod_{\gamma\in G} b(z, \gamma(\zeta))$, the
integral is zero since the finite product is meromorphic inside
$D_j$ with an equal number of (simple) zeros and (simple) poles. By
taking limits, \eqref{5.4} holds for $B(z,\zeta)$, and by
$(1/g)'/(1/g)=-g'/g$ for $B(z,\bar\zeta^{-1})^{-1}$. Thus,
\eqref{5.4} holds for $f$.
\end{proof}

\begin{definition} Let $y\in G_j$ for some $j$ and let $\zeta$ be the unique point in $\ti C_j^+$ with
\begin{equation} \lb{5.11}
\x^\sharp(\zeta)=y
\end{equation}
We define $\Theta(\,\cdot\,;y)$ to be the character automorphic
function
\begin{equation} \lb{5.12}
\Theta(z;y) = \biggl[\biggl( \f{\x(z)-
\x(\zeta)}{\x(z)-\x(\zeta_j)}\biggr)\,
\eta(z)\eta(0)^{-1}\biggr]^{1/2}
\end{equation}
and denote by $\frA(y)\in\Gamma^*$ its character. Moreover, we
define $\frA(\infty)$ to be the character of $B(z)$.
\end{definition}

By the lemma, $\Theta(\,\cdot\,;y)$ is indeed a character
automorphic function on $\bbC\cup\{\infty\}\setminus\calL$ with
simple zeros at $\{\gamma(\zeta)\}_{\gamma\in\Gamma}$ and simple
poles at $\{\gamma(\zeta_j)\}_{\gamma\in\Gamma}$ (and otherwise
nonvanishing and finite). By definition,
\begin{equation} \lb{5.14a}
\Theta(0;y)=1
\end{equation}

We also define
\begin{equation} \lb{5.13a}
\ti\frA \colon \bbG  \to \Gamma^*
\end{equation}
by
\begin{equation} \lb{5.14}
\ti\frA(y_1, \dots, y_\ell) = \frA(y_1) \cdots \frA(y_\ell)
\end{equation}
using the product in $\Gamma^*$. Here
\begin{equation} \lb{5.14b}
\bbG=G_1\times\cdots\times G_\ell.
\end{equation}

\begin{theorem} \lb{T5.3} The map $\ti\frA$ of \eqref{5.13a}--\eqref{5.14} is a real analytic homeomorphism of
$\ell$-dimensional tori.
\end{theorem}

\begin{remark} By real analytic functions, we do not mean real-valued but functions of real parameters
which are given locally by convergent series of those parameters---they are, of course, $C^\infty$.
\end{remark}

\begin{proof} By construction, $\Theta(z;y)$, as a map of $\bbD \times \cup_{j=1}^\ell G_j$ to $\bbC$, is
jointly real analytic. Since
\begin{equation} \lb{5.15}
\frA(y)(\gamma) = \f{\Theta(\gamma(0);y)}{\Theta(0;y)}
\end{equation}
$\frA$ and so $\ti\frA$ are real analytic maps.

Suppose $\vy=(y_1,\ldots,y_\ell)$ and $\ve{w}=(w_1,\ldots,w_\ell)$
lie in $\bbG$ and $\ti\frA(\vy)=\ti\frA(\ve{w})$. Then
\begin{equation} \lb{5.16}
f(z) =\prod_{j=1}^\ell \f{\Theta(z;y_j)}{\Theta(z;w_j)}
\end{equation}
is automorphic since the characters cancel. Hence, there is a unique
meromorphic function $F$ on $\calS$ so that
\begin{equation} \lb{5.16a}
f(z) = F(\x^\sharp(z))
\end{equation}

Let $m$ be the number of $j$ with $y_j\neq w_j$. $F$ has exactly $m$
poles and $m$ zeros, all simple (the poles/zeros where $y_j=w_j$
cancel), and so has degree $m\leq\ell$. If $m\neq 0$, $F\circ
\tau\not\equiv F$ since there is a gap with a single simple zero
(and if $F\circ\tau=F$, $F$ has either two zeros or a double zero at
a branch point). Thus, if $m\neq 0$, we get a contradiction to
Theorem~\ref{T2.2}. It follows that $m=0$, that is, $\vy=\ve{w}$ and
$\ti\frA$ is one-one.

Any smooth one-one map between two smooth, orientable compact
manifolds of the same dimension has degree $\pm 1$, and so is also
surjective (see \cite{FoGang,GuiPo,KrWu,Lloyd,Mil,Spi79} for
expositions of degree theory).
\end{proof}

We saw above that Theorem~\ref{T2.2} is the key to the proof of Theorem~\ref{T5.3}. It is also very powerful in
connection with Theorem~\ref{T5.3} as the following theorems show:

\begin{theorem}\lb{T5.4} Let $f$ be a character automorphic function on $\bbC\cup\{\infty\}\setminus\calL$ with no
zeros or poles. Then $f$ is constant.
\end{theorem}

\begin{proof} Let $c_f\in\Gamma^*$ be the character of $f$. By Theorem~\ref{T5.3},
find $\vy\in\bbG$ with $\ti\frA(\vy)=c_f$. Let
\begin{equation} \lb{5.17}
h(z) = \f{f(z)}{\prod_{j=1}^\ell \Theta(z;y_j)}
\end{equation}
Then $h$ is automorphic, so there is $H$ meromorphic on $\calS$ with
\begin{equation} \lb{5.18}
h(z)=H(\x^\sharp(z))
\end{equation}
$H$ has degree $m$ where $m= \#\{j\mid y_j\neq \beta_j\}$. By
Theorem~\ref{T2.2}, $m=0$, that is, $\vy=(\beta_1, \dots,
\beta_\ell)$ so $\ti\frA(\vy)=1$ and $f$ is automorphic. But then
\begin{equation} \lb{5.19}
f(z)=F(\x^\sharp(z))
\end{equation}
with $F$ analytic on $\calS$, and therefore $f$ is constant.
\end{proof}

\begin{corollary}\lb{C5.5} Let $\zeta\in\ti C_j^+$ and suppose $h$ is a character automorphic meromorphic function with
zeros only at $\{\gamma(\zeta)\}_{\gamma\in\Gamma}$ and poles only
at $\{\gamma(\zeta_j)\}_{\gamma\in\Gamma}$, all simple. If $h(0)=1$,
then
\begin{equation} \lb{5.20}
h(z) = \Theta(z;\x^\sharp(\zeta))
\end{equation}
Moreover,
\begin{equation} \lb{5.24a}
\Theta(\bar z;y) = \ol{\Theta(z;y)}
\end{equation}
and, in particular, $\Theta(\dott;y)$ is real and strictly positive on $\bbR$.
\end{corollary}

\begin{remark} Thus, $\Theta$ and so $\frA$ are unique.
\end{remark}

\begin{proof} Apply Theorem~\ref{T5.4} to $h(z)/\Theta(z;\x^\sharp(\zeta))$.
\end{proof}

\begin{definition} By a \emph{divisor}, we mean a finite subset $\Delta\subset \cup_{j=1}^\ell G_j$ and the assignment of
a nonzero integer $n_x$ to each $x\in\Delta$ plus an assignment of
an integer $n_+$ to $\infty_+$ and
\begin{equation} \lb{5.21}
n_- = -n_+
\end{equation}
to $\infty_-$. We write the divisor formally as
\begin{equation} \lb{5.22}
n_+ \delta_{\infty_+} + n_- \delta_{\infty_-} + \sum_{x\in\Delta}
n_x \delta_x
\end{equation}
\end{definition}

\begin{definition} By a \emph{special meromorphic function}, we mean a meromorphic function on $\calS$ so that
\begin{SL}
\item[(i)] All zeros and poles lie in $[\cup_{j=1}^\ell G_j]\cup\{\infty_\pm\}$.
\item[(ii)] If $n_\pm$ are the order of the zeros and poles at $\infty_\pm$ ($n_\pm >0$ means a zero of order $n_\pm$,
$n_\pm<0$ means a pole of order $-n_\pm$), then \eqref{5.21} holds.
\end{SL}
\end{definition}

\begin{definition} The divisor, $\delta(F)$, of a special meromorphic function $F$ is given by \eqref{5.22} where
$\Delta$ is the set of zeros and poles of $F$, and $n_x$ is the
order of the zero/pole at $x$.
\end{definition}

Notice that we have chosen a base point $\zeta_j$ that depends on
which $G_j$ the point $y$ lies in. For our later applications, where
for each $j$,
\begin{equation} \lb{5.22a}
\sum_{x\in G_j}  n_x =0
\end{equation}
that is very convenient. But when \eqref{5.22a} does not hold, we
will need a factor to move the base point to a fixed point, say
$\zeta_1$. So we define
\begin{equation} \lb{5.22b}
\Psi_j(z) = \sqrt{\f{\x(z)-\x(\zeta_j)}{\x(z)-\x(\zeta_1)}}
\end{equation}
where we can take special roots by Lemma~\ref{L5.1}, verifying
\eqref{5.4} as we did for $(\x(z)-\x(\zeta)) /(\x(z)-\x(\zeta_j))$.

We also need to define $\frA_j$ to be the character of the character
automorphic function $\Psi_j$. For $x\in G_j$, we let
\begin{equation} \lb{5.22c}
\frA^\sharp(x) =\frA_j \frA(x)
\end{equation}

\begin{theorem}[Abel's theorem for special meromorphic functions] \lb{T5.6} A divisor is the divisor of a
special meromorphic function $F$ if and only if
\begin{alignat}{2}
&\text{\rm{(a)}} \qquad && \sum_{x\in\Delta} n_x =0 \lb{5.23} \\
&\text{\rm{(b)}} \qquad && \frA(\infty)^{n_+} \prod_{x\in\Delta}
\frA^\sharp(x)^{n_x} =1 \lb{5.24}
\end{alignat}
the identity element of $\Gamma^*$.
\end{theorem}

\begin{remarks} 1. The proof provides an explicit formula for $F$, namely,
\begin{equation} \lb{5.25}
F(\x^\sharp(z)) = B(z)^{n_+} \prod_{j=1}^\ell \prod_{x\in\Delta\cap
G_j} \Theta(z;x)^{n_x} \Psi_j (z)^{n_x}
\end{equation}

\smallskip
2. Notice that if \eqref{5.22a} holds, we can drop the $\Psi_j$
factors from \eqref{5.25} and change $\frA^\sharp$ to $\frA$ in
\eqref{5.24}.
\end{remarks}

\begin{proof} \eqref{5.23} is an expression of the constancy of degree, that is, that the number of zeros
of $F$ is equal, counting multiplicities, to the number of poles. So
really we need to prove, assuming \eqref{5.23}, that a divisor is a
$\delta(F)$ if and only if \eqref{5.24} holds.

If the divisor obeys \eqref{5.24}, then the right side of
\eqref{5.25} is automorphic, so \eqref{5.25} holds for some
meromorphic $F$ with the proper zeros and poles. Hence, the divisor
is a $\delta(F)$. Because of the $\Psi_j$ factors, the poles of
$\Theta(z;x)$ at $\zeta_j$ are all moved to $\zeta_1$ and the poles
at $\zeta_1$ cancel each other by \eqref{5.23}.

On the other hand, if $F$ is meromorphic and $\delta(F)$ is its
divisor, then letting $f(z)$ be the right side of \eqref{5.25},
\begin{equation} \lb{5.26}
\f{F(\x^\sharp(z))}{f(z)}\equiv h(z)
\end{equation}
is character automorphic with no zeros and poles (again, the poles
and zeros at $\zeta_j$ cancel because of the $\Psi_j$'s and at
$\zeta_1$ by \eqref{5.23}). By Theorem~\ref{T5.4}, $h$ is constant,
so automorphic. Thus, since $F(\x^\sharp(z))$ is automorphic, so is
$f$, which implies \eqref{5.24}.
\end{proof}

\section{The Isospectral Torus} \lb{s6}

Once one has the Abel map and Abel's theorem, the construction of
the isospectral torus along the lines pioneered for KdV
\cite{DubMatNov,McvM1} is straightforward (see \cite[Ch.~5]{Rice}
for an exposition of the original papers \cite{FlMcL,Krich1,vMoer})
but in the covering map guise has a more explicit feel. For
additional discussions of the isospectral torus, see
\cite{Batch,Bulla,GesHol,Geszt,Teschl}.

\begin{definition} A \emph{minimal Herglotz function} for $\fre$ is a meromorphic function $m$ on $\calS$ with
degree precisely $\ell+1$ and which obeys
\begin{SL}
\item[(i)]
\begin{equation} \lb{6.1}
z\in\calS_+ \cap\bbC_+ \;\Rightarrow\; \Ima m(z) >0
\end{equation}
\item[(ii)] Near $\infty_+$,
\begin{equation} \lb{6.2}
m(z) = -\f{1}{z} + O(z^{-2})
\end{equation}
\item[(iii)] $m(z)$ has a pole at $\infty_-$.
\end{SL}
\end{definition}

In the usual way (see \cite{Rice}), $m\restriction\calS_+\cap\bbC_+$
determines a probability measure, $d\mu$, with
\begin{equation} \lb{6.3}
m(z) = \int \f{d\mu(x)}{x-z}
\end{equation}
for $z\in\calS_+\setminus\bbR$. Moreover, the continuity properties
of $m$ as one approaches $\fre$ (and the fact that we will see that
all poles of $m$ are simple) implies that
\begin{equation} \lb{6.4}
d\mu(x)=w(x)dx+d\mu_\s(x)
\end{equation}
where $w$ is real analytic on $\fre^\intt$ and nonvanishing there,
and $d\mu_\s$ is a pure point measure with pure points only in the
open gaps $\cup_{j=1}^\ell (\beta_j,\alpha_{j+1})$ and at most one
pure point per gap.

Condition (iii) may seem ad hoc. We mention now that one can show
(\cite[Ex.~5.13.4]{Rice}) that if (iii) is dropped, the
once-stripped $m$-function, $m_1$, (i.e., $m(z)=\calM(a_1, b_1,
m_1(z))$ in terms of \eqref{1.23b}) obeys condition (iii). Thus, the
extra possibilities allowed if (iii) is dropped result from taking
the Jacobi matrix of a minimal Herglotz function as we have defined
it and extending by one row and column, with the ``wrong'' values of
$a_0$ or $b_0$.

The main elements of the theory are:
\begin{SL}
\item[(i)] The minimal Herglotz functions are in one-one correspondence with $\bbG$, and so form an $\ell$-dimensional
torus, $\calT_\fre$.
\item[(ii)] The correspondence is that the coordinates of $(y_1, \dots, y_\ell)\in\bbG$ are the positions of $\ell$ of the poles of
$m$---the last pole is at $\infty_-$. The zeros are then determined
via the Abel map.
\item[(iii)] The Abel map ``linearizes'' coefficient stripping (i.e., the map \eqref{1.23b}) since the zeros of
$m$ are the poles of the once-stripped $m$-function, $m_1$.
Explicitly,
\begin{equation} \lb{6.5}
\ti\frA \colon\calT_\fre \to\Gamma^*
\end{equation}
and coefficient stripping corresponds to multiplying by the inverse
of the character of $B$.
\item[(iv)] The linearization shows that
the corresponding Jacobi parameters, $\{a_n,b_n\}_{n=1}^\infty$, are
almost periodic sequences with almost periods given by the harmonic
measures $\{\rho_\fre([\alpha_1,\beta_j])\}_{j=1}^\ell$. In
particular, one has periodicity if and only if these numbers are all
rational.
\item[(v)] The construction provides explicit formulae for $m$ and, thus, $a_1,b_1$ (and so, via the Abel map,
$a_n,b_n$) in terms of theta functions and the logarithmic
capacity of $\fre$.
\item[(vi)] Uniformly on $\calT_\fre$, there are bounds on the weight $w$ in \eqref{6.4} of the form
\begin{equation} \lb{6.5a}
C\sqrt{|R(x)|} \leq w(x) \leq D \sqrt{|R(x)|^{-1}}
\end{equation}
where $0<C, D<\infty$.
\end{SL}

\smallskip
We begin by recalling what one can get without using the covering or
Abel maps.

\begin{theorem}\lb{T6.1}
\begin{SL}
\item[{\rm{(i)}}] Every minimal Herglotz function has exactly one simple pole in each gap, one at $\infty_-$
and no others.
\item[{\rm{(ii)}}] For every choice $(y_1, \dots, y_\ell)\in\bbG$, there is exactly one minimal Herglotz function
with poles exactly at $y_1, \dots, y_\ell$ {\rm{(}}and $\infty_-${\rm{)}}.
\item[{\rm{(iii)}}] For every minimal Herglotz function, the once-stripped Herglotz function is also a minimal
Herglotz function.
\item[{\rm{(iv)}}] Every minimal Herglotz function has one zero in each gap, one at $\infty_+$ and no others.
\end{SL}
\end{theorem}

\begin{remark} Of course, (i) and (ii) set up a one-one correspondence between $\calT_\fre$, the set of minimal
Herglotz functions, and $\bbG$. We will often refer to $\bbG$ as the
\emph{isospectral torus}.
\end{remark}

\begin{proof}[Sketch] (See \cite[Ch.~5]{Rice} for details.) (i) \ Every minimal degree meromorphic function, $m$,
on $\calS$ with $m\circ\tau\not\equiv m$ has the form
\begin{equation} \lb{6.6}
m(z) = \f{p(z)+\sqrt{R(z)}}{a(z)}
\end{equation}
where $R(z)$ is given by \eqref{2.1b}, and $p,a$ have degree at most
$\ell+1$. Since $m$ has a zero at $\infty_+$, $p(z)$ must cancel the
leading $O(z^{\ell+1})$ term in $\sqrt{R(z)}$ at $\infty_+$, so
\begin{equation} \lb{6.7}
\deg(p) =\ell+1
\end{equation}

Because this cancellation takes place at $\infty_+$, it does not at
$\infty_-$ (since $\sqrt{R}$ flips sign but $p$ does not). For $m$
to have a simple pole at $\infty_-$, we must have
\begin{equation} \lb{6.8}
\deg(a)=\ell
\end{equation}

$m$ is real on each $[\beta_j,\alpha_{j+1}]$ so, by analyticity, on
the entire gap $G_j$. Thus, $m(z)$ is real on $\bbR\setminus\fre$ on
both sheets. Since $\sqrt{R(z)}$ is real on $\bbR\setminus\fre$, we
conclude first that $a$ is real and then that $p$ is real. In
particular, on $\ol\calS_+\cap\fre$,
\begin{equation} \lb{6.9}
\Ima m(x+i0)=\f{\Ima \sqrt{R(x)}}{a(x)}
\end{equation}

Since $R$ has two zeros between bands, $\Ima \sqrt{R(x)}$ changes
sign between successive bands. As $\Ima m(x+i0)\geq 0$, $a$ must
change signs between bands, that is, have an odd number of zeros in
each gap. Since, by \eqref{6.8}, $a$ has only $\ell$ zeros and there
are $\ell$ gaps, $a$ has one zero per gap.

If $a$ has a zero at $x_0\in(\beta_j,\alpha_{j+1})$, then
\begin{equation}
(p+\sqrt{R})-(p-\sqrt{R})=2\sqrt{R}\neq 0
\end{equation}
at $x_0$, so on one sheet or the other, $m$ must have a pole. If $a$
has a zero at $x_0\in\{\beta_j,\alpha_{j+1}\}$, the numerator is at
best $O((x-x_0)^{1/2})$ and the denominator is $O((x-x_0))$. So
again, $m$ has a pole at $x_0$.

Thus, $m$ has at least one pole per gap. So, since $\deg(m)=\ell+1$
and $m$ has a pole at $\infty_-$, $m$ has exactly one simple pole in
each gap.

\smallskip
(ii) \ Write $y_j=(\pi(y_j),\sigma_j)$ if $\pi(y_j)\in
(\beta_j,\alpha_{j+1})$ with $\sigma_j=1$ (resp.\ $-1$) if
$y_j\in\calS_+$ (resp.\ $\calS_-$). Since $a$ has to vanish at
$\pi(y_j)$ to get a pole at $y_j$, we see that $m$ has a pole at
$y_j$ and not at $\tau (y_j)$ if and only if $a(\pi(y_j))=0$ and
\begin{equation} \lb{6.10}
p(\pi(y_j)) = \sigma_j \sqrt{R(y_j)}
\end{equation}

If $\pi(y_j)\in \{\beta_j,\alpha_{j+1}\}$, then to avoid a double
pole, $p(\pi(y_j))=0$, that is, \eqref{6.10} still holds (since
$\sqrt{R(y_j)}=0$, it does not matter that $\sigma_j$ is undefined).

At $\infty_+$, $a(z)$ is $O(z^\ell)$. Thus, for $m(z)$ to vanish at
$\infty_+$, we must have
\begin{equation} \lb{6.11}
p(z) + \sqrt{R(z)} = O(z^{\ell-1})
\end{equation}
near $\infty_+$. Since $\sqrt{R(z)} =O(z^{\ell+1})$, \eqref{6.11}
determines the top two coefficients (with the top one nonzero) and
then, by standard polynomial interpolation, the $\ell$ conditions
\eqref{6.10} determine the remaining $\ell$ coefficients of $p$.

We have thus proven that given $\vy=(y_1,\ldots,y_\ell)\in\bbG$,
there is a meromorphic function of degree precisely $\ell+1$ with
poles at $y_1, \dots, y_\ell$ and $\infty_-$, and a zero at
$\infty_+$. Moreover, it is unique up to a single overall
constant---for the above determines $p$ and
\begin{equation} \lb{6.12}
a(z)= c \prod_{j=1}^\ell (z-\pi(y_j))
\end{equation}
for some constant $c$.

The fact that $a$ changes sign in each gap shows that the sign of
$c$ in \eqref{6.12} can be picked so that
\begin{equation} \lb{6.13x}
\Ima m(x+i0) >0
\end{equation}
on all bands in $\calS_+$. Keeping track of the argument of
$\sqrt{R(z)}$ as one crosses a branch point shows that with this
choice at each $y_j\in\calS_+$ such that
$\pi(y_j)\in(\beta_j,\alpha_{j+1})$ and $\sigma_j >0$,
\begin{equation} \lb{6.13}
m(z) = \f{c_j}{\pi(y_j)-z} + \Oh(1)
\end{equation}
with $c_j >0$. Thus, any limit point in the values of $\Ima m(z)$ as
$z$ approaches $\bbR$ is nonnegative. Since $\Ima m(z) \to 0$ at
$\infty_+$, the maximum principle applied to the harmonic function
$\Ima m(z)$ on $\calS_+\cap\bbC_+$ shows that \eqref{6.1} holds.

Any function obeying \eqref{6.1} with real boundary values on
$\bbR\setminus\fre$ has \eqref{6.2} holding up to a positive
constant. We can thus adjust $c$ in \eqref{6.12} so that \eqref{6.2}
holds. We have thus proven that there exists precisely one
meromorphic Herglotz function with poles at $y_1,\ldots,y_\ell$.

\smallskip
(iii) \ $m$ and the once-stripped function, $m_1$, are related by
\begin{equation} \lb{6.14}
m(z) = \f{1}{-z+b_1 -a_1^2 m_1(z)}
\end{equation}
where $(a_1,b_1)$ are (and can be) chosen so that \eqref{6.2} holds
for $m_1$. It is always true, of course, that $m_1$ obeys
\eqref{6.1} and \eqref{6.2}.

\eqref{6.14} sets up a one-one correspondence between poles of $m_1$
in $\calS\setminus \{\infty_\pm\}$ and zeros of $m$ there. Since $m$
has degree $\ell+1$, it has $\ell+1$ zeros and only a simple zero at
$\infty_+$ by \eqref{6.2}. Thus, $m$ has precisely $\ell$ zeros in
$\calS\setminus\{\infty_\pm\}$.

Therefore, $m_1$ has exactly $\ell$ poles on
$\calS\setminus\{\infty_\pm\}$ and, obviously, no pole at
$\infty_+$. Since $m(z)$ has a pole at $\infty_-$, \eqref{6.14}
shows that near $\infty_-$,
\begin{equation} \lb{6.15}
a_1^2 m_1(z) = -z+b_1 + O(z^{-1})
\end{equation}
that is, $m_1$ has a simple pole at $\infty_-$. Thus, $m_1$ has
degree exactly $\ell+1$ and we have proven condition (iii) in the
definition of a minimal Herglotz function.

\smallskip
(iv) \ Since $m_1$ has a pole on each $G_j$, by \eqref{6.14}, $m$
has a zero on each $G_j$. There is a zero at $\infty_+$, and this
accounts for all $\ell+1$.
\end{proof}

The above construction also lets us prove \eqref{6.5a}:

\begin{theorem}\lb{T6.1A}
\begin{SL}
\item[{\rm{(i)}}] There are constants, $A,B$, so that uniformly in $\vy\in\bbG$, we have {\rm{(}}with $c$ the
constant in \eqref{6.12}{\rm{)}}
\begin{equation} \lb{6.16a}
Ac^{-1} \sqrt{|R(x)|} \leq w_{\vec y}(x) \leq Bc^{-1}
\sqrt{|R(x)|^{-1}}
\end{equation}
and the residues $c_j$ of \eqref{6.13} obey
\begin{equation} \lb{6.16b}
0\leq c_j \leq Bc^{-1}
\end{equation}

\item[{\rm{(ii)}}] Uniformly in $\vy\in\bbG$, the constant $c$ in \eqref{6.12} is bounded and bounded away from zero.
Moreover, uniformly in $x\in\fre$ and $\vy\in\bbG$, \eqref{6.5a}
holds.
\end{SL}
\end{theorem}

\begin{proof} (i) \ Since $a$ has the form \eqref{6.12}, we have
\begin{equation} \lb{6.16c}
\sup_{x\in\fre}\, \abs{a(x)} \leq c(\beta_{\ell+1} -\alpha_1)^\ell
\end{equation}
which, given \eqref{6.9} and the relation
\begin{equation} \lb{6.16d}
w_{\vec y}(x)=\frac{1}{\pi} \Ima m(x+i0)
\end{equation}
implies the first inequality in \eqref{6.16a}.

Next, note that if we fix $x\in\fre$, the distance to the nearest
zero of $a$ is at least $\dist(x,\bbR\setminus\fre)$ and the
distance to all the other zeros of $a$ at least $\f12 \min_j
\abs{\beta_j-\alpha_j}$. Thus, for $x\in\fre$,
\begin{equation} \lb{6.16e}
\abs{a(x)} \geq c\bigl(\tfrac12 \min_j\,
\abs{\beta_j-\alpha_j}\bigr)^{\ell-1} \dist (x,\bbR\setminus\fre)
\end{equation}
On the other hand, for $x\in\fre$,
\begin{equation} \lb{6.16f}
|R(x)|\leq (\beta_{\ell+1}-\alpha_1)^{2\ell+1}
\dist(x,\bbR\setminus\fre)
\end{equation}
and
\begin{equation} \lb{6.16g}
|R(x)|\geq \bigl(\tfrac12 \min_j\,
\abs{\beta_j-\alpha_j}\bigr)^{2\ell+1} \dist(x, \bbR\setminus\fre)
\end{equation}
We get the second inequality in \eqref{6.16a} from
\eqref{6.16e}--\eqref{6.16g}, \eqref{6.9}, and \eqref{6.16d}.

By \eqref{6.9}, the residue $c_j$ in \eqref{6.13} is given by
\begin{align}
c_j &= 2c^{-1} \sqrt{\abs{R(y_j)}}\, \prod_{k\neq j} \abs{y_k-y_j}^{-1} \lb{6.16.h} \\
&\leq 2c^{-1} (\beta_{\ell+1}-\alpha_1)^{\ell+1} \bigl(\min_j
\abs{\beta_j-\alpha_j}\bigr)^{-\ell} \lb{6.16i}
\end{align}

\smallskip
(ii) \ By \eqref{6.16a} and \eqref{6.16b},
\begin{align} \lb{6.16j}
\notag Ac^{-1} \int_\fre \sqrt{|R(x)|}\, dx &\leq \int_\fre w(x)\,
dx + \sum_j c_j \\ &\leq Bc^{-1} \biggl[ \ell + \int_\fre
\sqrt{|R(x)|^{-1}}\, dx \biggr]
\end{align}
The total weight of the measure is $1$, so we get the claimed upper
and lower bounds on $c$. Given those, \eqref{6.16a} yields
\eqref{6.5a}.
\end{proof}

Given $\vy\in\bbG$, we denote by $m_{\vec y}$ the associated minimal
Herglotz function. The once-stripped $m$-function is also a minimal
Herglotz function and thus corresponds to some $\ve{w}\in\bbG$. We
define a map $U:\bbG\to\bbG$ by
\begin{equation} \lb{6.16x}
U(\vy)=\ve{w}
\end{equation}
so that $m_{U(\vy)}$ is the once-stripped $m$-function.

Now we can combine Theorem~\ref{T6.1} with the Abel map:
\begin{theorem}\lb{T6.2}
Suppose $\vy=(y_1,\ldots,y_\ell)\in\bbG$ and let $U$ be defined by
\eqref{6.16x}.
\begin{SL}
\item[{\rm{(i)}}] With $\ti\frA$ defined in \eqref{5.13a}--\eqref{5.14}, we have
\begin{equation} \lb{6.16}
\ti\frA (U(\vy))=\ti\frA(\vy)\frA(\infty)^{-1}
\end{equation}

\item[{\rm{(ii)}}] Let
\begin{equation} \lb{6.28a}
M_{\vec y}(z) = -m_{\vec y}(\x(z))
\end{equation}
Then we have that
\begin{equation} \lb{6.17}
M_{\vec y}(z) = \f{B(z)}{\ca(\fre)} \prod_{j=1}^\ell
\f{\Theta(z;U(\vy)_j)}{\Theta(z;y_j)}
\end{equation}
\end{SL}
\end{theorem}

\begin{remark} Since
\[
m(z) = -\f{1}{z} - \f{b_1}{z2} - \f{a_12 + b_12}{z3}+ O(z^{-4})
\]
\eqref{6.17} implies an explicit formula for $b_1$ and $a_1$ in
terms of theta functions.
\end{remark}

\begin{proof} $M_{\vec y}(z)$ is a meromorphic function with divisor
\begin{equation} \lb{6.19}
\delta_{\infty_+}-\delta_{\infty_-}+
\sum_{j=1}^\ell\bigl(\delta_{U(\vy)_j}-\delta_{y_j}\bigr)
\end{equation}
so \eqref{6.16} is just \eqref{5.24}.

By \eqref{5.25} (note that \eqref{5.22a} holds so there are no
$\Psi_j$ factors), we have \eqref{6.17} with $\ca(\fre)$ replaced by
a constant.

Since
\begin{equation}
m_{\vec y}(\x(z))=-\f{1}{\x(z)} + O(z2) \lb{6.20} = -\f{z}{x_\infty}
+ O(z2)
\end{equation}
and \eqref{4.31} holds, the constant is $\ca(\fre)$.
\end{proof}

\begin{corollary}\lb{C6.3} Under the map $\ti\frA$ from $\calT_\fre$ to $\Gamma^*$,
$\{U^n(\vy)\}_{n=0}^\infty$ is mapped to the ``equally spaced''
orbit $\{\ti\frA(\vy)\frA(\infty)^{-n}\}_{n=0}^\infty$ in
$\Gamma^*$. In particular,
\begin{equation} \lb{6.17x}
n\to m_{U^n(\vy)} \qquad n\to \{a_n, b_n\}
\end{equation}
are almost periodic sequences {\rm{(}}indeed, real analytic
quasiperiodic sequences{\rm{)}} with almost periods
$\{\rho_\fre([\alpha_1,\beta_j])\}_{j=1}^\ell$. These sequences are
periodic with period $p$ for one point in $\calT_\fre$ if and only
if they are for all points, and that holds if and only if
\eqref{4.42} holds.
\end{corollary}

\begin{remark}
By a \emph{quasiperiodic sequence}, $X_n$, we mean a sequence given
by
\begin{equation} \lb{6.18x}
X_n = x(e^{in\omega_1}, \dots, e^{in\omega_k})
\end{equation}
where $x$ is a continuous function on the $k$-torus
$(\partial\bbD)^k$. It is called \emph{real analytic} if $x$ is real
analytic. $(\omega_1, \dots,\omega_k)$ are called the \emph{almost
periods}.
\end{remark}

\begin{proof} \eqref{6.16} immediately implies that
\begin{equation} \lb{6.19x}
\ti\frA(U^n(\vy)) = \ti\frA(\vy) \frA(\infty)^{-n}
\end{equation}
so the orbit is as claimed.

Realize $\Gamma^*$ as $(\partial\bbD)^\ell$ by
\begin{equation} \lb{6.20x}
c \sim (c(\gamma_1), \dots, c(\gamma_\ell))
\end{equation}
Then, by \eqref{4.32},
\begin{equation} \lb{6.21}
\ti\frA(U^n(\vy))_j = \ti\frA(\vy)_j e^{-2\pi i \rho_\fre
([\alpha_1,\beta_j])n}
\end{equation}
which, given that $\ti\frA$ is real analytic and $\Theta(z;y)$ (and
so, $m,a_1,b_1$) are real analytic in the $y$'s, proves that the
sequences \eqref{6.17x} are almost periodic, indeed, real analytic
quasiperiodic.

For the final statement, note that, by \eqref{6.19x}, periodicity
for one or for all $\ti\frA(\vy)$ is equivalent to
$\frA(\infty)^p=1$. Now use Corollary~\ref{C4.5}.
\end{proof}

\section{Raw Jost Functions and the Jost Isomorphism} \lb{s7}

Recall that if $d\mu$ is a measure on $\bbR$ of the form \eqref{6.4}
so that off $[-2,2]$, $d\mu$ only has pure points,
$\{x_j\}_{j=1}^N$, ($N$ finite or infinite) with
\begin{equation} \lb{7.1a}
\sum_{j=1}^N \, (\abs{x_j}-2)^{1/2} <\infty
\end{equation}
and so that $w$ obeys a Szeg\H{o} condition,
\begin{equation} \lb{7.1b}
\int_{-2}^2 (4-x^2)^{-1/2} \log(w(x)) >-\infty
\end{equation}
one defines (see \cite{Jost1,KS,PY,OPUC2}) the Jost function, $u(z)$, on $\bbD$ by
\begin{equation} \lb{7.2}
u(z)=z B_\infty(z) \exp\biggl( \int
\f{z+e^{i\theta}}{z-e^{i\theta}}\, \log\biggl[ \f{\sin\theta}{\Ima
M(e^{i\theta})} \biggr] \, \f{d\theta}{4\pi}\biggr)
\end{equation}
where $M(e^{i\theta})$ is the boundary value of
\begin{equation} \lb{7.3}
M(z)= -m(z+z^{-1})
\end{equation}
(expressible in terms of $w$ via $\Ima M(e^{i\theta})=\pi w
(2\cos\theta)$) and $B_\infty$ is the Blaschke product of
$b(z,\zeta_j)$ with $\zeta_j\in\bbD$, $\zeta_j+\zeta_j^{-1}=x_j$ (by
\eqref{7.1a} and Lemma~\ref{L4.1}, this is a convergent Blaschke
product). In this section and the next, we will begin to discuss the
analog for elements of the isospectral torus, but in a way that
connects up to definitions that work in much greater generality and
will be the key to later papers in this series.

Surprisingly, our initial definition will be equal, up to a constant, to an object that is based on the
theta function formulae of the last section, and we will find there is a representation of the form \eqref{7.2}.
We regard this as one of the more interesting discoveries in the present paper.

Because there is only one natural choice for the reference measure
on $[-2,2]$, it is somewhat obscured that \eqref{7.2} involves not
only $d\mu$ but a reference measure which in \eqref{7.2} is
\begin{equation} \lb{7.5}
d\mu_0(x) = \f{1}{2\pi}\, \sqrt{4-x^2}\, \chi_{[-2,2]}(x)\, dx
\end{equation}
the measure of the free Jacobi matrix with $a_n\equiv 1$, $b_n\equiv
0$. After a change of variables via $x=2\cos\theta$ and scaling,
$\sqrt{4-x^2}$ turns into $\sin\theta$, which is where that factor
in \eqref{7.2} comes from.

When one shifts to  multiple gap situations, there is also a
reference measure needed. Our eventual choice will be to use a
particular point on the isospectral torus---and the next section
will explain change of reference measure. In this section, the
reference measure will be a different measure on the isospectral
torus, so we will call the resulting object the ``raw'' Jost
function.

As a bonus, we will also see that for any point on the isospectral
torus,
\begin{equation} \lb{7.6}
n\to \f{a_1 \cdots a_n}{\ca(\fre)^n}
\end{equation}
is an almost periodic sequence.

Let $\vy_0=(\beta_1,\beta_2, \dots, \beta_\ell)$ be the point in
$\bbG$ which serves as the base point for our $\Theta$'s. Given
$\vy=(y_1,\ldots,y_\ell)\in\bbG$, we let $\xi_j$ be the unique point
on $\ti C_j^+$ such that
\begin{equation} \lb{7.6x}
\x^\sharp(\xi_j)=y_j, \quad j=1,\ldots,\ell
\end{equation}
Moreover, we denote by $w_{\vec y}(x)$ the weight of the measure
$d\mu_{\vec y}$ associated to the $m$-function, $m_{\vec y}$, in the
isospectral torus.

For each $\vy\in\bbG$, we define a function on $\bbD$, the \emph{raw
Jost function}, by
\begin{equation} \lb{7.7}
\calR(z;\vy) = \prod_{\{j\,\mid\,\abs{\xi_j}<1\}} B(z,\xi_j) \exp
\biggl(\f{1}{4\pi} \int \f{e^{i\theta}+z}{e^{i\theta}-z} \, \log
\biggl[ \f{w_{\vec{y}_0}(\x(e^{i\theta}))}{w_{\vec y}
(\x(e^{i\theta}))}\biggr]\, d\theta\biggr)
\end{equation}
Here we use the estimate \eqref{6.5a} to be sure that
\begin{equation} \lb{7.8}
\int \log (w_{\vec y} (\x(e^{i\theta}))\, \f{d\theta}{2\pi} >-\infty
\end{equation}
for all $\vy\in\bbG$.

We will call the $\exp(\,\cdots)$ factor in \eqref{7.7}, the
\emph{Szeg\H{o} part}, and the first factor, the \emph{Blaschke
part}. It is easy to see that each is continuous in $\vy$, so
$\calR(z;\vy)$ is also continuous in $\vy$. The Blaschke factor is
only piecewise $C^1$ and not $C^1$ because whenever $\xi_j$ moves
from inside $\bbD$ to outside, a factor disappears. We have
$B(0,\xi_j)=\abs{\xi_j}$ and
\begin{equation} \lb{7.9}
\xi_j \to \begin{cases}
\abs{\xi_j} & \text{if } \abs{\xi_j} \leq 1 \\
1 & \text{if } \abs{\xi_j} \geq 1
\end{cases}
\end{equation}
has a discontinuous derivative as $\abs{\xi_j}$ passes from below
$1$ to above. Nevertheless, we will see below that for many cases
(where the harmonic measures
$\{\rho_\fre([\alpha_1,\beta_j])\}_{j=1}^\ell$ obey a Diophantine
condition), $\vy\mapsto\calR(z;\vy)$ is real analytic. It is an
interesting open question if this is always true!

We need one more piece of notation. Each $\vy\in\bbG$ determines a
unique $m_{\vec y}$ and, thereby, a unique two-sided Jacobi matrix,
$\ti J_{\vec y}$. Its Jacobi parameters will be denoted $\{a_n(\vy),
b_n(\vy)\}_{n=-\infty}^\infty$. Here is the main theorem of this
section:

\begin{theorem}\lb{T7.1} There exists a continuous, everywhere strictly positive function, $\varphi$, on $\bbG$
so that
\begin{equation} \lb{7.10}
\calR(z;\vy) = \varphi(\vy) \prod_{j=1}^\ell \Theta (z;y_j)
\end{equation}
Moreover, $\varphi$ obeys
\begin{gather}
\varphi(\vy_0) =1 \lb{7.11} \\
\f{a_1(\vy)}{\ca(\fre)} = \f{\varphi(U(\vy))}{\varphi(\vy)}
\lb{7.12}
\end{gather}
\end{theorem}

\begin{remarks} 1. In \eqref{7.12}, $U$ is the map from \eqref{6.16x}.

\smallskip
2. If the harmonic measures
$\{\rho_\fre([\alpha_1,\beta_j])\}_{j=1}^\ell$ are rationally
independent, the orbit $\{U^n(\vy_0)\}_{n=0}^\infty$ is dense in
$\Gamma^*$ and \eqref{7.11}--\eqref{7.12} determine $\varphi$
uniquely. In general, $\varphi$ is continuous in
$\{\alpha_j,\beta_j\}_{j=1}^\ell$, so this, in principle, determines
it uniquely.

\smallskip
3. It is useful to define, for $\vy=(y_1,\ldots,y_\ell)\in\bbG$,
\begin{equation} \lb{7.13}
\ti\Theta(z;\vy) = \prod_{j=1}^\ell \Theta(z;y_j)
\end{equation}
\end{remarks}

We want to note some interesting corollaries before we turn to the proof:

\begin{corollary}\lb{C7.2} For each $\vy\in\bbG$, $\calR(z;\vy)$ has a meromorphic continuation to $\bbC\cup
\{\infty\}\setminus\calL$ with poles, all simple, only at
$\{\gamma(\zeta_j)\}_{\gamma\in\Gamma,\, j=1, \dots,\ell}$, and
zeros, all simple, at $\{\gamma(\xi_j)\}_{\gamma\in\Gamma,\, j=1,
\dots, \ell}$, where $\zeta_j$ and $\xi_j$ are given by \eqref{5.1}
and \eqref{7.6x}, respectively.
\end{corollary}

\begin{remark} Thus, the Szeg\H{o} part cancels the poles of
$B(z,\xi_j)$ at $\{\gamma(\bar{\xi}_j^{-1})\}_{\gamma\in\Gamma}$
for $j$ with $\abs{\xi_j}<1$.

\end{remark}

\begin{proof} Obvious from \eqref{7.10}.
\end{proof}

Define the \emph{raw Jost isomorphism} from $\bbG$ to $\Gamma^*$ by
\begin{equation} \lb{7.14}
\calJ_r(\vy) = \ti\frA(\vy)
\end{equation}
By Theorem~\ref{T5.3}, it is a real analytic homeomorphism. The
following corollary of Theorem~\ref{T7.1} is so important, we call
it a theorem.

\begin{theorem}\lb{T7.3} For each $\vy\in\bbG$, $\calR(\,\cdot\,;\vy)$ is a character automorphic function and its character
is $\calJ_r(\vy)$. $\calJ_r$ is a real analytic bijection between
the isospectral torus and $\Gamma^*$.
\end{theorem}

\begin{remark} The Blaschke part of $\calR$ is character automorphic, and we will show directly in
Lemma~\ref{L8.1} that the Szeg\H{o} part is, too.  What is difficult
without \eqref{7.10} is that the map from $\vy$ to the character of
$\calR(\dott;\vy)$ is a bijection. That this map is a bijection will
be critical in our proof of Szeg\H{o} asymptotics in \cite{CSZ2}.
\end{remark}

\begin{proof} Immediate from \eqref{7.10} and the definition of $\ti\frA$.
\end{proof}

\begin{corollary}\lb{C7.4} For each $\vy\in\bbG$, the sequence
\begin{equation} \lb{7.15}
n\to \f{a_1(\vy)\cdots a_n(\vy)}{\ca(\fre)^n}
\end{equation}
is bounded and almost periodic, also bounded away from $0$. The
upper and lower bounds are bounded uniformly in $\vy$.
\end{corollary}

\begin{remark} That $(a_1 \cdots a_n)^{1/n}\to \ca(\fre)$ is a general fact about regular measures, and
each $d\mu_{\vec y}$ is regular by a theorem of Widom \cite{Wid} and
Van Assche \cite{vA} (see \cite{StT,EqMC}). This more subtle result
is a special case of a theorem of Widom \cite{Widom} (see also
\cite{Apt}).
\end{remark}

\begin{proof} By \eqref{7.12},
\begin{equation} \lb{7.15a}
\f{a_1(\vy) \cdots a_n(\vy)}{\ca(\fre)^n} =
\f{\varphi(U^n(\vy))}{\varphi(\vy)}
\end{equation}
is given by the values of a continuous function (namely,
$\varphi\circ\ti\frA^{-1}$) on $\Gamma^*$ along the orbit
$\ti\frA(\vy)\frA(\infty)^{-n} $. This function is bounded and
bounded away from $0$.
\end{proof}

One key to the proof of Theorem~\ref{T7.1} is a nonlocal
step-by-step sum rule. Such sum rules began with Killip--Simon
\cite{KS} for $[-2,2]$, formalized by Simon \cite{S288}, and one
version was found for periodic Jacobi matrices by
Damanik--Killip--Simon \cite{DKSppt}. The extension of the sum rules
to the covering map context, which relies on Beardon's theorem, is a
major theme in this paper and especially in the second paper
\cite{CSZ2} in this series. Here is the version for the isospectral
torus:

\begin{theorem} \lb{T7.5} Let $m_{\vec y}$ be the $m$-function for a point $\vy\in\bbG$.
Define for $z\in\bbD$,
\begin{equation} \lb{7.16}
M_{\vec y}(z)=-m_{\vec y}(\x(z))
\end{equation}
Then for all $z\in\bbD$,
\begin{equation} \lb{7.17}
a_1(\vy) M_{\vec y}(z) = B(z)\f{\calR(z;U(\vy))}{\calR(z;\vy)}
\end{equation}
\end{theorem}

\begin{proof} Let $B^{(\vy)}(z)$ be the Blaschke part of $\calR$, that is, the Blaschke product in \eqref{7.7}.
Define
\begin{equation} \lb{7.18}
B_\infty^{(\vy)} (z)=B(z)\f{B^{(U(\vy))}(z)}{B^{(\vy)}(z)}
\end{equation}
Then $B^{(\vy)}_\infty(z)$ has zeros and poles precisely at all the
zeros and poles of $M_{\vec y}(z)$ inside $\bbD$, so
\begin{equation} \lb{7.19}
h(z) = \f{a_1(\vy) M_{\vec y}(z)}{B_\infty^{(\vy)}(z)}
\end{equation}
is analytic and nonvanishing in $\bbD$.

The same argument that led to \eqref{4.72} shows that
\begin{equation} \lb{7.20}
z\in\calF \Rightarrow \abs{\arg(B_\infty^{(\vy)}(z))}\leq C
\end{equation}
for some constant $C$. Moreover,
\begin{equation} \lb{7.21}
z\in\calF \Rightarrow \abs{\arg(M_{\vec y}(z))}\leq \pi
\end{equation}
so
\begin{equation} \lb{7.22}
z\in\calF \Rightarrow \abs{\arg(h(z))} \leq C_1
\end{equation}
and, in particular, $\arg(h)$ varies by at most $2C_1$ over $\calF$.

$B^{(\vy)}_\infty$ is character automorphic and $M_{\vec y}$ is
automorphic, so $h(z)$ is character automorphic. Thus, $\arg(h(z))$
varies by at most $2C_1$ over any $\gamma(\calF)$, which means that
\begin{equation} \lb{7.23}
z\in\bbD_k \Rightarrow \abs{\arg(h(z))} \leq (2k+1) C_1
\end{equation}
Thus, by \eqref{2.11}, for any $r$,
\begin{equation} \lb{7.24}
\{\theta\mid \abs{\Ima \log(h(re^{i\theta}))} \geq (2k+1) C_1\}
\subset\partial\calR_k
\end{equation}
So, by Beardon's theorem in the form \eqref{3.15}, for any $p<\infty$,
\begin{equation} \lb{7.25x}
\sup_r \int  \abs{\Ima \log(h(re^{i\theta}))}^p\, \f{d\theta}{2\pi} < \infty
\end{equation}
By M.~Riesz's theorem (see Rudin \cite{Rudin}),
\begin{equation} \lb{7.25}
\log(h)\in\bigcap_{p<\infty} H^p (\bbD)
\end{equation}
so by the standard representation for $H^p$ functions, $p\geq 1$
(see, e.g., \cite{Rudin}), we get
\begin{equation} \lb{7.26}
a_1(\vy) M_{\vec y}(z) = B_\infty^{(\vy)}(z) \exp \biggl(
\f{1}{2\pi} \int \f{e^{i\theta}+z}{e^{i\theta}-z}\, \log
\bigl(a_1(\vy)\abs{M_{\vec y}(e^{i\theta})}\bigr)\, d\theta \biggr)
\end{equation}
where we used that for a.e.\ $\theta$,
$\abs{B_\infty^{(\vy)}(e^{i\theta})}=1$, so
\begin{equation} \lb{7.27}
\log \abs{h(e^{i\theta})} = \log \abs{M_{\vec y}(e^{i\theta})}
\end{equation}

Taking boundary values in
\begin{equation} \lb{7.28}
M_{\vec y}(z)^{-1} = \x(z) -b_1 - a_1^2 M_{U(\vy)}(z)
\end{equation}
we see that
\begin{equation} \lb{7.29}
\f{\Ima M_{\vec y}(e^{i\theta})}{\abs{M_{\vec y}(e^{i\theta})}^2} =
a_1^2 \Ima M_{U(\vy)} (e^{i\theta})
\end{equation}
or
\begin{align}
\log(a_1(\vy)\abs{M_{\vec y}(e^{i\theta})}) &= \f12 \log\biggl[
\f{\Ima M_{\vec y}(e^{i\theta})}
{\Ima M_{U(\vy)}(e^{i\theta})}\biggr] \lb{7.30} \\
&= \f12 \log \biggl[ \f{\Ima M_{\vec y}(e^{i\theta})}{\Ima
M_{\vec{y}_0} (e^{i\theta})}\biggr] - \f12 \log \biggl[ \f{\Ima
M_{U(\vy)}(e^{i\theta})}{\Ima M_{\vec{y}_0} (e^{i\theta})}\biggr]
\lb{7.31}
\end{align}
Since $\Ima M_{\vec y}(e^{i\theta})= \pi w_{\vec y} (\x
(e^{i\theta}))$, \eqref{7.26} plus \eqref{7.31} and the definition
\eqref{7.2} implies \eqref{7.17}.
\end{proof}

\begin{corollary}\lb{C7.6} For any $\vy\in\bbG$, we have
\begin{equation} \lb{7.32}
\f{\calR(z;U(\vy))}{\ti\Theta(z;U(\vy))} = \f{a_1(\vy)}{\ca(\fre)}
\, \f{\calR(z;\vy)}{\ti\Theta(z;\vy)}
\end{equation}
\end{corollary}

\begin{proof} Immediate from \eqref{6.17} and \eqref{7.17}.
\end{proof}

\begin{proof}[Proof of Theorem~\ref{T7.1}] For $\vy_0$, we have
\begin{equation} \lb{7.33}
\ti\Theta(z;\vy_0) = \calR(z;\vy_0) =1
\end{equation}
Thus, by \eqref{7.32}, we have \eqref{7.10} for $\vy=U^n(\vy_0)$
with
\[
\varphi(U^n (\vy_0)) = \f{a_1 (\vy_0) \cdots a_n
(\vy_0)}{\ca(\fre)^n}
\]
since
\begin{equation} \lb{7.34}
a_1 (U^n(\vy_0)) = a_{n+1} (\vy_0)
\end{equation}
Thus, \eqref{7.12} also holds for $\vy=U^n (\vy_0)$.

Suppose now $\{\alpha_j,\beta_j\}_{j=1}^{\ell+1}$ are such that
$\{\rho_\fre ([\alpha_1, \beta_j])\}_{j=1}^\ell$ are rationally
independent. Then $S_0\equiv \{U^n (\vy_0)\}$ is dense in the
isospectral torus. Define
\begin{equation} \lb{7.35}
\varphi(\vy) \equiv \f{\calR(z=0;\vy)}{\Theta(z=0;\vy)}
\end{equation}
Since $\calR$ is continuous in $\vy$, so is $\varphi$. As
\eqref{7.10} holds on the dense set $S_0$ and both sides are
continuous, \eqref{7.10} holds for all $\vy\in\bbG$. Similarly, both
sides of \eqref{7.12} are continuous and as \eqref{7.12} holds on
$S_0$, it holds in general.

To handle general $\{\alpha_j,\beta_j\}_{j=1}^{\ell+1}$, we just
repeat the continuity argument ``at a higher level.'' Fix $\alpha_1,
\dots, \alpha_{\ell+1}$ and $\beta_{\ell+1}$ and vary $(\beta_1,
\dots, \beta_\ell)$. By a theorem of Totik \cite{Totik01}, the map
$(\beta_1, \dots, \beta_\ell)\to \{\rho_\fre ([\alpha_j,
\beta_j])\}_{j=1}^\ell$ is a $C^\infty$ local bijection so the set
of $\beta$'s with rationally independent $\{\rho_\fre ([\alpha_j,
\beta_j])\}_{j=1}^\ell$ is dense. By conveniently labelling $\bbG$
(say, measure angles in  $\ti C_j^+$ about its center starting at
the point closest to $0$) in a way that is independent of
$\{\alpha_j,\beta_j\}_{j=1}^{\ell+1}$, all objects, that is,
$\calR(z;\vy)$ and $\Theta(z;\vy)$, are continuous in
$(\beta_1,\ldots,\beta_\ell)$ according to Theorem \ref{T2.1}.
Therefore, by repeating the above argument, we get \eqref{7.10} and
\eqref{7.12} by continuity.
\end{proof}

Finally, we want to note that sometimes $\varphi$ and so $\calR$ are
real analytic in $\vy$. We say that $\fre$ obeys a \emph{Diophantine
condition} if there are a constant, $C$, and an integer, $k$, so
that for $(n_1, \dots, n_\ell)\neq (0,\dots, 0)$ in $\bbZ^\ell$,
\begin{equation} \lb{7.36}
\biggl| \, \sum_{j=1}^\ell n_j \rho_\fre ([\alpha_1, \beta_j])\biggr| \geq C(1+\abs{n})^{-k}
\end{equation}
As is well known, given Totik's theorem quoted above, Lebesgue a.e.\ $\{\alpha_j,\beta_j\}_{j=1}^{\ell+1}$ lead
to a Diophantine $\fre$.

\begin{theorem}\lb{T7.7} If $\fre$ is Diophantine, the function $\varphi$ of \eqref{7.10} is real analytic in
$\vy$ and thus, $\calR(z;\vy)$ is real analytic in $\vy$.
\end{theorem}

\begin{proof} $\Theta$ is real analytic in $y$, so by \eqref{7.10}, the statement about $\varphi$ implies that for
$\calR$.

Let
\begin{equation} \lb{7.37}
L(\vy) =\log \biggl( \f{a_1(\vy)}{\ca(\fre)}\biggr)
\end{equation}
Then
\begin{equation} \lb{7.38}
S(\vy)\equiv \log(\varphi(\vy))
\end{equation}
obeys
\begin{equation} \lb{7.39}
S(U(\vy)) - S(\vy) =L(\vy)
\end{equation}
Since $L$ is real analytic on the torus, its Fourier coefficients,
$l_{(n_1, \dots, n_\ell)}$, obey
\begin{equation} \lb{7.40}
\abs{l_{(n_1, \dots, n_\ell)}} \leq Ce^{-D\abs{n}}
\end{equation}
for some $C,D >0$. By \eqref{7.39}, the Fourier coefficients
$s_{(n_1, \dots, n_\ell)}$ for $S$ obey
\begin{equation} \lb{7.41}
(e^{in\cdot\omega} -1) s_{(n_1, \dots, n_\ell)} = l_{(n_1, \dots,
n_\ell)}
\end{equation}
where
\begin{equation} \lb{7.42}
n\cdot \omega =\sum_{j=1}^\ell n_j \rho_\fre ([\alpha_1, \beta_j])
\end{equation}
\eqref{7.41} implies $\ell_0=0$ and \eqref{7.40}--\eqref{7.41}
together with \eqref{7.36} imply that for any $\veps >0$,
\begin{equation} \lb{7.43}
\abs{s_{(n_1, \dots, n_\ell)}} \leq C_\veps e^{-(D-\veps)\abs{n}}
\end{equation}
This implies that $S(\vy)$ and so $\varphi(\vy)= e^{S(\vy)}$ are
real analytic.
\end{proof}

\section{Change of Reference Measure in Jost Functions} \lb{s8}

As is well known, OPRL obey a difference equation, \eqref{1.6}. We
will also be interested in other solutions. Since we have labelled
the Jacobi parameters starting at $n=1$, it will be useful to label
solutions that way too, that is, to look for solutions of
\begin{equation} \lb{8.1}
a_n u_{n+1} + (b_n-z) u_n + a_{n-1} u_{n-1} =0
\end{equation}
for $n=1,2, \dots$ where $a_0$ is often picked to be $1$, but in the
isospectral torus will be the natural two-sided almost periodic
$a_0$. Note that
\begin{equation} \lb{8.2}
u_n = p_{n-1}(z)
\end{equation}
$n=1,2, \dots$ is a solution of \eqref{8.1} with $p_{-1}=0$ (so $a_0$ is unimportant).

For any bounded Jacobi matrix, $J$, and $z\in\bbC_+$, there is a
unique solution of \eqref{8.1} which is $\ell^2$ at infinity, unique
up to an overall constant (see, e.g., \cite{Rice}). One natural
choice is the Weyl solution,
\begin{equation} \lb{8.3}
W_n(z) = \jap{\delta_n, (z-J)^{-1} \delta_1}
\end{equation}
It obeys \eqref{8.1} for $n=2,3,\dots$. At $n=1$, one has
\begin{equation} \lb{8.4}
a_1 W_2 + (b_1-z) W_1= -1
\end{equation}
(since $(J-z)(J-z)^{-1}\delta_1 =\delta_1$), and so \eqref{8.1} holds at $n=1$ if we define
\begin{equation} \lb{8.5}
W_0 = a_0^{-1}
\end{equation}

One defect of this solution is that, of course, $W_n(z)$ has a pole
at each discrete eigenvalue of $J$. To get around this, Jost had the
idea (for continuum Schr\"odinger operators) of multiplying $W_n(z)$
by a $z$-dependent constant $u_0(z)$ with a zero at the eigenvalues.
Then $u_0(z)W_n(z)$ can have a removable singularity at such $z$'s.
Our raw Jost functions have zeros at the $\xi_j$'s, so when we look
at solutions in the next two sections, the poles in the Weyl
solutions will be cancelled by these zeros. But we are going to want
to continue these solutions up to the bands also, and our raw Jost
functions have poles at the $\zeta_j$'s. These poles are not
intrinsic to the $y$'s, but come from the $w_{\vec{y}_0}$ term in
\eqref{7.7}. We thus want to consider modifying that.

\begin{definition} By the \emph{Szeg\H{o} class} for $\fre$, we mean the set of measures $d\mu$ of the form \eqref{6.4} where $w$
is supported on $\fre$ and obeys the Szeg\H{o} condition
\eqref{4.45} (with $f=w$), and so that outside $\fre$, $d\mu_\s$
only has countably many pure points $\{x_j\}_{j=1}^N$ obeying a
Blaschke condition of the form \eqref{4.64}.
\end{definition}

We are now prepared to define the Jost functions with reference
measure $d\mu_1$.

\begin{definition} Suppose $d\mu, d\mu_1$ are two measures in the Szeg\H{o} class for $\fre$. The \emph{Jost function},
$\J(z;\mu,\mu_1)$, with reference measure $d\mu_1$, is the
meromorphic function on $\bbD$ defined by
\begin{equation} \lb{8.6}
\J(z;\mu,\mu_1) = \f{\prod_{j=1}^N B(z,z_j)}{\prod_{j=1}^{N_1}
B(z,z_j^{(1)})} \, \exp \biggl( \f{1}{4\pi} \int
\f{e^{i\theta}+z}{e^{i\theta}-z} \, \log \biggl[
\f{w_1(\x(e^{i\theta}))}{w(\x(e^{i\theta}))}\biggr]\, d\theta\biggr)
\end{equation}
where $\{z_j\}_{j=1}^N$ are the points in $\calF$ with $\x(z_j)=x_j$
(and similarly for $z_j^{(1)}$ and $x_j^{(1)}$).
\end{definition}

Notice that, by Proposition~\ref{P4.8} and condition \eqref{4.64},
the Blaschke products converge. By Corollary~\ref{C4.6} and
\eqref{4.45}, the integral in \eqref{8.6} converges. Notice that if
$d\mu_1$ has pure points in the gaps, then $\J$ has poles in $\bbD$.
For this reason, we will normally consider only $d\mu_1$'s with no
such pure points, but since we want to consider the entire
isospectral torus later in this section, we allow for the
possibility.

We are heading towards proving that $\J$ is character automorphic. The key is

\begin{lemma} \lb{L8.1} Suppose $f$ is a real-valued function on $\partial\bbD$ so that
\begin{equation} \lb{8.7}
\int \abs{f(e^{i\theta})}\, \f{d\theta}{2\pi} <\infty
\end{equation}
and for all $\gamma\in\Gamma$,
\begin{equation} \lb{8.8}
f(\gamma(e^{i\theta})) = f(e^{i\theta})
\end{equation}
Define for $z\in\bbD$,
\begin{equation} \lb{8.9}
S_f(z) =\exp\biggl( \int \f{e^{i\theta}+z}{e^{i\theta}-z}\, f(e^{i\theta})\, \f{d\theta}{2\pi}\biggr)
\end{equation}
Then $S_f$ is character automorphic.
\end{lemma}

\begin{remark} If
\begin{equation} \lb{8.10}
f(e^{i\theta}) = \log \biggl[ \f{\Ima
w_{\vec{y}_0}(\x(e^{i\theta}))}{\Ima w_{\vec
y}(\x(e^{i\theta}))}\biggr]
\end{equation}
where $\vy$ has no points in $\bbD$, then $S_f$ is a raw Jost
function which, if $\vy\neq \vy_0$, is not automorphic but only
character automorphic (by Theorem~\ref{T7.3}). Thus, $S_f$ may have
a nontrivial character.
\end{remark}

\begin{proof} Suppose first there is $a>0$ so that
\begin{equation} \lb{8.11}
-a\leq f(e^{i\theta})\leq a
\end{equation}
for all $\theta$. Then, since
\begin{equation} \lb{8.12}
\int \Real \biggl( \f{ e^{i\theta}+z}{e^{i\theta}-z}\biggr)\, \f{d\theta}{2\pi} =1
\end{equation}
with positive integrand, we have
\begin{equation} \lb{8.13}
e^{-a} \leq \abs{S_f(z)} \leq e^a
\end{equation}
In particular, if $\gamma\in\Gamma$, then
\begin{equation} \lb{8.14}
h(z) \equiv \f{S_f(\gamma(z))}{S_f(z)}
\end{equation}
is analytic and
\begin{equation} \lb{8.15}
e^{-2a} \leq \abs{h(z)} \leq e^{2a}
\end{equation}
so one can define $\log(h)$ on $\bbD$ and it belongs to
$H^\infty(\bbD)$.

By \eqref{8.9}, for Lebesgue a.e.\ $\theta$,
\begin{equation} \lb{8.16}
\lim_{r\uparrow 1}\, \abs{S_f(re^{i\theta})} =e^{f(e^{i\theta})}
\end{equation}
Since $\gamma$ maps $\bbD$ to $\bbD$ and boundary values are nontangential limits, for Lebesgue a.e.\ $\theta$,
\begin{equation} \lb{8.17}
\lim_{r\uparrow 1}\, \abs{S_f(\gamma(re^{i\theta}))} = e^{f(\gamma(e^{i\theta}))} = e^{f(e^{i\theta})}
\end{equation}
by the hypothesis \eqref{8.8}.

It follows that for a.e.\ $\theta$,
\begin{equation} \lb{8.18}
\lim_{r\uparrow 1}\, \Real (\log(h(re^{i\theta}))) = \lim_{r\uparrow 1}\, \log\abs{h(re^{i\theta})} =0
\end{equation}
Since $\log\abs{h}$ is a bounded harmonic function, $\Real \log
(h(z))=0$, so $\log(h(z))=i\psi_\gamma$ for some real $\psi_\gamma$,
that is,
\begin{equation} \lb{8.19}
S_f(\gamma(e^{i\theta})) = e^{i\psi_\gamma} S_f(e^{i\theta})
\end{equation}
As usual, this implies that $\gamma\to e^{i\psi_\gamma}$ is a
character, and so $S_f$ is character automorphic.

If $f$ does not obey \eqref{8.11}, it is easy to write it as an $L^1$ limit of functions that do. Thus, $S_f$
is a uniform (on compact subsets of $\bbD$) limit of character automorphic functions. By the compactness of
$\Gamma^*$, it is easy to see that any such limit is character automorphic.
\end{proof}

\begin{theorem}\lb{T8.2} For any $d\mu, d\mu_1$ in the Szeg\H{o} class, $\J(z;\mu,\mu_1)$ is a character
automorphic function.
\end{theorem}

\begin{remark} For $d\mu, d\mu_1$ both in the isospectral torus, this follows from Theorem~\ref{T7.3} and
Theorem~\ref{T8.3} below.
\end{remark}

\begin{proof} Immediate from \eqref{8.6}, \eqref{4.15}, and Lemma~\ref{L8.1}.
\end{proof}

\begin{theorem}[Change of reference measure in $\J$] \lb{T8.3} Let $d\mu, d\mu_1, d\mu_2$ be three measures
in the Szeg\H{o} class. Then for all $z\in\bbD$,
\begin{equation} \lb{8.20}
\J(z;\mu,\mu_1) = \f{\J(z;\mu,\mu_2)}{\J(z;\mu_1,\mu_2)}
\end{equation}
In particular, for $\vy, \vy_1\in\bbG$, we have
\begin{equation} \lb{8.21}
\J(z;\mu_{\vec y},\mu_{\vec{y}_1}) =
\f{\calR(z;\vy)}{\calR(z;\vy_1)}
\end{equation}
\end{theorem}

\begin{remark} By ``all $z\in\bbD$,'' we either mean except for the discrete set of poles and zeros or else
in the sense of meromorphic functions.
\end{remark}

\begin{proof} In \eqref{8.6}, the Blaschke products and the $\log[w_1]$, $\log[w]$ factors can be separated
out and cancelled and recombined.
\end{proof}

\begin{corollary}\lb{C8.4} For any Szeg\H{o} class measure, $d\mu_1$, the character $\calJ_{\mu_1}(\vy)$ of
$\J(z;\mu_{\vec y},\mu_1)$ defines a real analytic bijection of
$\bbG$ and $\Gamma^*$.
\end{corollary}

\begin{proof} By \eqref{8.20},
\begin{equation} \lb{8.22}
\calJ_{\mu_1}(\vy) = \calJ_r(\vy) \calJ_{\mu_1}(\vy_0)
\end{equation}
Since $\vy\mapsto\calJ_r(\vy)$ is a bijection of $\bbG$ and
$\Gamma^*$ (by Theorem~\ref{T7.3}) and $\chi\mapsto \chi
\calJ_{\mu_1}(\vy_1)$ is a bijection of $\Gamma^*$, $\vy\mapsto
\calJ_{\mu_1}(\vy)$ is a bijection.
\end{proof}

\begin{corollary}\lb{C8.5} For any $\vy,\vy_1\in\bbG$,
\begin{equation} \lb{8.23}
\J(z;\mu_{\vec y},\mu_{\vec{y}_1}) = \varphi(\vy)
\varphi(\vy_1)^{-1}\,\f{ \ti\Theta(z;\vy)}{\ti\Theta(z;\vy_1)}
\end{equation}
where $\varphi$ is the function of Theorem~\ref{T7.1} and
$\ti\Theta$ is given by \eqref{7.13}.
\end{corollary}

\begin{proof} Immediate from \eqref{7.10} and \eqref{8.21}.
\end{proof}

Since we want to make the poles of $\J$ as far from $\calS_+$ as
possible, we define $\ve{w}$ to be the point on $\bbG$ whose
coordinates $(w_1, \dots, w_\ell)$ have points $\ti\zeta_1, \dots,
\ti\zeta_\ell$ in $\ti C_j^+$ with $\x^\sharp(\ti\zeta_j)=w_j$ and
\begin{equation} \lb{8.24}
\abs{\ti\zeta_j} = \max_{\zeta\in\ti C_j^+}\, \abs{\zeta}
\end{equation}

\begin{definition} Let $d\nu$ be the measure in $\calT_\fre$ associated to $\ve{w}$. For any $d\mu$ in the
Szeg\H{o} class, we define the \emph{Jost function}, $u(z;\mu)$, by
\begin{equation} \lb{8.25}
u(z;\mu) = \J(z;\mu,\nu)
\end{equation}
For $\vy\in\bbG$, we use $u(z;\vy)$ for $u(z;\mu_{\vec y})$.
\end{definition}

$u(z;\mu)$ will play a major role in the later papers of this
series. $u(z;\vy)$ will concern us in the rest of this paper. We
begin by noting

\begin{theorem}\lb{T8.6} There is a neighborhood, $N$\!, of $\ol\calF$ {\rm{(}}closure and neighborhood in
$\bbC${\rm{)}} so that each $u(\,\cdot\,;\vy)$ is analytic in $N$
and $u$ is uniformly bounded on $N$ and in $\vy\in\bbG$.
\end{theorem}

\begin{proof} Obvious from \eqref{8.23}, which says that
\begin{equation} \lb{8.26}
u(z;\vy) = \varphi(\vy) \varphi(\ve{w})^{-1}\,
\f{\ti\Theta(z;\vy)}{\ti\Theta(z;\ve{w})}
\end{equation}
and the fact that the ratio of $\ti\Theta$'s has poles only at
$\{\gamma(\ti\zeta_j)\}_{\gamma\in\Gamma,\, j=1,\dots,\ell}$.
\end{proof}

\begin{proposition}\lb{P8.7} For $z\in\bbR$, $u(z;\vy)>0$.
\end{proposition}

\begin{proof} Follows from \eqref{8.26} and the facts that $\varphi$
is positive and that $\Theta$ is positive on $\bbR$.
\end{proof}

\section{Jost Solutions} \lb{s9}

One big benefit of the covering map formalism is that it provides
explicit information about solutions of \eqref{8.1} for $J$ in the
isospectral torus and, thereby, of ground states, spectral
theorist's Green's function, etc. We begin by moving the Weyl
solutions, \eqref{8.3}, to $\bbD$:

\begin{definition} For $z\in\bbD$, the \emph{Weyl solution} is defined by
\begin{equation} \lb{9.1}
W_n(z) =\jap{\delta_n, (\x(z)-J)^{-1}\delta_1}
\end{equation}
\end{definition}

For the case $\fre=[-2,2]$, this function is studied in Section~13.9
of \cite{OPUC2}. The proof of Proposition 13.9.3 of \cite{OPUC2} is
purely algebraic and immediately extends to our context:

\begin{theorem}\lb{T9.1} Suppose $J$ is the Jacobi matrix of any OPRL with
$\sigma_\ess(J)=\fre$ and define $M$ by
\begin{equation}
M(z)=-m(\x(z))
\end{equation}
for $z\in\bbD$. Let $J^{(n)}$ be the $n$-times stripped Jacobi
matrix, that is,
\begin{equation} \lb{9.2}
a_j^{(n)} =a_{n+j} \qquad b_j^{(n)} =b_{n+j}
\end{equation}
and denote by $M^{(n)}$ its $m$-function on $\bbD$. Then
\begin{equation} \lb{9.3}
W_n(z) =M(z) (a_1 M^{(1)}(z))\cdots (a_{n-1}M^{(n-1)}(z))
\end{equation}
\end{theorem}

\begin{definition} Suppose $d\mu$ lie in the Szeg\H{o} class. Then the \emph{Jost solution} is defined for $z\in\bbD$ by
\begin{equation} \lb{9.4}
u_n(z) =u(z;\mu) W_n(z)
\end{equation}

We focus here on the case $d\mu=d\mu_{\vec y}$ for $\vy\in\bbG$, in
which case we use the notation $u_n(z;\vy)$.
\end{definition}

\begin{theorem}\lb{T9.2} For $n\geq 1$, we have
\begin{equation} \lb{9.5}
u_n(z;\vy) = a_n^{-1} B(z)^n u(z;U^n(\vy))
\end{equation}
If \eqref{9.5} is used to define $u_n$ for all $n\in\bbZ$, then
\begin{SL}
\item[{\rm{(i)}}] $u_n$ solves \eqref{8.1}.
\item[{\rm{(ii)}}] There is a neighborhood, $N$\!, of $\ol\calF$ so that $u_n(z;\vy)$ has an analytic
continuation to $N$ and is real analytic in $\vy$.
\item[{\rm{(iii)}}] $B(z)^{-n} u_n(z;\vy)$ is almost periodic in $n$. Indeed, uniformly for $z\in N$
and $\vy\in\bbG$, it is real analytic quasiperiodic.
\end{SL}
\end{theorem}

\begin{proof} By \eqref{7.17},
\begin{equation} \lb{9.6}
a_{j+1} M^{(j)}(z) = B(z)\, \f{u(z;U^{j+1}(\vy))}{u(z;U^j(\vy))}
\end{equation}
Thus, by \eqref{9.3},
\begin{align}
a_n W_n(z) &= \prod_{j=0}^{n-1} a_{j+1} M^{(j)}(z) \notag \\
&= B(z)^n \, \f{u(z;U^n(\vy))}{u(z;\vy)} \lb{9.7}
\end{align}
which is \eqref{9.5}.

By \eqref{9.5}, we have
\begin{equation} \lb{9.8}
a_{n+j}u_{n+j}(z;\vy) = a_nB(z)^j u_n(z;U^j(\vy))
\end{equation}
for all $n,j\in\bbZ$. Since $W_n$ solves \eqref{8.1} for $n\geq 1$,
$u_n$ does also, and then by \eqref{9.8}, $u_n$ solves \eqref{8.1}
for all $n\in\bbZ$. (ii) is immediate from Theorem~\ref{T8.6}. (iii)
is then immediate from \eqref{9.5} and the fact that under
$\ti\frA$, $U$ is transformed to multiplication (by
$\frA(\infty)^{-1}$) on $\Gamma^*$.
\end{proof}

For $x\in\fre$, define $u^+$ by picking $\z(x)\in\ol\calF$ with
$\Ima \z\geq 0$, so that $\x(\z(x))=x$, and letting
\begin{equation} \lb{9.9}
u_n^+(x;\vy) = u_n(\z(x);\vy)
\end{equation}
We also let
\begin{equation} \lb{9.10}
u_n^- (x;\vy) = \ol{u_n^+ (x;\vy)} = u_n(\ol{\z(x)};\vy)
\end{equation}

\begin{theorem}\lb{T9.3} Define the Wronskian, $\Wr(f,g)$, of two solutions of \eqref{8.1} by
\begin{equation} \lb{9.11}
\Wr(f,g) =a_n (f_{n+1} g_n - f_n g_{n+1})
\end{equation}
{\rm{(}}which is $n$-independent{\rm{)}}. Then, if $x=\x(\z(x))$ and
\begin{equation} \lb{9.12}
\Ima \z>0 \qquad x\in\fre^\intt
\end{equation}
we have
\begin{equation} \lb{9.13}
\Wr(u_\bddot^+ (x;\vy), u_\bddot^- (x;\vy)) = 2\pi i
\abs{u(\z(x);\vy)}^2 w_{\vec y}(x)
\end{equation}
where $w_{\vec y}$ is the weight in the spectral measure,
$d\mu_{\vec y}$.
\end{theorem}

\begin{proof} Since $W_0=a_0^{-1}$,
\begin{align}
\Wr (\bar W,W) &= a_0 (\bar W_1 a_0^{-1} -W_1 a_0^{-1}) \notag \\
&= -2i \Ima W_1 \lb{9.14}
\end{align}
Taking into account that
\begin{equation} \lb{9.15}
\f 1\pi\Ima m(x_0 + i0) = w(x_0)
\end{equation}
and
\begin{equation} \lb{9.16}
\Wr(\bar c\bar f,cf) = \abs{c}^2 \Wr(\bar f,f)
\end{equation}
we get \eqref{9.13}.
\end{proof}

Recall (\cite[Ch.~3]{Rice}) that the transfer matrix, $T_n(z)$, for
a Jacobi matrix updates solutions of \eqref{8.1} via
\begin{equation} \lb{9.17}
\begin{pmatrix} u_{n+1} \\ a_n u_n \end{pmatrix} = T_n(z)
\begin{pmatrix} u_1 \\ a_0 u_0 \end{pmatrix}
\end{equation}
and is given by
\begin{equation} \lb{9.18}
T_n(z) = \begin{pmatrix}
p_n(z) & -q_n(z) \\
a_np_{n-1}(z) & - a_nq_{n-1}(z)
\end{pmatrix}
\end{equation}
where $q_n$ are the second kind polynomials. As usual, if we want to
indicate the underlying point in the isospectral torus, we write
$T_n(z;\vy)$.

\begin{theorem}\lb{T9.4} There is a constant $C$ so that uniformly for $\vy\in\bbG$ and $x\in\fre^\intt$,
\begin{equation} \lb{9.19}
\norm{T_n(x;\vy)} \leq C\, \dist(x,\bbR\setminus\fre)^{-1/2}
\end{equation}
\end{theorem}

\begin{remark} This result is used in Proposition~7.2 of \cite{BLS}.
\end{remark}

\begin{proof} Let $U_n(x;\vy)$ be the matrix
\begin{equation} \lb{9.20}
U_n(x;\vy) = \begin{pmatrix}
u_{n+1}^+ (x;\vy) & u_{n+1}^- (x;\vy) \\
a_n u_n^+ (x;\vy) & a_n u_n^- (x;\vy)
\end{pmatrix}
\end{equation}
Then
\begin{equation} \lb{9.21}
T_n(x;\vy) U_0(x;\vy) = U_n(x;\vy)
\end{equation}
so
\begin{equation} \lb{9.22}
T_n(x;\vy) = U_n(x;\vy) U_0(x;\vy)^{-1}
\end{equation}
and, since $2\times 2$ matrices obey $\norm{C^{-1}}=\abs{\det(C)}^{-1} \norm{C}$,
\begin{equation} \lb{9.23}
\norm{T_n(x;\vy)} \leq \abs{\det(U_0(x;\vy))}^{-1}
\norm{U_0(x;\vy)}\, \norm{U_n(x;\vy)}
\end{equation}

As $u_n$ is uniformly bounded in $n$, $\vy$, and $x\in\fre$, and
$\det(U_0)$ is the Wronskian, by \eqref{9.13},
\[
\norm{T_n(x;\vy)} \leq C w_{\vec y}(x)^{-1}
\]
which yields \eqref{9.19}, given \eqref{6.5a}.
\end{proof}

\begin{corollary} \lb{C9.5} Uniformly in $n$, $x\in\fre^\intt$, and $\vy\in\bbG$,
\[
C_1\, \dist (x,\bbR\setminus\fre) \leq \abs{p_n(x)}^2 +
\abs{p_{n-1}(x)}^2 \leq C_2\, \dist (x,\bbR\setminus\fre)^{-1}
\]
for suitable constants $C_1$ and $C_2$.
\end{corollary}

\begin{proof} Immediate from \eqref{9.18}, $\det(T_n)=1$ (so $\norm{T_n^{-1}} =\norm{T_n}$), and
\eqref{9.19}, recalling that the $a_n$'s are bounded and bounded
away from $0$.
\end{proof}

Next, we look at band edges where $\norm{T_n}$ can diverge. We begin with a critical fact about the outer
edges:

\begin{theorem} \lb{T9.6} There are positive, finite constants $C_1$ and $C_2$, so that uniformly in $n$ and $\vy$,
\begin{align}
C_1 &\leq u_n (\beta_{\ell+1};\vy) \leq C_2 \lb{9.24} \\
C_1 &\leq (-1)^n u_n (\alpha_1;\vy) \leq C_2 \lb{9.25}
\end{align}
\end{theorem}

\begin{remark} \eqref{9.24} says, in the language of \cite{FSW}, that each whole-line $\ti J_{\vec y}$ has a regular
ground state (see \cite[Example~1.5]{FSW}). It implies critical
Lieb--Thirring bounds for perturbations of $\ti J_{\vec y}$ in
$(\beta_{\ell+1},\infty)$. \eqref{9.25} implies similar bounds for
$(-\infty, \alpha_1)$.
\end{remark}

\begin{proof} $B(z)/z$ is positive at $z=0$ and real and nonvanishing on $(-1,1)$, so $B(x)>0$ on $(0,1]$
and $B(x)<0$ on $[-1,0)$. Since $\abs{B(\pm 1)} =1$, we conclude
that
\begin{equation} \lb{9.26}
B(\pm 1) =\pm 1
\end{equation}
Since $u(x;\vy)$ is bounded, strictly positive and continuous in $x$
and $\vy$ for $x\in [-1,1]$ and $\vy\in\bbG$, \eqref{9.5} implies
\eqref{9.24}--\eqref{9.25}.
\end{proof}

$u_n^+(x)$ is real at $x\in \{\alpha_j,\beta_j\}_{j=1}^{\ell+1}$, so
$u_n^-=u_n^+$ and the Jost solutions are not linearly independent.
The following gives a second solution which grows linearly in $n$.

\begin{theorem}\lb{T9.7} Uniformly in $\vy\in\bbG$ and $z\in\partial\calF\cap\bbD$,
\begin{SL}
\item[{\rm{(i)}}]
\begin{equation} \lb{9.27}
\biggl| \f{\partial u_n^+ (\x(z);\vy)}{\partial z}\biggr| \leq
C(\abs{n}+1)
\end{equation}
\item[{\rm{(ii)}}]
\begin{equation} \lb{9.28}
\liminf_{n\to\infty}\, \biggl| \f{1}{n} \, \f{\partial
u_n^+(\x(z);\vy)}{\partial z}\biggr| >0
\end{equation}
\item[{\rm{(iii)}}] At $\x(z)\in \{\alpha_j,\beta_j\}_{j=1}^{\ell+1}$,
\begin{equation} \lb{9.29}
v_n(\vy) = \f{\partial u_n^+}{\partial z}\, (\x(z);\vy)
\end{equation}
is a solution of \eqref{8.1}, linearly independent of
$u_n^+(\x(z);\vy)$.
\end{SL}
\end{theorem}

\begin{proof} (i), (ii) \ By \eqref{9.5},
\begin{equation} \lb{9.30}
\begin{split}
\f{\partial u_n^+(\x(z);\vy)}{\partial z} &= a_n^{-1} nB^{n-1}(z) B'(z) u(z;U^n(\vy))\\
&\qquad + a_n^{-1} B^n(z) \, \f{\partial}{\partial z}\,
u(z;U^n(\vy))
\end{split}
\end{equation}
Since $u(z;U^n(\vy))$ and $\f{\partial}{\partial z}(u(z;U^n(\vy))$
are uniformly bounded in $\vy$ and $n$, and $B'(e^{i\theta}) >0$ for
all $\theta$, \eqref{9.27}--\eqref{9.28} are immediate.

\smallskip
(iii) \ $u_n^+(\x(z))$ obeys \eqref{8.1} with $z$ replaced by
$\x(z)$. Since $\x'(z)=0$ at points with $\x(z)\in
\{\alpha_j,\beta_j\}_{j=1}^{\ell+1}$, we see that $v_n$ also solves
\eqref{8.1}. Since $u_n^+$ is bounded and $v_n$ is not, they are
linearly independent.
\end{proof}

\begin{corollary}\lb{C9.8} For $z\in\fre$, \eqref{8.1} has no solution which belongs to $\ell^2$ at $+\infty$ or at
$-\infty$.
\end{corollary}

\begin{remark} This result is used in \cite{BLS}.
\end{remark}

\begin{proof} If $z\in\fre^\intt$, this follows from the fact that $\norm{T_n(z)^{-1}}$ is bounded, and for $z\in\partial\fre$,
it follows from Theorem~\ref{T9.7}.
\end{proof}

\begin{corollary} \lb{C9.9} If $x\in\{\alpha_j,\beta_j\}_{j=1}^{\ell+1}$, then
\begin{equation} \lb{9.31}
\norm{T_n(x)} \leq C(1+\abs{n})
\end{equation}
\end{corollary}

\begin{proof} Let
\begin{equation} \lb{9.32}
\ti U_n(x) = \begin{pmatrix}
u_{n+1}^+(x) & v_{n+1}(x) \\
a_n u_n^+(x) & a_n v_n(x)
\end{pmatrix}
\end{equation}
As in \eqref{9.22},
\begin{equation} \lb{9.33}
T_n(x) = \ti U_n(x) U_0(x)^{-1}
\end{equation}
Since $u^+,v$ are independent, $U_0$ is invertible and, clearly, $\norm{\ti U_n}\leq C(1+\abs{n})$.
\end{proof}

The bound \eqref{9.19} diverges as $x$ approaches a point in
$\{\alpha_j,\beta_j\}_{j=1}^{\ell+1}$. Since $\norm{T_n}$ is not
bounded at these points, it must. However, we are heading towards a
uniform (on $\fre$) $O(n)$ bound. As a starting point, we need to
know more about the right side of \eqref{9.13} than the crude bound
from \eqref{6.5a}.

\begin{proposition}\lb{P9.10} For any $\vy\in\bbG$ and fixed $x_0\in\{\alpha_j,\beta_j\}_{j=1}^{\ell+1}$,
\begin{equation} \lb{9.34}
\lim_{\substack{ x\to x_0 \\ x\in\fre^\intt}}\, \abs{x-x_0}^{-1/2}
\abs{u(\z(x);\vy)}^2 w_{\vec y}(x)
\end{equation}
exists, is finite and nonvanishing, and is continuous in $\vy$. In
\eqref{9.34}, $\z(x)$ is the point obeying \eqref{9.12} with
$\z\in\ol\calF\cap\bbC_+$ and $x=\x(\z(x))$.
\end{proposition}

\begin{remark} This result is subtle because $\abs{u(z;\vy)}^2$ can vanish at $x_0$. In that case, $w_{\vec y}(x)$
has $O(\sqrt{\abs{x-x_0}}^{-1})$ asymptotics rather than $O(\sqrt{x-x_0})$ asymptotics.
\end{remark}

\begin{proof} By the explicit formula for $w_{\vec y}$ (\eqref{6.9} and \eqref{6.16d}), each factor in $w_{\vec y}(x)
\abs{x-x_0}^{-1/2}$ is continuous in $x$ and $\vy$ except for the
$\abs{x-y_j}^{-1}$ factor with the $y_j$ closest to $x_0$. There is
a cancelling factor in $\abs{u(z;\vy)}^2$, so the limit exists and
is continuous in $\vy$.
\end{proof}

\begin{theorem} \lb{T9.11} There is a constant $C$ so that uniformly in $x\in\fre$ and $\vy\in\bbG$,
\begin{equation} \lb{9.40}
\norm{T_n(x;\vy)} \leq C(\abs{n}+1)
\end{equation}
\end{theorem}

\begin{proof} For each $x_0\in\{\alpha_j,\beta_j\}_{j=1}^{\ell+1}$, we prove \eqref{9.40} in the half-band starting
at $x_0$. Form a matrix $\ti U_n$ like \eqref{9.20} but with $u^-$
replaced by $\ti v = (u^--u^+)/\abs{x-x_0}^{1/2}$. As we have seen,
$\ti v$ has a limit as $x$ approaches $x_0$ from $\fre^\intt$. By
the Wronskian calculation, $\det(U_n)$ (which is $n$-independent) is
bounded as $x$ approaches $x_0$.

Finally, writing $(u^--u^+)$ as the integral of a derivative and using \eqref{9.27}, we get
\begin{equation} \lb{9.41}
\abs{\ti U_n} \leq C(\abs{n}+1)
\end{equation}
so
\begin{equation} \lb{9.42}
\norm{\ti U_n} \leq C(\abs{n}+1)
\end{equation}
which implies \eqref{9.40}.
\end{proof}

At last, we want to note that \eqref{9.5} implies a result about the
Jost solutions used in \cite{2ext}.

\begin{theorem}\lb{T9.12} For any compact interval $I\subset\fre^\intt$, we can write
\begin{equation} \lb{9.43}
u_n^+(x) = e^{in\theta(x)} f_n(x)
\end{equation}
where
$\theta'(x) = \pi\rho_\fre (x)$
and
$f_n$ is real analytic in $x$ with derivatives uniformly bounded in $n$.
\end{theorem}

\begin{proof} Let $\z(x)\in\ol\calF\cap\bbC_+$ with $\x(\z(x))=x$ and write
\[
B(\z(x)) =e^{i\theta(x)}
\]
By \eqref{9.5}, $u_n^+$ has the form \eqref{9.43} and we see that
$f_n$ has the required properties. By the calculation that led to
\eqref{4.50}, we get the expression for $\theta'$.
\end{proof}

In \cite{2ext}, this was used to prove that for any $d\mu_{\vec y}$,
with $\vy\in\bbG$,
\begin{equation} \lb{9.55}
\f{1}{n}\, K_n(x,x) \to \f{\rho_\fre(x)}{w_{\vec y}(x)}
\end{equation}
uniformly on $I$, where $K_n$ is the CD kernel (see \cite{CD} for
definition and background on the classical work of
M\'at\'e--Nevai--Totik and Totik on limits like \eqref{9.55}). Using
the calculations in \cite{ALS} and identifying $u_n^+$ as a multiple
of the Deift--Simon eigenfunctions of \cite{ALS}, one sees that
\eqref{9.55} implies that

\begin{proposition}\lb{P9.13} Uniformly on compact subsets of $\fre^\intt \times\bbG$,
\begin{equation} \lb{9.56}
\f{1}{n} \sum_{k=0}^{n-1} \Real [u_k^+ (x;\vy)^2] \to 0 \,\mbox{ as
} \, n\to\infty
\end{equation}
\end{proposition}

\begin{remark} That is, $(\Real u_n^+)^2$ and $(\Ima u_n^+)^2$ have the same average.
\end{remark}

\section{Bounds on (Spectral Theorist's) Green's Function} \lb{s10}

For $\vy\in\bbG$, let $J_{\vec y}$ be the associated Jacobi matrix
and $\ti J_{\vec y}$ the associated full-line Jacobi matrix. The
(spectral theorist's) Green's functions are defined by
\begin{alignat}{2}
G_{nm}(z) &= \jap{\delta_n, (J_{\vec y}-z)^{-1}\delta_m} && \qquad n,m=1,2,\dots \lb{10.1} \\
\ti G_{nm}(z) &= \jap{\delta_n, (\ti J_{\vec y}-z)^{-1}\delta_m}
&&\qquad n,m\in\bbZ \lb{10.2}
\end{alignat}
We will use $G_{nm}(z;\vy)$ when we need $\vy$ to be explicit. It is
unfortunate that ``Green's function'' is used both for these objects
and for $G_\fre(z)$, the potential theorist's Green's function, but
both names are ubiquitous. $G_\fre$ will appear below in our
discussions of $G_{nm}$.

The analogs of $G_{nm}$ and $\ti G_{nm}$ for $-d^2/dx^2$ on $L^2(0,\infty)$ or $L^2 (-\infty,\infty)$ are
given by
\begin{align}
\ti G(x,y;E) &= \f{e^{-\kappa\abs{x-y}}}{2\kappa} \lb{10.3} \\
G(x,y;E) &= \f{e^{-\kappa x_>} \sinh(\kappa x_<)}{\kappa} \lb{10.4}
\end{align}
where $E= -\kappa^2$, $x_>=\max (x,y)$, $x_< =\min(x,y)$. The bounds
\begin{align}
\abs{\ti G(x,y;E)} &\leq (2\kappa)^{-1} \lb{10.5} \\
\abs{G(x,y;E)} &\leq x_< \lb{10.6}
\end{align}
play important roles in the analysis of bound states of
Schr\"odinger operators with short-range potentials. Here we find
analogs of these bounds for Jacobi matrices in the isospectral torus
for $z$ in $\bbR\setminus\fre$. These bounds were used in
\cite{HS2008} to obtain bounds on perturbations of $J_{\vec y}$ and
$\ti J_{\vec y}$.

We need to begin by defining $u_n^\pm(x;\vy)$ for
$x\in\bbR\setminus\fre$. Define $\z(x)$ to be the unique point in $
[\cup_{j=1}^\ell C_j^+]\cup(-1,1)$ with
\begin{equation} \lb{10.7}
\x(\z(x))=x
\end{equation}
Then we define
\begin{align}
u_n^+(x;\vy) &= a_n^{-1} B(\z(x))^n u(\z(x);U^n(\vy)) \lb{10.8} \\
u_n^- (x;\vy) &= a_n^{-1} B(\z(x))^{-n}\, \ol{u(1/\ol{\z(x)};
U^n(\vy))} \lb{10.9}
\end{align}
$u_n^+$ is the analytic continuation of $u_n^+$ as defined for
$x\in\fre$ in the last section if we keep $x\in\ol\bbC_+$. So is
$u_n^-$ since $u_n^-$ is defined on the lower lip of the cuts, so
continuing in $\ol\bbC_+$ brings us to the second sheet and
$1/\ol{\z(x)}$. The $-n$ in $u_n^-$ comes from
\begin{equation} \lb{10.10}
\ol{B(1/\ol{\z(x)})} =B(\z(x))^{-1}
\end{equation}

Since
\begin{equation} \lb{10.11}
\abs{B(\z(x))} = e^{-G_\fre(x)}
\end{equation}
$u_n^\pm$ decays exponentially as $n\to\pm\infty$ and grows
exponentially as $n\to \mp\infty$. It follows that $u_n^\pm$ must
have constant phase (and perhaps we should redefine them to be
real). Indeed, the phase is constant on each gap and $u^+$ and $u^-$
have opposite phases.

\begin{theorem} \lb{T10.1} For $x\in\bbR\setminus\fre$ and $n\leq
m$, we have
\begin{equation} \lb{10.12}
\ti G_{nm}(x) = \f{u_n^-(x) u_m^+(x)}{\Wr(x)}
\end{equation}
where
\begin{equation} \lb{10.13}
\Wr(x) = a_n (u_{n+1}^+(x) u_n^-(x) - u_{n+1}^-(x) u_n^+(x))
\end{equation}
Uniformly in $x\in\bbR\setminus\fre$ and $\vy\in\bbG$, we have
\begin{equation} \lb{10.14}
\abs{\ti G_{nm}(x)} \leq Ce^{-G_\fre(x)\abs{n-m}} \dist(x,\fre)^{-1/2}
\end{equation}
\end{theorem}

\begin{proof} \eqref{10.12}--\eqref{10.13} is a standard formula for $\ti G$ in terms of any solutions decaying
at $\pm\infty$. As $x$ runs through $\bbR\setminus\fre$, $\z(x)$
runs through $[\cup_{j=1}^\ell C_j^+]\cup(-1,1)\setminus\{0\}$, and
$u(\z(x);\vy)$ is uniformly bounded there. So, by \eqref{10.11}, we
get \eqref{10.14} from
\begin{equation} \lb{10.15}
\abs{\Wr(x)} \geq C\dist (x,\fre)^{1/2}
\end{equation}

This is trivial, except near the points $\abs{x}=\infty$ and
$x\in\{\alpha_j,\beta_j\}_{j=1}^{\ell+1}$ since $\Wr(x)$ is
nonvanishing and continuous away from those points.

Since $u$ is regular at $z=0$ and $z=\infty$, the dominant term in
\eqref{10.8}--\eqref{10.9} is the $B(z)$ term. In
\[
\Wr(x)=a_0 (u_1^+ u_0^- -u_1^- u_0^+)
\]
the dominant term is $B(z)^{-1}$ in $u_1^-$, so
\begin{equation} \lb{10.16}
\abs{\Wr(x)} \sim C\abs{\z(x)}^{-1} \sim C\abs{x} \geq
C\abs{x}^{1/2}
\end{equation}
and \eqref{10.15} holds near $|x|=\infty$.

Near points $x_0\in\{\alpha_j,\beta_j\}_{j=1}^{\ell+1}$, we are
looking at the Wronskian of two solutions $u^+$ and $u^-$  which
approach each other. Thus, we get a Wronskian which goes to zero as
$(\z(x)- \z(x_0))$ times the Wronskian of  $u^+$ and $du^-/dz$. We
have already seen in the last section that these are two linearly
independent solutions, so their Wronskian is nonzero and thus, near
$x_0\in\{\alpha_j,\beta_j\}_{j=1}^{\ell+1}$, for some $C>0$,
\begin{equation}
\Wr(x) \geq C\abs{\z(x)-\z(x_0)} = C\cdot O(\abs{x-x_0}^{1/2})
\end{equation}
proving \eqref{10.15}.
\end{proof}

As $x$ approaches a point
$x_0\in\{\alpha_j,\beta_j\}_{j=1}^{\ell+1}$ from
$\bbR\setminus\fre$, typically (i.e., except for special values of
$\vy$, $n$, and $m$), $\ti G_{nm}(x)\to\infty$. For $G_{nm}(x)$,
this is normally  not true, which is why one expects bounds in this
case that are not divergent at $x_0$. However, there are special
values of $\vy$ for which this is not the case. One sees this for
$n=m=1$ since
\begin{equation} \lb{10.17}
G_{11}(z)=m(z)
\end{equation}
the $m$-function of \eqref{6.3}. $m$ is meromorphic on $\calS$, so it normally has a finite value at $x_0$ but
might have a pole there.

\begin{definition} Fix $\vy\in\bbG$. A point $x_0\in\{\alpha_j,\beta_j\}_{j=1}^{\ell+1}$ is said to be a
\emph{resonance} if and only if it is a pole of $m_{\vec y}(z)$ (in
the sense of poles on $\calS$ which means $(x-x_0)^{-1/2}$
divergence since $x_0$ is a branch point). Otherwise, we say $x_0$
is \emph{nonresonant}.
\end{definition}

It is easy to see that resonances are equivalent to $u_0^+(x_0)=0$.

\begin{theorem}\lb{T10.2} Fix $\vy\in\bbG$. For $n,m\geq 1$ and
$x\in\bbR\setminus\fre$, we have
\begin{equation} \lb{10.17a}
G_{nm}(x)=\ti G_{nm}(x)-\ti G_{0n}(x)\ti G_{0m}(x)\ti G_{00}(x)^{-1}
\end{equation}
Suppose $x_0$ is nonresonant for $\vy$. Let $I$ be the open interval
with one end at $x_0$ and the other at the middle of the gap if
$x_0$ is a boundary point of a finite gap and
$I\subset\bbR\setminus\fre$ with $\abs{I}=1$ if
$x_0\in\{\alpha_1,\beta_{\ell+1}\}$. Then for all $n,m\geq 1$ and
$x\in I$, we have
\begin{align}
\abs{G_{nm}(x)} &\leq C \min(n,m) \lb{10.17b} \\
\abs{G_{nm}(x)} &\leq C\abs{x-x_0}^{-1/2} \lb{10.18}
\end{align}
for some constant $C$.
\end{theorem}

\begin{proof} Since, for $x\in\bbR\setminus\fre$ fixed,
\begin{equation} \lb{10.19}
u_0^+ u_n^- - u_0^- u_n^+ \equiv q_n
\end{equation}
vanishes at $n=0$, $q_n(x)=C(x) p_n(x)$ for some constant (depending
on $x$). Thus, by the standard formula for $G_{nm}$,
\begin{equation} \lb{10.20}
G_{nm} = \f{p_n\, u_m^+}{\Wr( u^+,p)}, \quad 1\leq n\leq m
\end{equation}
we get
\begin{equation} \lb{10.21}
G_{nm}(x) = q_n(x) u_m^+(x) \wti{\Wr}(x)^{-1}
\end{equation}
where $\wti{\Wr}$ is the Wronskian of $u^+$ and $q$.

By \eqref{10.19},
\begin{equation} \lb{10.22}
\wti{\Wr}(x) = u_0^+(x) \Wr(x)
\end{equation}
so \eqref{10.21} becomes
\begin{equation} \lb{10.23}
G_{nm}(x) = \f{u_0^+(x) u_n^-(x) u_m^+(x) - u_0^-(x) u_n^+(x)
u_m^+(x)}{u_0^+(x)\Wr(x)}
\end{equation}
The first term in \eqref{10.23} is, by \eqref{10.12}, $\ti G_{nm}(x)$. If we note that (also by \eqref{10.12})
\begin{equation} \lb{10.24}
\ti G_{0n}(x) \ti G_{0m}(x) \ti G_{00}(x)^{-1} = \f{u_0^-(x)u_n^+(x)
u_m^+(x)}{u_0^+(x) \Wr(x)}
\end{equation}
we see that the second term in \eqref{10.23} is the second term in
\eqref{10.17a}, so we have proven \eqref{10.17a}.

As we noted above, $x_0$ nonresonant implies that $u_0^+(x_0)\neq
0$. Thus, \eqref{10.23} shows that
\begin{equation} \lb{10.24x}
\sup_{\substack{n,m\geq 1 \\ x\in I}}\, \abs{G_{nm}(x)} \leq C\, \Wr(x)^{-1}
\end{equation}
which, by \eqref{10.16}, proves \eqref{10.18}.

We claim first that \eqref{10.17b} is implied by
\begin{equation} \lb{10.25}
\abs{q_n(x) B(\z(x))^n} \leq Cn\abs{x-x_0}^{1/2}
\end{equation}
For, by \eqref{10.23} and \eqref{10.16},
\begin{equation} \lb{10.26}
\abs{G_{nm}(x)} \leq C \abs{u_m^+(x)B(\z(x))^{-m}}\, \abs{B(\z(x))}^{m-n}
\abs{q_n(x) B(\z(x))^n}\, \abs{x-x_0}^{-1/2}
\end{equation}
so \eqref{10.25} together with $\abs{B(\z(x))} \leq 1$ and the fact
that $\abs{u_m^+(x) B(\z(x))^{-m}}$ is bounded implies
\eqref{10.17b}.

Next, note that because of the definitions
\eqref{10.8}--\eqref{10.9} and the constancy of the phase of
$u_n^\pm$, we see that for all $n$,
\begin{equation} \lb{10.26x}
q_n(x_0)=0
\end{equation}
Define
\begin{equation} \lb{10.27}
h_n(z) = q_n(\x(z)) B(z)^n
\end{equation}
By \eqref{10.26x} and $\abs{\z(x)-\z(x_0)} =O(\abs{x-x_0}^{1/2})$,
\eqref{10.25} follows from
\begin{equation} \lb{10.28}
\sup_{z\in\z(I)}\, \biggl| \f{dh_n(z)}{dz}\biggr| \leq Cn
\end{equation}
$h_n$ is built out of $u$'s which have bounded derivatives and
$B(z)^{n}$ which has a derivative bounded by $Cn$, so \eqref{10.28}
holds.
\end{proof}

\bigskip


\end{document}